\documentclass[american,english]{amsart}
\usepackage[T1]{fontenc}
\usepackage[latin9]{inputenc}
\usepackage{babel}
\usepackage{amstext}
\usepackage{amsthm}
\usepackage{amssymb}
\usepackage{graphicx}
\usepackage[unicode=true,
 bookmarks=true,bookmarksnumbered=true,bookmarksopen=false,
 breaklinks=false,pdfborder={0 0 1},backref=false,colorlinks=false]
 {hyperref}
\hypersetup{pdftitle={Exotic Cluster Structures on SLn with BD data of Minimal Size},
 pdfauthor={Idan Eisner}}
\usepackage{breakurl}

\makeatletter
\numberwithin{equation}{section}
\numberwithin{figure}{section}
\theoremstyle{plain}
\newtheorem{thm}{\protect\theoremname}[section]
  \theoremstyle{plain}
  \newtheorem{prop}[thm]{\protect\propositionname}
  \theoremstyle{plain}
  \newtheorem{conjecture}[thm]{\protect\conjecturename}
  \theoremstyle{remark}
  \newtheorem{rem}[thm]{\protect\remarkname}
  \theoremstyle{plain}
  \newtheorem{lem}[thm]{\protect\lemmaname}
  \theoremstyle{plain}
  \newtheorem{cor}[thm]{\protect\corollaryname}

\usepackage{amsfonts}
\usepackage{bbm}

\DeclareMathOperator{\diag}{diag}
\DeclareMathOperator{\Tr}{Tr}
\DeclareMathOperator{\End}{End}

\DeclareMathOperator{\rank}{rank}

\providecommand{\abs}[1]{\lvert#1\rvert}
\providecommand{\Id}{\mathbbm{1}}

\makeatother

  \addto\captionsamerican{\renewcommand{\conjecturename}{Conjecture}}
  \addto\captionsamerican{\renewcommand{\corollaryname}{Corollary}}
  \addto\captionsamerican{\renewcommand{\lemmaname}{Lemma}}
  \addto\captionsamerican{\renewcommand{\propositionname}{Proposition}}
  \addto\captionsamerican{\renewcommand{\remarkname}{Remark}}
  \addto\captionsamerican{\renewcommand{\theoremname}{Theorem}}
  \addto\captionsenglish{\renewcommand{\conjecturename}{Conjecture}}
  \addto\captionsenglish{\renewcommand{\corollaryname}{Corollary}}
  \addto\captionsenglish{\renewcommand{\lemmaname}{Lemma}}
  \addto\captionsenglish{\renewcommand{\propositionname}{Proposition}}
  \addto\captionsenglish{\renewcommand{\remarkname}{Remark}}
  \addto\captionsenglish{\renewcommand{\theoremname}{Theorem}}
  \providecommand{\conjecturename}{Conjecture}
  \providecommand{\corollaryname}{Corollary}
  \providecommand{\lemmaname}{Lemma}
  \providecommand{\propositionname}{Proposition}
  \providecommand{\remarkname}{Remark}
\providecommand{\theoremname}{Theorem}

\begin{document}

\subjclass[2000]{53D17,13F60}

\keywords{Poisson--Lie group, cluster algebra, Belavin--Drinfeld triple}

\title[Cluster structures with minimal BD data, I]{Exotic Cluster Structures on $SL_{n}$ with Belavin--Drinfeld Data
of Minimal Size, I. The Structure}

\date{February, 2016}

\author{Idan Eisner}
\address{Department of Mathematics, 
         University of Haifa,\footnote{Current affiliation: Department of Mathematics, Technion - Israel Institute of Technology, Haifa, 32000, Israel}\\
         Haifa, 3498838, Israel}

\email{idaneisner@technion.ac.il}
\begin{abstract}
Abstract. Using the notion of compatibility between Poisson brackets and
cluster structures in the coordinate rings of simple Lie groups, 
Gekhtman Shapiro and Vainshtein conjectured the existence 
of a cluster structure for each Belavin--Drinfeld solution of the 
classical Yang-Baxter equation compatible with the corresponding 
Poisson-Lie bracket on the simple Lie group.
Poisson Lie groups are classified by the Belavin--Drinfeld classification
of solutions to the classical Yang Baxter equation. For any non trivial
Belavin--Drinfeld data of minimal size for $SL_{n}$, we give an algorithm
for constructing an initial seed $\Sigma$ in $\mathcal{O}\left(SL_{n}\right)$.
The cluster structure $\mathcal{C}=\mathcal{C}\left(\Sigma\right)$
is then proved to be compatible with the Poisson bracket associated
with that Belavin--Drinfeld data, and the seed $\Sigma$ is  locally regular.

This is the first of two papers, and the second one proves the rest
of the conjecture: the upper cluster algebra $\overline{\mathcal{A}}_{\mathbb{C}}(\mathcal{C})$
is naturally isomorphic to $\mathcal{O}\left(SL_{n}\right)$, and
the correspondence between Belavin--Drinfeld classes and cluster structures
is one to one. 
\end{abstract}

\maketitle
\maketitle

\section{Introduction}
Upon the introduction of cluster algebras in \cite{FZ1}, 
a natural question has arisen: do (multiple) cluster structures 
in the coordinate rings of a given algebraic variety $V$ exist?
Partial answers were given for Grassmannians
$V=Gr_{k}\left(n\right)$ \cite{Scott} and double Bruhat cells \cite{BFZ}.
If $V=\mathcal{G}$ is a simple Lie group, one can extend the cluster
structure found in the double Bruhat cell to one in $\mathcal{O}\left(\mathcal{G}\right)$.
The compatibility of cluster structures and Poisson brackets, as characterized
in \cite{GSV1} suggested a connection between the two: given a Poisson
bracket, does a compatible cluster structure exist? Is there a way
to find it?

In the case that $V=\mathcal{G}$ is a simple complex Lie group, R-matrix
Poisson brackets on $\mathcal{G}$ are classified by the Belavin--Drinfeld
classification of solutions to the classical Yang Baxter equation
\cite{BDsolCYBE}. Given a solution of that kind, a Poisson bracket
can be defined on $\mathcal{G}$, making it a Poisson -- Lie group.

The Belavin--Drinfeld (BD) classification is based on pairs of isometric
subsets of simple roots of the Lie algebra $\mathfrak{g}$ of $\mathcal{G}$.
The trivial case when the subsets are empty corresponds to the standard
Poisson bracket on $\mathcal{G}$ . It has been shown in \cite{gekhtman2012cluster}
that extending the cluster structure introduced in \cite{BFZ} from
the double Bruhat cell to the whole Lie group $V$ yields a cluster
structure that is compatible with the standard Poisson bracket. This
led to naming this cluster structure ``standard'', and trying to
find other cluster structures, compatible with brackets associated
with non trivial BD subsets. The term ``exotic'' was suggested for
these non standard structures \cite{gekhtman2013arxiv}.

Gekhtman, Shapiro and Vainshtein conjectured the existence of a corresponding
cluster structure for every BD class for a given simple Lie group
\cite{gekhtman2012cluster,gekhtman2013exotic}. According to the conjecture, 
for a given BD class for $\mathcal{G}$,
there exists a cluster structure on $\mathcal{G}$, with rank determined
by the BD data. This cluster structure is compatible with the associated
Poisson bracket. The conjecture also states that the structure is
regular, and that the upper cluster algebra coincides with the ring
of regular functions on $\mathcal{G}$. The conjecture was proved
for the standard case and for $\mathcal{G}=SL_{n}$ with $n<5$ in
\cite{gekhtman2012cluster}. The Cremmer -- Gervais case, which in
some sense is the ``furthest'' from the standard one, was proved
in \cite{gekhtman2013exotic}. It was also found to be true for all
possible BD classes for $SL_{5}$ \cite{eisner2014SL5}.

This paper proves parts of the conjecture for $SL_{n}$ when the BD
data is of minimal size, i.e., the two subsets contain only one simple
root. Starting with two such subsets $\left\{ \alpha\right\} $ and
$\left\{ \beta\right\} $, Section \ref{sub:Constructing-a-log} describes
an algorithm for construction of a set $\mathcal{B}_{\alpha\beta}$
of functions that will serve as the initial cluster. It is then proved
that this set is \emph{log canonical }with respect to the associated
Poisson bracket $\left\{ \cdot,\cdot\right\} _{\alpha\beta}$. Adding
a quiver $Q_{\alpha\beta}$ (or an exchange matrix $\tilde{B}_{\alpha\beta}$)
defines a cluster structure on $SL_{n}$. It is shown in Section \ref{sec:The-cluster-structure}
that this structure is indeed compatible with the Poisson bracket.
Then Section \ref{sec:Regularity} proves that the initial seed is
locally regular.

This proves that for minimal size BD data for $SL_{n}$ there exists
a cluster structure which is compatible with the associated Poisson
bracket. The companion paper \cite{eisner2014part2} will complete
the proof of the conjecture: the bijection between cluster structures
and BD classes of this type, the fact that the upper cluster algebra
is naturally isomorphic to the ring of regular functions on $SL_{n}$,
and the description of a global toric action.

\section{Background and main results}

\subsection{Cluster structures}

Let $\{z_{1},\ldots,z_{m}\}$ be a set of independent variables, and
let $S$ denote the ring of Laurent polynomials generated by $z_{1},\ldots,z_{m}$
- 
\[
S=\mathbb{Z}\left[z_{1}^{\pm1},\ldots,z_{m}^{\pm1}\right].
\]
 (Here and in what follows $z^{\pm1}$ stands for $z,z^{-1}$). The
\emph{ambient field} $\mathcal{F}$ is the field of rational functions
in $n$ independent variables (distinct from $z_{1},\ldots,z_{m}$),
with coefficients in the field of fractions of $S$.

A \emph{seed} (of geometric type) is a pair $(\textbf{x},\tilde{B})$,
where $\textbf{x}=(x_{1},\ldots,x_{n})$ is a transcendence basis
of $\mathcal{F}$ over the field of fractions of $S$, and $\tilde{B}$
is an $n\times(n+m)$ integer matrix whose principal part $B$ (that
is, the $n\times n$ matrix formed by columns $1\ldots n$) is skew-symmetric.
The set $\textbf{x}$ is called a \emph{cluster}, and its elements
$(x_{1},\ldots,x_{n})$ are called \emph{cluster variables}. Set $x_{n+i}=z_{i}$
for $i\in[1,m]$ (where we use the notation $[a,b]$ for the set of
integers $\left\{ a,a+1,\ldots,b\right\} .$ Sometimes we write just
$\left[m\right]$ for the set $\left[1,m\right]$). The elements $x_{n+1},\ldots,x_{n+m}$
are called \emph{stable variables} (or \emph{frozen variables}). The
set $\tilde{\textbf{x}}=(x_{1},\ldots,x_{n},x_{n+1},\ldots,x_{n+m})$
is called an \emph{extended cluster}. The square matrix $B$ is called
the \emph{exchange matrix}, and $\tilde{B}$ is called the \emph{extended
exchange matrix}. We sometimes denote the entries of $\tilde{B}$
by $b_{ij}$, or say that $\tilde{B}$ is skew-symmetric when the
matrix $B$ has this property.

Let $\Sigma=(\tilde{\mathbf{x}},\tilde{B})$ be a seed. The adjacent
cluster in direction $k\in\left[n\right]$ is $\tilde{\mathbf{x}}_{k}=(\tilde{\mathbf{x}}\setminus x_{k}\cup\left\{ x'_{k}\right\} $,
where $x'_{k}$ is defined by the \emph{exchange relation 
\begin{equation}
x_{k}\cdot x'_{k}=\prod_{b_{kj}>0}x_{j}^{b_{kj}}+\prod_{b_{kj}<0}x_{j}^{-b_{kj}}\label{eq:exrltn}
\end{equation}
 } \emph{A matrix mutation} $\mu_{k}(\tilde{B})$\emph{
}of $\tilde{B}$ in direction $k$ is defined by 
\[
b'_{ij}=\begin{cases}
-b_{ij} & \text{ if }i=k\text{ or }j=k\\
b_{ij}+\frac{1}{2}\left(\abs{b_{ik}}b_{kj}+b_{ik}\abs{b_{kj}}\right) & \text{ otherwise.}
\end{cases}
\]
 Seed mutation in direction $k$ is then defined $\mu_{k}\left(\Sigma\right)=(\tilde{\mathbf{x}}_{k},\mu_{k}(\tilde{B})).$

Two seeds are said to be mutation equivalent if they can be connected
by a sequence of seed mutations.

Given a seed $\Sigma=(\textbf{\ensuremath{\mathbf{x}}},\tilde{B})$,
the \emph{cluster structure} $\mathcal{C}(\Sigma)$ (sometimes denoted
$\mathcal{C}(\tilde{B})$, if $\mathbf{\mathbf{x}}$ is understood
from the context) is the set of all seeds that are mutation equivalent
to $\Sigma$. The number $n$ of rows in the matrix $\tilde{B}$ is
called the \emph{rank} of $\mathcal{C}$.

Let $\Sigma$ be a seed as above, and $\mathbb{A}=\mathbb{Z}\left[x_{n+1},\ldots,x_{n+m}\right]$.
The \emph{cluster algebra} $\mathcal{A}=\mathcal{A}(\mathcal{C})=\mathcal{A}(\tilde{B})$
associated with the seed $\Sigma$ is the $\mathbb{A}$-subalgebra
of $\mathcal{F}$ generated by all cluster variables in all seeds
in $\mathcal{C}(\tilde{B})$. The \emph{upper cluster algebra }$\mathcal{\overline{A}}=\mathcal{\overline{A}}(\mathcal{C})=\mathcal{\overline{A}}(\tilde{B})$\emph{
}is the intersection of the rings of Laurent polynomials over $\mathbb{A}$
in cluster variables taken over all seeds in $\mathcal{C}(\tilde{B})$.
The famous \emph{Laurent phenomenon} \cite{FZ2} claims the inclusion
$\mathcal{A}(\mathcal{C})\subseteq\mathcal{\overline{A}}(\mathcal{C})$.

It is sometimes convenient to describe a cluster structure $\mathcal{C}(\tilde{B})$
in terms of the \emph{quiver} $Q(\tilde{B}):$ it is a directed graph
with $n+m$ nodes labeled $x_{1},\ldots,x_{n+m}$ (or just $1,\ldots,n+m$),
and an arrow pointing from $x_{i}$ to $x_{j}$ with weight $b_{ij}$
if $b_{ij}>0$.

Let $V$ be a quasi-affine variety over $\mathbb{C}$, $\mathbb{C}\left(V\right)$
be the field of rational functions on $V$, and $\mathcal{O}\left(V\right)$
be the ring of regular functions on $V$. Let $\mathcal{C}$ be a
cluster structure in $\mathcal{F}$ as above, and assume that $\left\{ f_{1},\ldots,f_{n+m}\right\} $
is a transcendence basis of $\mathbb{C}\left(V\right)$. Then the
map $\varphi:x_{i}\to f_{i}$, $1\leq i\leq n+m$, can be extended
to a field isomorphism $\varphi:\mathcal{F}_{\mathbb{C}}\to\mathbb{C}(V)$.
with $\mathcal{F}_{\mathbb{C}}=\mathcal{F}\otimes\mathbb{C}$ obtained
from $\mathcal{F}$ by extension of scalars. The pair$\left(\mathcal{C},\varphi\right)$
is then called a cluster structure in $\mathbb{C}\left(V\right)$
(or just a cluster structure on $V$), and the set $\left\{ f_{1},\ldots,f_{n+m}\right\} $
is called an extended cluster in $\left(\mathcal{C},\varphi\right)$.
Sometimes we omit direct indication of $\varphi$ and just say that
$C$ is a cluster structure on $V$. A cluster structure $\left(\mathcal{C},\varphi\right)$
is called \emph{regular} if $\varphi\left(x\right)$ is a regular
function for any cluster variable $x$, and a seed $\Sigma$ is called
\emph{locally regular }if all the cluster variables in $\Sigma$ and
in all the adjacent seeds are regular functions. The two algebras
defined above have their counterparts in $\mathcal{F}_{\mathbb{C}}$
obtained by extension of scalars; they are denoted $\mathcal{A}_{\mathbb{C}}$
and $\overline{\mathcal{A}}_{\mathbb{C}}$. If, moreover, the field
isomorphism $\varphi$ can be restricted to an isomorphism of $\mathcal{A}_{\mathbb{C}}$
(or $\overline{\mathcal{A}}_{\mathbb{C}}$) and $\mathcal{O}\left(V\right)$,
we say that $\mathcal{A}_{\mathbb{C}}$ (or $\overline{\mathcal{A}}_{\mathbb{C}}$)
is \emph{naturally isomorphic} to $\mathcal{O}\left(V\right)$.

Let $\left\{ \cdot,\cdot\right\} $ be a Poisson bracket on the ambient
field $\mathcal{F}$. Two elements $f_{1},f_{2}\in\mathcal{F}$ are
\emph{log canonical }if there exists a rational number $\omega_{f_{1},f_{2}}$
such that $\left\{ f_{1},f_{2}\right\} =\omega_{f_{1},f_{2}}f_{1}f_{2}$.
A set $F\subseteq\mathcal{F}$ is called a log canonical set if every
pair $f_{1},f_{2}\in F$ is log canonical.

A cluster structure $\mathcal{C}$ in $\mathcal{F}$ is said to be
compatible with the Poisson bracket $\left\{ \cdot,\cdot\right\} $
if every cluster is a log canonical set with respect to $\left\{ \cdot,\cdot\right\} $.
In other words, for every cluster $\mathbf{x}$ and every two cluster
variables $x_{i},x_{j}\in\mathbf{\tilde{x}}$ there exists $\omega_{ij}$
s.t. 
\begin{equation}
\left\{ x_{i},x_{j}\right\} =\omega_{ij}x_{i}x_{j}
\end{equation}

The skew symmetric matrix $\Omega^{\mathbf{\tilde{x}}}=(\omega_{ij})$
is called the \emph{coefficient matrix} of $\left\{ \cdot,\cdot\right\} $
(in the basis $\mathbf{\tilde{x}}$).

If $\mathcal{C}(\tilde{B})$ is a cluster structure of maximal rank
(i.e., $\rank\tilde{B}=n$), one can give a complete characterization
of all Poisson brackets compatible $\mathcal{C}(\tilde{B})$ (see
\cite{GSV1}, and also \cite[Ch. 4]{GSV}). In particular, an immediate
corollary of Theorem 1.4 in \cite{GSV1} is:
\begin{prop}
\label{prop:PoissCompStruc}If $\rank\tilde{B}=n$ then a Poisson
bracket is compatible with $\mathcal{C}(\tilde{B})$ if and only if
its coefficient matrix $\Omega^{\tilde{\mathbf{x}}}$ satisfies $\tilde{B}\Omega^{\tilde{\mathbf{x}}}=\left[D\ 0\right]$,
where $D$ is a diagonal matrix. 
\end{prop}

\subsection{Poisson--Lie groups}

A Lie group $\mathcal{G}$ with a Poisson bracket $\{\cdot,\cdot\}$
is called a \emph{Poisson--Lie group} if the multiplication map $\mu:\mathcal{G}\times\mathcal{G}\to\mathcal{G}$,
$\mu:(x,y)\mapsto xy$ is Poisson. That is, $\mathcal{G}$ with a
Poisson bracket $\{\cdot,\cdot\}$ is a Poisson--Lie group if 
\[
\{f_{1},f_{2}\}(xy)=\{\rho_{y}f_{1},\rho_{y}f_{2}\}(x)+\{\lambda_{x}f_{1},\lambda_{x}f_{2}\}(y),
\]
 where $\rho_{y}$ and $\lambda_{x}$ are, respectively, right and
left translation operators on $\mathcal{G}$.

Given a Lie group $\mathcal{G}$ with a Lie algebra $\mathfrak{g}$,
let $(\ ,\ )$ be a nondegenerate bilinear form on $\mathfrak{g}$,
and $\mathfrak{t}\in\mathfrak{g}\otimes\mathfrak{g}$ be the corresponding
Casimir element. For an element $r=\sum_{i}a_{i}\otimes b_{i}\in\mathfrak{g}\otimes\mathfrak{g}$
denote 
\[
\left[\left[r,r\right]\right]=\sum_{i,j}\left[a_{i},a_{j}\right]\otimes b_{i}\otimes b_{j}+\sum_{i,j}a_{i}\otimes\left[b_{i},a_{j}\right]\otimes b_{j}+\sum_{i,j}a_{i}\otimes a_{j}\otimes\left[b_{i},b_{j}\right]
\]
 and $r^{21}=\sum_{i}b_{i}\otimes a_{i}$.

The \emph{Classical Yang--Baxter equation (CYBE) }is 
\begin{equation}
\left[\left[r,r\right]\right]=0,\label{eq:CYBE}
\end{equation}
 an element $r\in\mathfrak{g}\otimes\mathfrak{g}$ that satisfies
\eqref{eq:CYBE} together with the condition 
\begin{equation}
r+r^{21}=\mathfrak{t}
\end{equation}
 is called a classical R-matrix.

A classical R-matrix $r$ induces a Poisson-Lie structure on $\mathcal{G}$:
choose a basis $\left\{ I_{\alpha}\right\} $ in $\mathfrak{g}$,
and denote by $\partial_{\alpha}$ (resp., $\partial'_{\alpha}$)
the left (resp., right) invariant vector field whose value at the
unit element is $I_{\alpha}$. Let $r=\sum_{\alpha,\beta}r_{\alpha,\beta}I_{\alpha}\otimes I_{\beta}$,
then 
\begin{equation}
\{f_{1},f_{2}\}_{r}=\sum_{\alpha,\beta}r_{\alpha,\beta}\left(\partial_{\alpha}f_{1}\partial_{\beta}f_{2}-\partial_{\alpha}^{\prime}f_{1}\partial_{\beta}^{\prime}f_{2}\right)\label{eq:sklnPB}
\end{equation}
 defines a Poisson bracket on $\mathcal{G}$. This is called the \emph{Sklyanin
bracket} corresponding to $r$.

In \cite{BDsolCYBE} Belavin and Drinfeld give a classification of
classical R-matrices for simple complex Lie groups: let $\mathfrak{g}$
be a simple complex Lie algebra with a fixed nondegenerate invariant
symmetric bilinear form $(\ ,\ )$. Fix a Cartan subalgebra $\mathfrak{h}$,
a root system $\Phi$ of $\mathfrak{g}$, and a set of positive roots
$\Phi^{+}$. Let $\Delta\subseteq\Phi^{+}$ be a set of positive simple
roots.

A Belavin--Drinfeld (BD) triple is two subsets $\Gamma_{1},\Gamma_{2}\subset\Delta$
and an isometry $\gamma:\Gamma_{1}\to\Gamma_{2}$ with the following
property: for every $\alpha\in\Gamma_{1}$ there exists $m\in\mathbb{N}$
such that $\gamma^{j}(\alpha)\in\Gamma_{1}$ for $j=0,\ldots,m-1$,
but $\gamma^{m}(\alpha)\notin\Gamma_{1}$. The isometry $\gamma$
extends in a natural way to a map between root systems generated by
$\Gamma_{1},\Gamma_{2}$. This allows one to define a partial ordering
on the root system: $\alpha\prec\beta$ if $\beta=\gamma^{j}\left(\alpha\right)$
for some $j\in\mathbb{N}$.

Select now root vectors $E_{\alpha}\in\mathfrak{g}$ that satisfy
$\left(E_{\alpha},E_{-\alpha}\right)=1$. According to the Belavin--Drinfeld
classification, the following is true (see, e.g., \cite[Ch. 3]{chriprsly}). 
\begin{prop}
(i) Every classical R-matrix is equivalent (up to an action of $\sigma\otimes\sigma$
where $\sigma$ is an automorphism of $\mathfrak{g}$) to 
\begin{equation}
r=r_{0}+\sum_{\alpha\in\Phi^{+}}E_{-\alpha}\otimes E_{\alpha}+\sum_{\begin{subarray}{c}
\alpha\prec\beta\\
\alpha,\beta\in\Phi^{+}
\end{subarray}}E_{-\alpha}\wedge E_{\beta}\label{eq:RmtxCons}
\end{equation}
 (ii) $r_{0}\in\mathfrak{h}\otimes\mathfrak{h}$ in \eqref{eq:RmtxCons}
satisfies 
\begin{equation}
\left(\gamma\left(\alpha\right)\otimes\Id\right)r_{0}+\left(\Id\otimes\alpha\right)r_{0}=0\label{eq:r0Cond1}
\end{equation}
 for any $\alpha\in\Gamma_{1}$, and 
\begin{equation}
r_{0}+r_{0}^{21}=\mathfrak{t}_{0},\label{eq:r0Cond2}
\end{equation}
 where $\mathfrak{t}_{0}$ is the $\mathfrak{h}\otimes\mathfrak{h}$
component of $\mathfrak{t}$. \\
(iii) Solutions $r_{0}$ to \eqref{eq:r0Cond1},\eqref{eq:r0Cond2}
form a linear space of dimension $k_{T}=\left|\Delta\setminus\Gamma_{1}\right|$. 
\end{prop}
Two classical R-matrices of the form \eqref{eq:RmtxCons} that are
associated with the same BD triple are said to belong to the same
Belavin--\emph{Drinfeld class. }The corresponding bracket defined
in~\eqref{eq:sklnPB} by an R-matrix $r$ associated with a triple
$T$ will be denoted by $\{\ ,\ \}_{T}$.

Given a BD triple $T$ for $\mathcal{G}$, write 
\[
\mathfrak{h}_{T}=\left\{ h\in\mathfrak{h}:\alpha(h)=\beta(h)\text{ if }\alpha\prec\beta\right\} ,
\]
 and define the torus $\mathcal{H}_{T}=\exp\mathfrak{h}_{T}\subset\mathcal{G}$.

\subsection{Main results and outline\label{sub:Main-results}}

The following conjecture was given by Gekhtman, Shapiro and Vainshtein
in \cite{gekhtman2012cluster}:
\begin{conjecture}
\label{Conj:GSV-BD-CS}Let $\mathcal{G}$ be a simple complex Lie
group. For any Belavin--Drinfeld triple $T=(\Gamma_{1},\Gamma_{2},\gamma)$
there exists a cluster structure $\mathcal{C}_{T}$ on $\mathcal{G}$
such that 
\begin{enumerate}
\item the number of stable variables is $2k_{T}$, and the corresponding
extended exchange matrix has a full rank. \label{Conj:NumStbVar} 
\item $\mathcal{C}_{T}$ is regular.\label{Conj:CTisregular} 
\item the corresponding upper cluster algebra $\overline{\mathcal{A}}_{\mathbb{C}}(\mathcal{C}_{T})$
is naturally isomorphic to $\mathcal{O}(\mathcal{G})$; \label{Conj:A(C)eqlsO(G)} 
\item the global toric action of $(\mathbb{C}^{*})^{2k_{T}}$ on $\mathbb{C}\left(\mathcal{G}\right)$
is generated by the action of $\mathcal{H}_{T}\otimes\mathcal{H}_{T}$
on $\mathcal{G}$ given by $\left(H_{1},H_{2}\right)\left(X\right)=H_{1}XH_{2}$
; \label{conj:global-toric} 
\item for any solution of CYBE that belongs to the Belavin--Drinfeld class
specified by $T$, the corresponding Sklyanin bracket is compatible
with $\mathcal{C}_{T}$; \label{Conj:Compatible} 
\item a Poisson--Lie bracket on $\mathcal{G}$ is compatible with $\mathcal{C}_{T}$
only if it is a scalar multiple of the Sklyanin bracket associated
with a solution of CYBE that belongs to the Belavin--Drinfeld class
specified by $T$. \label{Conj:a-Poisson--Lie-bracket} 
\end{enumerate}
\end{conjecture}

The main result of this paper is the following theorem:
\begin{thm}
\label{thm:MainRes}For any Belavin--Drinfeld triple $T=\left(\left\{ \alpha\right\} ,\left\{ \beta\right\} ,\gamma:\alpha\mapsto\beta\right)$,
there exists a cluster structure on $SL_{n}$ with a locally regular
initial seed and with $2k_{T}$ stable variables, that is compatible
with the Sklyanin bracket associated with $T$. 
\end{thm}
In other words, Theorem \ref{thm:MainRes} states that part \ref{Conj:NumStbVar}
of Conjecture \ref{Conj:GSV-BD-CS} ise true for $SL_{n}$ for BD
triple with $\left|\Gamma_{1}\right|=1$.

For a given $n$ and a BD triple $T_{\alpha\beta}$, a set $\mathcal{B}_{\alpha\beta}$
of functions in $\mathcal{O}\left(SL_{n}\right)$ is constructed in
Section \ref{sub:Constructing-a-log}. The rest of Section \ref{sec:A-log-canonical-basis}
is dedicated to proving that this set is log canonical with respect
to the Sklyanin bracket $\left\{ \cdot,\cdot\right\} _{\alpha\beta}$
associated with $T_{\alpha\beta}$. After declaring some of these
functions as frozen variables and introducing the quiver $Q_{\alpha\beta}^{n}$
in Section \ref{sub:The-quiver}, the initial seed $\left(\mathcal{B}_{\alpha\beta},Q_{\alpha\beta}^{n}\right)$
determines a cluster structure $\mathcal{C}_{\alpha\beta}$. Theorem
\ref{thm:Compatible} states that $\mathcal{C}_{\alpha\beta}$ is
compatible with the bracket $\left\{ \cdot,\cdot\right\} _{\alpha\beta}$,
and Section \ref{sec:Regularity} proves
that the initial seed is locally regular. Last, Section \ref{sec:Technical}
has some technical computations and results that were used through
the paper.

Parts and \ref{Conj:CTisregular} -- \ref{Conj:a-Poisson--Lie-bracket}
of the conjecture will be proved in the companion paper \cite{eisner2014part2}.

A Poisson--Lie bracket on $SL_{n}$ can be extended to one on $GL_{n},$
with the determinant being a Casimir function. From here on we discuss
$GL_{n}$, and any statement can be restricted to $SL_{n}$ by removing
the determinant function.

\section{A log canonical basis}

\label{sec:A-log-canonical-basis}

This section describes a log canonical set of function, that will
serve as an initial cluster for the structure $\mathcal{C}_{\alpha\beta}.$
After constructing this set in Section \ref{sub:Constructing-a-log},
we show it is log canonical with respect to the bracket $\left\{ \cdot,\cdot\right\} _{\alpha\beta}$
in Section \ref{sub:BisLC}, using results from Section \ref{sec:Technical}.

Before moving on, note the following two isomorphisms of the BD data
for $SL_{n}$: the first reverses the direction of $\gamma$ and transposes
$\Gamma_{1}$ and $\Gamma_{2}$, while the second one takes each root
$\alpha_{j}$ to $\alpha_{\omega_{0}\left(j\right)}$, where $\omega_{0}$
is the longest element in the Weyl group (which in $SL_{n}$ is naturally
identified with the symmetric group $S_{n-1}$). These two isomorphisms
correspond to the automorphisms of $SL_{n}$ given by $X\mapsto-X^{t}$
and $X\mapsto\omega_{0}X\omega_{0}$, respectively. Since R-matrices
are considered up to an action of $\sigma\otimes\sigma$, from here
on we do not distinguish between BD triples obtained one from the
other via these isomorphisms. We will also assume that in the map
$\gamma:\alpha_{i}\mapsto\alpha_{j}$ we always have $i<j$.

Slightly abusing the notation, we sometime refer to a root $\alpha_{i}\in\Delta$
just as $i,$ and write $\gamma:i\mapsto j$ instead of $\gamma:\alpha_{i}\mapsto\alpha_{j}$.
For shorter notation, denote the BD triple
$\left(\left\{ \alpha\right\} ,\left\{ \beta\right\} ,\gamma:\alpha\mapsto\beta\right)$
by $T_{\alpha\beta}$, and naturally the corresponding Sklyanin bracket
will be $\left\{ \cdot,\cdot\right\} _{\alpha\beta}$ .

\subsection{Constructing a log canonical basis \label{sub:Constructing-a-log}}

For a triple $T_{\alpha\beta}$ we will construct a set of matrices
$\mathcal{M}$ such that the set of all their trailing principal minors
is log canonical with respect to $\left\{ \cdot,\cdot\right\} _{\alpha\beta}$
. A trailing principal minor of an $n\times n$ matrix $M$ is a minor
of $M$ of the form $\det M_{[i,n]}^{[j,n]}$.

Following \cite{reyman1994group}, recall the construction of \emph{the
Drinfeld double} of a Lie algebra $\mathfrak{g}$ with the Killing
form $\left\langle \ ,\ \right\rangle $: define $D\left(\mathfrak{g}\right)=\mathfrak{g}\oplus\mathfrak{g}$,
with an invariant nondegenerate bilinear form 
\[
\left\langle \left\langle \left(\xi,\eta\right),\left(\xi',\eta'\right)\right\rangle \right\rangle =\left\langle \xi,\xi'\right\rangle -\left\langle \eta,\eta'\right\rangle .
\]
 Define subalgebras $\mathfrak{d}_{\pm}$ of $D\left(\mathfrak{g}\right)$
by 
\begin{equation}
\mathfrak{d}_{+}=\left\{ \left(\xi,\xi\right):\xi\in\mathfrak{g}\right\} ,\quad\mathfrak{d_{-}}=\left\{ \left(R_{+}\left(\xi\right),R_{-}\left(\xi\right)\right):\xi\in\mathfrak{g}\right\} ,\label{eq:DefSubAlgDpDm}
\end{equation}

where $R_{\pm}\in\End\mathfrak{g}$ are defined for any R-matrix $r$
by 
\begin{equation}
\left\langle R_{+}\left(\eta\right),\zeta\right\rangle =-\left\langle R_{-}\left(\zeta\right),\eta\right\rangle =\left\langle r,\eta\otimes\zeta\right\rangle _{\otimes},\label{eq:RplsDef}
\end{equation}
 and $\langle\ ,\ \rangle_{\otimes}$ is the corresponding Killing
form on the tensor square of $\mathfrak{g}$.

For a matrix $X$ let $M_{ij}\left(X\right)$ be the maximal contiguous
submatrix of $X$ with $x_{ij}$ at the upper left hand corner. That
is, 
\begin{eqnarray*}
M_{ij}\left(X\right) & = & \begin{cases}
\left[\begin{array}{ccc}
x_{ij} & \cdots & x_{in}\\
\vdots &  & \vdots\\
x_{n-j+i,j} & \cdots & x_{n-j+i,n}
\end{array}\right] & \text{if }j>i\\
\\
\left[\begin{array}{ccc}
x_{ij} & \cdots & x_{i,n-i+j}\\
\vdots &  & \vdots\\
x_{nj} & \cdots & x_{n,n-i+j}
\end{array}\right] & \text{otherwise.}
\end{cases}
\end{eqnarray*}

Slightly abusing the notation, define $M_{ij}\left(X,Y\right)$ on
the double $D\left(\mathfrak{gl}_{n}\right)$
by 
\[
M_{ij}\left(X,Y\right)=\begin{cases}
M_{ij}\left(X\right) & \text{if }i\geq j\\
M_{ij}\left(Y\right) & \text{otherwise, }
\end{cases}
\]
and we can then write $M_{ij}\left(X\right)=M_{ij}\left(X,X\right).$
Let $X_{R}^{C}$ denote the submatrix of $X$ with rows in the set
$R$ and columns in $C$ (with $R,C\subseteq\left[n\right]$). Then
define two special families of matrices: for $1\leq j\leq\alpha$
and $i=n+j-\alpha$ set 
\[
\tilde{M}_{ij}(X,Y)=\left[\begin{array}{cc}
X_{\left[i,n\right]}^{\left[j,\alpha+1\right]} & 0_{\left(n-i+1\right)\times\mu}\\
0_{\mu\times\left(n-i+1\right)} & Y_{\left[1,\mu\right]}^{\left[\beta,n\right]}
\end{array}\right]
\]
with $\mu=n-\beta$, and for $1\le i\le\beta$ and $j=n+i-\beta$,
set 
\[
\tilde{M}_{ij}(X,Y)=\left[\begin{array}{cc}
Y_{\left[i,\beta+1\right]}^{\left[j,n\right]} & 0_{\left(n-j+1\right)\times\mu}\\
0_{\mu\times\left(n-j+1\right)} & X_{\left[\alpha,n\right]}^{\left[1,\mu\right]}
\end{array}\right],
\]
 and here $\mu=n-\alpha$. Note that these matrices are not block
diagonal: in the first case the number of columns in each of the two
blocks is greater than the number of rows by one, while in the second
case the number of rows in each block is greater than the number of
columns by one. As above, we set $\tilde{M}_{ij}(X)=\tilde{M}_{ij}(X,X)$.

When $n$ is even there are two special cases - $\alpha=\frac{n}{2}$
or $\beta=\frac{n}{2}$. We discuss here the case $\beta=\frac{n}{2}$,
as the case of $\alpha=\frac{n}{2}$ is symmetric (and isomorphic
under $\alpha\longleftrightarrow\beta$): for $i=j+\alpha$ the matrix
$\tilde{M}_{ij}\left(X,Y\right)$ now involves three blocks (submatrices
of $X$ and $Y$), and it has the form 
\[
\tilde{M}_{ij}\left(X,Y\right)=\left[\begin{array}{ccccccccc}
x_{ij} & \cdots & \cdots & x_{i,\alpha+1} & 0 & \cdots &  & \cdots & 0\\
 & \ddots &  & \vdots & \vdots &  &  &  & \vdots\\
x_{n1} & \cdots & x_{n\alpha} & x_{n,\alpha+1} & 0 & \cdots &  &  & \vdots\\
0 & \cdots & y_{1\beta} & y_{1,\beta+1} & \cdots & y_{1n} & 0 & \cdots & 0\\
 &  & \vdots & \vdots & \ddots & \vdots & \vdots &  & \vdots\\
 &  & \vdots & y_{\beta j} & \cdots & y_{\beta n} & x_{\alpha1} & \cdots & x_{\alpha\mu}\\
 &  & y_{\beta+1} & y_{\beta+1,j} & \cdots & y_{\beta+1,n} & x_{\alpha+1,1} &  & \vdots\\
 &  &  & 0 & \cdots & 0 & \vdots & \ddots & \vdots\\
 &  &  & \vdots &  & \vdots & x_{n1} & \cdots & x_{n\mu}
\end{array}\right].
\]
 Now define 
\begin{equation}
f_{ij}(X)=\det M_{ij}(X).\label{eq:fijDefasDet}
\end{equation}
The set $\mathcal{B}_{std}=\left\{ f_{ij}(X)|i,j\in\left[n\right]\right\} $
of determinants of all matrices $M_{ij}(X)$ forms a log canonical
set with respect to the standard bracket \cite[Ch. 4.3]{GSV}. For
the $\alpha\mapsto\beta$ case, take this set and for all $i=n+j-\alpha$
and $j=n+i-\beta$ replace $M_{ij}$$\left(X\right)$ with $\tilde{M}_{ij}\left(X\right)$.
This assures that for a fixed pair $\left(i,j\right)\in[n]\times[n]$
there is still a unique matrix in the set. Denote this matrix (either
$M_{ij}(X)$ or $\tilde{M}_{ij}\left(X\right)$)
by $\overline{M}_{ij}$, and set 
\[
\varphi_{ij}=\det\overline{M}_{ij}.
\]

Our set of log canonical functions (with respect to the bracket $\left\{ \cdot,\cdot\right\} _{\alpha\beta}$) -- that will later serve as an initial cluster -- is the set $\mathcal{B}_{\alpha\beta}=\left\{ \varphi_{ij}|i,j\in\left[n\right]\right\} $.
Further on we will also need the set $\mathcal{B}^{D}$ of functions
on $D\left(\mathfrak{gl}_{n}\right)$, defined
by $\mathcal{B}^{D}=\left\{ \varphi_{ij}^{D}\left(X,Y\right)=\det M_{ij}\left(X,Y\right)|i,j\in\left[n\right]\right\} $.

Some matrices in the above construction contain others: for example,
$M_{1j}\left(X\right)$ contains all matrices $M_{ik}$ with $k=j+i-1$.
Therefore, we can see the set $\mathcal{B}$ as the set of all trailing
principal minors of matrices $M_{1j}\left(X\right)$ and $M_{i1}\left(X\right)$,
excluding $M_{\alpha+1,1}\left(X\right)$ and $M_{1,\beta+1}\left(X\right)$.
So the set $\mathcal{B}_{std}$ can be viewed as all trailing principal
minors of the matrices $M_{1j}$ and $M_{i1}$ with $i,j\in[n]$.
We will denote this set of matrices by $\mathcal{M}_{std}$. Equivalently,
define the set 
\begin{equation}
\mathcal{M}_{\alpha\beta}=\left\{ \overline{M}_{1j},\overline{M}_{i1}|i,j\in[n]\right\} \setminus\left\{ \overline{M}_{1,\beta+1},\overline{M}_{\alpha+1,1}\right\} ,\label{eq:Mabdef}
\end{equation}
and it is not hard to see that $\mathcal{B}_{\alpha\beta}=\left\{ \varphi_{ij}|i,j\in\left[n\right]\right\} $
is the set of trailing principal minors of all matrices in $\mathcal{M}_{\alpha\beta}$.

Clearly, $\left|\mathcal{B}_{\alpha\beta}\right|=n^{2}$, since the
map $\left(i,j\right)\mapsto\varphi_{ij}$ is a bijection between
$\left[n\right]\times\left[n\right]$ and $\mathcal{B}_{\alpha\beta}$.
In Section \ref{sub:BisLC} we show that $\mathcal{B}_{\alpha\beta}$
is log canonical with respect to the bracket $\left\{ \cdot,\cdot\right\} _{\alpha\beta}.$
\begin{rem}
Further details about the construction of a log canonical set from
determinants of matrices as above can be found in \cite{eisner2014SL5}.
The special case of $n=5$ is addressed there, with any BD data, but
it can be easily generalized to any $n$ (with the restriction $\left|\Gamma_{1}\right|=\left|\Gamma_{2}\right|=1$).
\end{rem}

\subsection{The log canonical set $\mathcal{B}_{\alpha\beta}$ \label{sub:BisLC}}

Comparing the bracket $\left\{ \cdot,\cdot\right\} _{\alpha\beta}$
with the standard one will allow us to compute $\left\{ f,g\right\} _{\alpha\beta}$
for every pair of functions $f,g\in\mathcal{B}_{\alpha\beta}.$ We
will use results from Section \ref{sec:Technical}.

Since some of the proofs involve the standard Poisson bracket and
cluster structure on $SL_{n}$, we start with a reminder: there are
multiple Poisson brackets on $SL_{n}$ that correspond to the trivial
BD data $\Gamma_{1}=\Gamma_{2}=\emptyset$, since $r_{0}$ is not
uniquely determined. For a pair $\alpha,\beta$ we will use $r_{0}$
as defined in \eqref{eq:r0def}, and call the associated Poisson bracket
\emph{the }standard Poisson bracket on $SL_{n}.$ The corresponding
cluster structure on $SL_{n}$ that will be called the standard one
and denoted $\mathcal{C}_{std}$ is described in \cite{BFZ} and \cite{gekhtman2012cluster}.
Note that this cluster structure is independent on the choice of $r_{0}$
and the Poisson bracket. The initial seed of this cluster structure
is the set $\{f_{ij}\}_{i,j=1}^n$
defined by 
\begin{equation}
f_{ij}=\begin{cases}
          \det X_{\left[i,n+i-j\right]}^{\left[j,n\right]} & \text{ if } j\geq i,\\
          \det X_{\left[i,n\right]}^{\left[j,n-i+j\right]} & \text{ otherwise. }
          \end{cases}         
          \label{eq:fijAsXSubMat}
\end{equation}
looks as follows: set $\mu\left(i,j\right)=\min\left(n,n+i-j\right)$
and write 

This definition coincides with \eqref{eq:fijDefasDet}, and for all
$\varphi_{ij}\in\mathcal{B}_{\alpha\beta}\cap\mathcal{B}_{std}$ we
have $\varphi_{ij}=f_{ij}$.

The function $f_{11}=\det X$ is constant on $SL_{n}$. Take the set
\[
\left\{ f_{ij}\right\} _{i,j=1}^{n}\setminus\left\{ f_{11}\right\} 
\]
 as the set of cluster variables. Set the variables $f_{i1}$ and
$f_{1j}$ to be frozen, so there are $n^{2}-1$ cluster variables
with $2\left(n-1\right)$ of them frozen. Let $Q_{std}^{n}$ be the
quiver of $\mathcal{C}_{std}^{n}$ (see \cite{BFZ,gekhtman2012cluster}.
The vertices of $Q_{std}^{n}$ are placed on an $n\times n$ grid
with rows numbered from top to bottom and columns numbered from left
to right. The cluster variable $f_{ij}$ corresponds to the node $\left(i,j\right)$
(that is, the node on the $i$-th row and the $j$-th column). There
are arrows from each node $(i,j)$ to $\left(i,j+1\right)$ (as long
as $j\neq n$), from $\left(i,j\right)$ to $\left(i+1,j\right)$
(when $i\neq n)$ and from $\left(i+1,j+1\right)$ to $\left(i,j\right)$.
Arrows connecting two frozen variables can be ignored. As explained
at the end of Section \ref{sub:Main-results}, we can extend from
$SL_{n}$ to $GL_{n}$ by adding the function $f_{11}=\det X$. Figure
\ref{fig:GL5StdQvr} shows the initial quiver of the standard cluster
structure on $GL_{5}$ (remove the upper left node with the arrow
incident to it to get the initial standard quiver for $SL_{5}$).
Mutable variables are represented by circles, while frozen ones are
represented by squares. Then $\Sigma_{std}=\left(\mathcal{B}_{std},Q_{std}^{n}\right)$
is an initial seed for the standard cluster structure on $SL_{n}$
\cite{BFZ,gekhtman2012cluster}.

\begin{figure}
\begin{centering}
\includegraphics[scale=0.45]{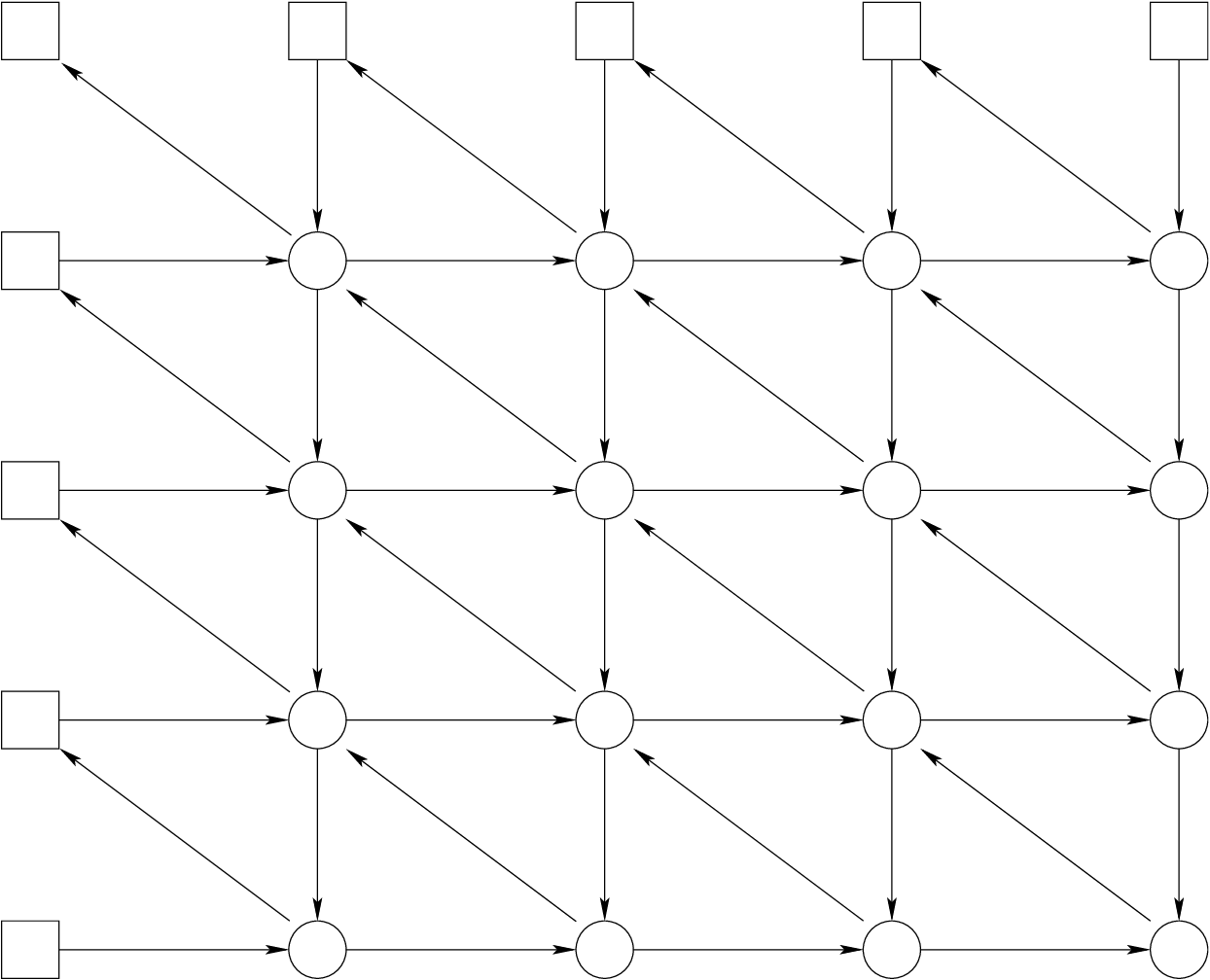} 
\par\end{centering}

\caption{The standard quiver for $GL_{5}$}

\centering{}\label{fig:GL5StdQvr} 
\end{figure}

We will use the following notations: 
\begin{eqnarray}
f^{i\leftarrow j} & = & \left(\nabla f\cdot X\right)_{ij}=\sum_{k=1}^{n}\frac{\partial f}{\partial x_{ki}}x_{kj}\\
f_{j\leftarrow i} & = & \left(X\cdot\nabla f\right)_{ij}=\sum_{k=1}^{n}\frac{\partial f}{\partial x_{jk}}x_{ik}.
\end{eqnarray}
 Note that if $f$ is a determinant of a submatrix $S$ of a matrix
$X$, then $f^{i\leftarrow j}$ (or $f_{i\leftarrow j}$) is the same
determinant, with column (or row) $i$ replaced by column (row) $j$.
If $S$ does not contain column (row) $i$, then $f^{i\leftarrow j}=0\ \left(f_{i\leftarrow j}=0\right).$
If $f=f_{ij}=\det S$, where $S=X_{\left[i,k\right]}^{\left[j,\ell\right]}$
is a dense submatrix of $X$, we write 
\begin{eqnarray*}
f^{\rightarrow} & = & f^{\ell\leftarrow\ell+1}\\
f^{\leftarrow} & = & f^{j\leftarrow j-1}\\
f^{\uparrow} & = & f_{i\leftarrow i-1}\\
f^{\downarrow} & = & f_{k\leftarrow k+1}.
\end{eqnarray*}

For a pair $\left(f,g\right)$ of log canonical functions, denote
by $\omega_{f,g}$ the Poisson coefficient 
\begin{equation}
\omega_{f,g}=\frac{\left\{ f,g\right\} _{std}}{fg}.\label{eq:PsnCfDef}
\end{equation}

Our first result states that the functions we defined in \ref{sub:Constructing-a-log}
are indeed log canonical:
\begin{thm}
The set $\mathcal{B}_{\alpha\beta}$ is log canonical with respect
to the bracket $\left\{ \cdot,\cdot\right\} _{\alpha\beta}$. \label{thm:BisLC}\end{thm}
\begin{proof}
Compute the bracket $\left\{ f,g\right\} _{\alpha\beta}$ for all
$f,g\in\mathcal{B}_{\alpha\beta}$: first, by Corollary \ref{cor:fginSbrckteq},
if $f,g\in\mathcal{B}_{\alpha\beta}\cap\mathcal{B}_{std}$ then $\left\{ f,g\right\} _{\alpha\beta}=\left\{ f,g\right\} _{std}$,
and therefore $f,g$ are log canonical with respect to $\left\{ \cdot,\cdot\right\} _{\alpha\beta}$.
Now turn to $\left\{ f,g\right\} _{\alpha\beta}$ where $f$ or $g$
are non standard basis functions. These are functions of the form
$\varphi_{ij}$ with $i=n+j-\alpha$ or $j=n+i-\beta$, so for $k\in\left[\alpha\right]$
and $m\in\left[\beta\right]$ define 
\begin{eqnarray*}
\theta_{k} & =\varphi_{n+k-\alpha,k}= & f_{n+k-\alpha,k}\cdot f_{1,\beta+1}-f_{n+k-\alpha,k}^{\rightarrow}\cdot f_{1,\beta+1}^{\leftarrow}\\
\psi_{m} & =\varphi_{m,n+m-\beta}= & f_{m,n+m-\beta}\cdot f_{\alpha+1,1}-f_{m,n+m-\beta}^{\downarrow}\cdot f_{\alpha+1,1}^{\uparrow},
\end{eqnarray*}
and so $f$ or $g$ (or both) are either $\theta_{k}$ or $\psi_{m}$. 

Take the bracket $\left\{ \theta_{k},g\right\} _{\alpha\beta}$ with
$g\in\mathcal{B}_{\alpha\beta}\cap\mathcal{B}_{std}$ and look at
three cases:

1. If $g\neq f_{m,\beta+1}$ for some $m\in\left[n\right],$ and $g\neq f_{n+m-\alpha,m}$
for some $m>k$, we can write 
\begin{eqnarray}
\left\{ \theta_{k},g\right\} _{\alpha\beta} & = & \left\{ f_{n+k-\alpha,k}\cdot f_{1,\beta+1},g\right\} _{\alpha\beta}-\left\{ f_{n+k-\alpha,k}^{\rightarrow}\cdot f_{1,\beta+1}^{\leftarrow},g\right\} _{\alpha\beta}\nonumber \\
 & = & \left\{ f_{n+k-\alpha,k}\cdot f_{1,\beta+1},g\right\} _{std}-\left\{ f_{n+k-\alpha,k}^{\rightarrow}\cdot f_{1,\beta+1}^{\leftarrow},g\right\} _{std}\text{.}\label{eq:abbrcktPhig}
\end{eqnarray}
 The last equality holds since by Lemma \ref{lem:PsnBrcktDiff} the
difference between the $\alpha\beta$ bracket and the standard bracket
is 
\begin{align*}
\left\{ f,g\right\} _{\alpha\beta}-\left\{ f,g\right\} _{std}= & f^{\alpha\leftarrow\alpha+1}g^{\beta+1\leftarrow\beta}-f^{\beta+1\leftarrow\beta}g^{\alpha\leftarrow\alpha+1}\\
 & +f_{\beta\leftarrow\beta+1}g_{\alpha+1\leftarrow\alpha}-f{}_{\alpha+1\leftarrow\alpha}g_{\beta\leftarrow\beta+1}.
\end{align*}
Since $g\in\mathcal{B}_{\alpha\beta}\cap\mathcal{B}_{std}$ and $g\neq f_{m,\beta+1}$
we get $g^{\beta+1\leftarrow\beta}=0$, and also $g^{\alpha\leftarrow\alpha+1}=0$,
because $g\neq f_{n+m-\alpha,m}$. Similarly, $f_{\beta\leftarrow\beta+1}=0$
for the functions $f=f_{n+k-\alpha,k}$ and $f=f_{n+k-\alpha,k}^{\rightarrow},$ and
of course $g_{\beta\leftarrow\beta+1}=0$ (unless $g$ is of the form
$g=\det X_{[1,\beta]}^{[j,\mu]}$ but this function can not in $\mathcal{B}_{\alpha\beta}).$
Therefore, the $\alpha\beta$ bracket and the standard bracket are
equal in this case.

According to Lemma \ref{lem:NewFuncLCandCf}, the functions $f_{n+k-\alpha,k}^{\rightarrow}$
and $f_{1,\beta+1}^{\leftarrow}$ are both log canonical with $g$
(w.r.t. the standard bracket) with Poisson coefficients 
\begin{eqnarray*}
\omega_{f_{n+k-\alpha,k}^{\rightarrow},g} & = & \omega_{f_{n+k-\alpha,k},g}+\omega_{x_{n,\alpha+1},g}-\omega_{x_{n\alpha},g},\\
\omega_{f_{1,\beta+1}^{\leftarrow},g} & = & \omega_{f_{1,\beta+1},g}+\omega_{x_{n\beta},g}-\omega_{x_{n,\beta+1},g},
\end{eqnarray*}
 so \eqref{eq:abbrcktPhig} turns to 
\begin{eqnarray*}
\left\{ \theta_{k},g\right\} _{\alpha\beta} & = & \omega_{1}f_{n+k-\alpha,k}\cdot f_{1,\beta+1}\cdot g\\
 &  & -\left(\omega_{1}-s\omega_{\alpha\beta}\left(g\right)\right)f_{n+k-\alpha,k}^{\leftarrow}\cdot f_{1,\beta+1}^{\rightarrow}\cdot g,
\end{eqnarray*}
 with 
\[
\omega_{1}=\omega_{f_{n+k-\alpha,k},g}+\omega_{f_{1,\beta+1},g}
\]
 and 
\[
s\omega_{\alpha\beta}\left(g\right)=\omega_{f_{n\alpha},g}-\omega_{f_{n,\alpha+1},g}-\omega_{f_{n\beta},g}+\omega_{f_{n,\beta+1},g},
\]
 as defined in \eqref{eq:swabDef}. \\
Now, using Lemma \ref{lem:SumCfeq0} we get 
\[
\left\{ \theta_{k},g\right\} _{\alpha\beta}=\left(\omega_{f_{n+k-\alpha,k},g}+\omega_{f_{1,\beta+1},g}\right)\theta_{k}\cdot g.
\]

2. If $g=f_{m,\beta+1}$ for some $m\in\left[2,n\right]$, we write
$q_{k}=f_{n+k-\alpha,k}$, and then 
\begin{eqnarray*}
\left\{ \theta_{k},g\right\} _{\alpha\beta} & = & \left\{ q_{k}\cdot f_{1,\beta+1},g\right\} _{\alpha\beta}-\left\{ q_{k}^{\rightarrow}\cdot f_{1,\beta+1}^{\leftarrow},g\right\} _{\alpha\beta}\\
 & = & q_{k}\left\{ f_{1,\beta+1},g\right\} _{\alpha\beta}+f_{1,\beta+1}\left\{ q_{k},g\right\} _{\alpha\beta}\\
 &  & -q_{k}^{\rightarrow}\left\{ f_{1,\beta+1}^{\leftarrow},g\right\} _{\alpha\beta}-f_{1,\beta+1}^{\leftarrow}\left\{ q_{k}^{\rightarrow},g\right\} _{\alpha\beta}\\
 & = & q_{k}\left\{ f_{1,\beta+1},g\right\} _{std}+f_{1,\beta+1}\left\{ q_{k},g\right\} _{std}\\
 &  & +f_{1,\beta+1}q_{k}^{\rightarrow}g^{\leftarrow}-q_{k}^{\rightarrow}\left\{ f_{1,\beta+1}^{\leftarrow},g\right\} _{std}\\
 &  & -f_{1,\beta+1}^{\leftarrow}\left\{ q_{k}^{\rightarrow},g\right\} _{std}.
\end{eqnarray*}
These are brackets of log canonical functions, except for $\left\{ f_{1,\beta+1}^{\leftarrow},g\right\} _{std}$
which is given in Lemma \ref{lem:SpeFuncBrckt}, so it is 
\begin{eqnarray*}
\left\{ \theta_{k},g\right\} _{\alpha\beta} & = & \left(\omega_{f_{1,\beta+1},g}+\omega_{q_{k},g}\right)q_{k}\cdot f_{1,\beta+1}\cdot g\\
 &  & +f_{1,\beta+1}q_{k}^{\rightarrow}g^{\leftarrow}\\
 &  & -\left(\omega_{f_{1,\beta+1},g}+\omega_{x_{n\beta},g}-\omega_{x_{n,\beta+1},g}\right)q_{k}^{\rightarrow}\cdot f_{1,\beta+1}^{\leftarrow}\cdot g\\
 &  & -q_{k}^{\rightarrow}f_{1,\beta+1}g^{\leftarrow}\\
 &  & -\left(\omega_{q_{k},g}-\omega_{x_{n\alpha},g}+\omega_{x_{n,\alpha+1},g}\right)f_{1,\beta+1}^{\leftarrow}\cdot q_{k}^{\rightarrow}\cdot g\\
 & = & \left(\omega_{f_{1,\beta+1},g}+\omega_{q_{k},g}\right)q_{k}\cdot f_{1,\beta+1}\cdot g-\\
 &  & \left(\omega_{f_{1,\beta+1},g}+\omega_{q_{k},g}-s\omega_{\alpha\beta}\left(g\right)\right)f_{1,\beta+1}^{\leftarrow}\cdot q_{k}^{\rightarrow}\cdot g,
\end{eqnarray*}
and with Lemma \ref{lem:SumCfeq0} this comes down to 
\begin{equation}
\left\{ \theta_{k},g\right\} _{\alpha\beta}=\left(\omega_{f_{1,\beta+1},g}+\omega_{q_{k},g}\right)\theta_{k}g.\label{eq:Phigcft}
\end{equation}

3. Now look at $g=f_{n+m-\alpha,m}$ for some $m>k$ : with Lemma
\ref{lem:SpeFuncBrckt} we can compute 
\begin{eqnarray*}
\left\{ \theta_{k},g\right\} _{\alpha\beta} & = & \left\{ q_{k}\cdot f_{1,\beta+1},g\right\} _{\alpha\beta}-\left\{ q_{k}^{\rightarrow}\cdot f_{1,\beta+1}^{\leftarrow},g\right\} _{\alpha\beta}\\
 & = & q_{k}\left\{ f_{1,\beta+1},g\right\} _{\alpha\beta}+f_{1,\beta+1}\left\{ q_{k},g\right\} _{\alpha\beta}\\
 &  & -q_{k}^{\rightarrow}\left\{ f_{1,\beta+1}^{\leftarrow},g\right\} _{\alpha\beta}-f_{1,\beta+1}^{\leftarrow}\left\{ q_{k}^{\rightarrow},g\right\} _{\alpha\beta}\\
 & = & q_{k}\left\{ f_{1,\beta+1},g\right\} _{std}-q_{k}f_{1,\beta+1}^{\leftarrow}g^{\rightarrow}+f_{1,\beta+1}\left\{ q_{k},g\right\} _{std}\\
 &  & -q_{k}^{\rightarrow}\left\{ f_{1,\beta+1}^{\leftarrow},g\right\} _{std}-f_{1,\beta+1}^{\leftarrow}\left\{ q_{k}^{\rightarrow},g\right\} _{std}\\
 & = & \left(\omega_{f_{1,\beta+1},g}+\omega_{q_{k},g}\right)q_{k}\cdot f_{1,\beta+1}\cdot g\\
 &  & -\left(\omega_{q_{k},g}+\omega_{f_{1,\beta+1},g}-s\omega_{\alpha\beta}\left(g\right)\right)q_{k}^{\rightarrow}\cdot f_{1,\beta+1}^{\leftarrow}\cdot g,
\end{eqnarray*}
and with Lemma \ref{lem:SumCfeq0} this is 
\begin{equation}
\left\{ \theta_{k},g\right\} _{\alpha\beta}=\left(\omega_{f_{1,\beta+1},g}+\omega_{q_{k},g}\right)\theta_{k}\cdot g.
\end{equation}
 We now turn to look at $\left\{ \theta_{k},\theta_{m}\right\} _{\alpha\beta}$:
w.l.o.g. assume $m>k$: 
\begin{eqnarray}
\left\{ \theta_{k},\theta_{m}\right\} _{\alpha\beta} & = & \left\{ q_{k}\cdot f_{1,\beta+1}-q_{k}^{\rightarrow}\cdot f_{1,\beta+1}^{\leftarrow},q_{m}\cdot f_{1,\beta+1}-q_{m}^{\rightarrow}\cdot f_{1,\beta+1}^{\leftarrow}\right\} _{\alpha\beta}\nonumber \\
 & = & \left\{ q_{k}\cdot f_{1,\beta+1},q_{m}\cdot f_{1,\beta+1}\right\} _{\alpha\beta}-\left\{ q_{k}^{\rightarrow}\cdot f_{1,\beta+1}^{\leftarrow},q_{m}\cdot f_{1,\beta+1}\right\} _{\alpha\beta}\nonumber \\
 &  & -\left\{ q_{k}\cdot f_{1,\beta+1},q_{m}^{\rightarrow}\cdot f_{1,\beta+1}^{\leftarrow}\right\} _{\alpha\beta}\label{eq:brktPhikPhim}\\
 &  & +\left\{ q_{k}^{\rightarrow}\cdot f_{1,\beta+1}^{\leftarrow},q_{m}^{\rightarrow}\cdot f_{1,\beta+1}^{\leftarrow}\right\} _{\alpha\beta}.\nonumber 
\end{eqnarray}
 The Poisson bracket satisfy the Leibniz rule: 
\[
\left\{ f_{1}\cdot f_{2},f_{3}\right\} =f_{1}\cdot\left\{ f_{2},f_{3}\right\} +\left\{ f_{1},f_{3}\right\} \cdot f_{2},
\]
 so each of the four brackets above can break into four terms of the
form $f_{1}\cdot f_{2}\cdot\left\{ f_{3},f_{4}\right\} $. We have
already seen that 
\begin{eqnarray}
 &  & \left\{ q_{k},q_{m}^{\rightarrow}\right\} _{\alpha\beta}=\begin{cases}
\left(\omega_{q_{k},q_{m}}-\omega_{q_{k},x_{n,\alpha}}+\omega_{q_{k},x_{n,\alpha+1}}\right)q_{k}q_{m}^{\rightarrow} & \text{ if }m>k\\
\left(\omega_{q_{k},q_{m}}-\omega_{q_{k},x_{n,\alpha}}+\omega_{q_{k},x_{n,\alpha+1}}\right)q_{k}q_{m}^{\rightarrow}+q_{k}^{\rightarrow}q_{m} & \text{ if }m<k
\end{cases}\nonumber \\
 &  & \left\{ q_{k},f_{1,\beta+1}^{\leftarrow}\right\} _{\alpha\beta}=\left(\omega_{q_{k},f_{1,\beta+1}}+\omega_{q_{k},x_{n\beta}}-\omega_{q_{k},x_{n,\beta+1}}\right)q_{k}f_{1,\beta+1}^{\leftarrow},\nonumber \\
 &  & \left\{ f_{1,\beta+1},f_{1,\beta+1}^{\leftarrow}\right\} _{\alpha\beta}=\left(\omega_{f_{1,\beta+1},x_{n,\beta}}-\omega_{f_{1,\beta+1},x_{n,\beta+1}}\right)f_{1,\beta+1}f_{1,\beta+1}^{\leftarrow},\nonumber \\
 &  & \left\{ q_{k}^{\rightarrow},f_{1,\beta+1}^{\leftarrow}\right\} _{\alpha\beta}=\left(\omega_{q_{k},f_{1,\beta+1}^{\leftarrow}}-\omega_{x_{n\alpha},f_{1,\beta+1}^{\leftarrow}}+\omega_{x_{n,\alpha+1},f_{1,\beta+1}^{\leftarrow}}\right)q_{k}^{\rightarrow}f_{1,\beta+1}^{\leftarrow}\nonumber \\
 &  & =\left[\left(\omega_{qk,f_{1,\beta+1}}+\omega_{q_{k},x_{n\beta}}-\omega_{q_{k},x_{n,\beta+1}}\right)\right.\label{eq:brckfktf1+bt}\\
 &  & -\left(\omega_{x_{n\alpha},f_{1,\beta+1}}+\omega_{x_{n\alpha},x_{n\beta}}-\omega_{x_{n\alpha},x_{n,\beta+1}}\right)\nonumber \\
 &  & \left.+\left(\omega_{x_{n,\alpha+1},f_{1,\beta+1}}+\omega_{x_{n,\alpha+1},x_{n\beta}}-\omega_{x_{n,\alpha+1},x_{n,\beta+1}}\right)\right]q_{k}^{\rightarrow}f_{1,\beta+1}^{\leftarrow}.\nonumber 
\end{eqnarray}
We will look at the four brackets of \eqref{eq:brktPhikPhim} one
at a time. The first one is 
\begin{eqnarray}
\left\{ q_{k}\cdot f_{1,\beta+1},q_{m}\cdot f_{1,\beta+1}\right\} _{\alpha\beta} & = & \left(f_{1,\beta+1}\right)^{2}\left\{ q_{k},q_{m}\right\} _{\alpha\beta}+q_{k}f_{1,\beta+1}\left\{ f_{1,\beta+1},q_{m}\right\} _{\alpha\beta}\nonumber \\
 &  & +q_{m}f_{1,\beta+1}\left\{ q_{k},f_{1,\beta+1}\right\} _{\alpha\beta}\nonumber \\
 & = & \left(f_{1,\beta+1}\right)^{2}\left\{ q_{k},q_{m}\right\} _{std}+q_{k}f_{1,\beta+1}\left\{ f_{1,\beta+1},q_{m}\right\} _{std}\nonumber \\
 &  & -q_{k}f_{1,\beta+1}f_{1,\beta+1}^{\leftarrow}q_{m}^{\rightarrow}+q_{m}f_{1,\beta+1}\left\{ q_{k},f_{1,\beta+1}\right\} _{std}\nonumber \\
 &  & +q_{m}f_{1,\beta+1}q_{k}^{\rightarrow}f_{1,\beta+1}^{\leftarrow}\nonumber \\
 & = & \left(\omega_{q_{k},q_{m}}+\omega_{q_{k},f_{1,\beta+1}}+\omega_{f_{1,\beta+1},q_{m}}\right)\left(f_{1,\beta+1}\right)^{2}q_{k}q_{m}\nonumber \\
 &  & +q_{k}^{\rightarrow}f_{1,\beta+1}^{\leftarrow}q_{m}f_{1,\beta+1}-q_{k}f_{1,\beta+1}f_{1,\beta+1}^{\leftarrow}q_{m}^{\rightarrow}.\label{eq:Brckfirstprt}
\end{eqnarray}
 The second bracket: 
\begin{eqnarray*}
 &  & \left\{ q_{k}^{\rightarrow}\cdot f_{1,\beta+1}^{\leftarrow},q_{m}\cdot f_{1,\beta+1}\right\} _{\alpha\beta}\\
 & = & f_{1,\beta+1}^{\leftarrow}q_{m}\left\{ q_{k}^{\rightarrow},f_{1,\beta+1}\right\} _{\alpha\beta}+f_{1,\beta+1}^{\leftarrow}f_{1,\beta+1}\left\{ q_{k}^{\rightarrow},q_{m}\right\} _{\alpha\beta}\\
 &  & +q_{k}^{\rightarrow}q_{m}\left\{ f_{1,\beta+1}^{\leftarrow},f_{1,\beta+1}\right\} _{\alpha\beta}+q_{k}^{\rightarrow}f_{1,\beta+1}\left\{ f_{1,\beta+1}^{\leftarrow},q_{m}\right\} _{\alpha\beta}\\
 & = & f_{1,\beta+1}^{\leftarrow}q_{m}\left\{ q_{k}^{\rightarrow},f_{1,\beta+1}\right\} _{std}+f_{1,\beta+1}^{\leftarrow}f_{1,\beta+1}\left\{ q_{k}^{\rightarrow},q_{m}\right\} _{std}\\
 &  & +q_{k}^{\rightarrow}q_{m}\left\{ f_{1,\beta+1}^{\leftarrow},f_{1,\beta+1}\right\} _{std}+q_{k}^{\rightarrow}f_{1,\beta+1}\left\{ f_{1,\beta+1}^{\leftarrow},q_{m}\right\} _{std},
\end{eqnarray*}
 and with \eqref{eq:brckfktf1+bt} and Lemma \ref{lem:SumCfeq0} this
is 
\begin{eqnarray}
\left\{ q_{k}^{\rightarrow}\cdot f_{1,\beta+1}^{\leftarrow},q_{m}\cdot f_{1,\beta+1}\right\} _{\alpha\beta} & = & \omega_{2}q_{k}^{\rightarrow}f_{1,\beta+1}^{\leftarrow}q_{m}f_{1,\beta+1}\label{eq:BrckscndPrt}\\
 &  & -q_{k}f_{1,\beta+1}f_{1,\beta+1}^{\leftarrow}q_{m}^{\rightarrow},\nonumber 
\end{eqnarray}
 where 
\[
\omega_{2}=\omega_{q_{k},q_{m}}+\omega_{q_{k},f_{1,\beta+1}}+\omega_{f_{1,\beta+1},q_{m}}+1.
\]
 The third one is 
\begin{eqnarray*}
 &  & \left\{ q_{k}\cdot f_{1,\beta+1},q_{m}^{\rightarrow}\cdot f_{1,\beta+1}^{\leftarrow}\right\} _{\alpha\beta}\\
 & = & q_{k}q_{m}^{\rightarrow}\left\{ f_{1,\beta+1},f_{1,\beta+1}^{\leftarrow}\right\} _{\alpha\beta}+q_{k}f_{1,\beta+1}^{\leftarrow}\left\{ f_{1,\beta+1},q_{m}^{\rightarrow}\right\} _{\alpha\beta}\\
 &  & +f_{1,\beta+1}q_{m}^{\rightarrow}\left\{ q_{k},f_{1,\beta+1}^{\leftarrow}\right\} _{\alpha\beta}+f_{1,\beta+1}f_{1,\beta+1}^{\leftarrow}\left\{ q_{k},q_{m}^{\rightarrow}\right\} _{\alpha\beta}\\
 & = & q_{k}q_{m}^{\rightarrow}\left\{ f_{1,\beta+1},f_{1,\beta+1}^{\leftarrow}\right\} _{std}+q_{k}f_{1,\beta+1}^{\leftarrow}\left\{ f_{1,\beta+1},q_{m}^{\rightarrow}\right\} _{std}\\
 &  & +f_{1,\beta+1}q_{m}^{\rightarrow}\left\{ q_{k},f_{1,\beta+1}^{\leftarrow}\right\} _{std}+f_{1,\beta+1}f_{1,\beta+1}^{\leftarrow}\left\{ q_{k},q_{m}^{\rightarrow}\right\} _{std}
\end{eqnarray*}
 and with Lemma \ref{lem:SumCfeq0} it makes 
\begin{equation}
\left\{ q_{k}\cdot f_{1,\beta+1},q_{m}^{\rightarrow}\cdot f_{1,\beta+1}^{\leftarrow}\right\} _{\alpha\beta}=\omega_{3}q_{k}f_{1,\beta+1}q_{m}^{\rightarrow}f_{1,\beta+1}^{\leftarrow},
\end{equation}
 with 
\[
\omega_{3}=\omega_{q_{k},q_{m}}+\omega_{q_{k},f_{1,\beta+1}}+\omega_{f_{1,\beta+1},q_{m}}.
\]
 The last bracket is 
\begin{eqnarray*}
 &  & \left\{ q_{k}^{\rightarrow}\cdot f_{1,\beta+1}^{\leftarrow},f_{j}^{\rightarrow}\cdot f_{1,\beta+1}^{\leftarrow}\right\} _{\alpha\beta}\\
 & = & q_{k}^{\rightarrow}f_{1,\beta+1}^{\leftarrow}\left\{ f_{1,\beta+1}^{\leftarrow},f_{j}^{\rightarrow}\right\} _{\alpha\beta}+f_{1,\beta+1}^{\leftarrow}f_{j}^{\rightarrow}\left\{ q_{k}^{\rightarrow},f_{1,\beta+1}^{\leftarrow}\right\} _{\alpha\beta}\\
 &  & +f_{1,\beta+1}^{\leftarrow}f_{1,\beta+1}^{\leftarrow}\left\{ q_{k}^{\rightarrow},f_{j}^{\rightarrow}\right\} _{\alpha\beta}\\
 & = & q_{k}^{\rightarrow}f_{1,\beta+1}^{\leftarrow}\left\{ f_{1,\beta+1}^{\leftarrow},f_{j}^{\rightarrow}\right\} _{std}+f_{1,\beta+1}^{\leftarrow}f_{j}^{\rightarrow}\left\{ q_{k}^{\rightarrow},f_{1,\beta+1}^{\leftarrow}\right\} _{std}\\
 &  & +f_{1,\beta+1}^{\leftarrow}f_{1,\beta+1}^{\leftarrow}\left\{ q_{k}^{\rightarrow},f_{j}^{\rightarrow}\right\} _{std}
\end{eqnarray*}
 and again, Lemma \ref{lem:SumCfeq0} turns it to 
\begin{equation}
\left\{ q_{k}^{\rightarrow}\cdot f_{1,\beta+1}^{\leftarrow},f_{j}^{\rightarrow}\cdot f_{1,\beta+1}^{\leftarrow}\right\} _{\alpha\beta}=\omega_{4}q_{k}^{\rightarrow}\cdot f_{1,\beta+1}^{\leftarrow},f_{j}^{\rightarrow}\cdot f_{1,\beta+1}^{\leftarrow},\label{eq:BrkctFrthprt}
\end{equation}
 with 
\[
\omega_{4}=\omega_{q_{k},q_{m}}+\omega_{q_{k},f_{1,\beta+1}}+\omega_{f_{1,\beta+1},q_{m}}.
\]
 Summing \eqref{eq:Brckfirstprt}--\eqref{eq:BrkctFrthprt} proves
that 
\[
\left\{ \theta_{k},\theta_{m}\right\} _{\alpha\beta}=\left(\omega_{q_{k},q_{m}}+\omega_{q_{k},f_{1,\beta+1}}+\omega_{f_{1,\beta+1},q_{m}}\right)\theta_{k}\theta_{m}.
\]
 Last, we check that every pair $\theta_{k},\psi_{m}$ is log canonical
w.r.t. $\left\{ \cdot,\cdot\right\} _{\alpha\beta}$. The process
is pretty much like the one for $\theta_{k}$ and $\theta_{m}$: break
the two functions into their components, 
\[
\theta_{k}=f_{n+k-\alpha,k}f_{1,\beta+1}-f_{n+k-\alpha,k}^{\rightarrow}f_{1,\beta+1}^{\leftarrow}
\]
 and 
\[
\psi_{m}=f_{m,n+m-\beta}f_{\alpha+1,1}-f_{m,n+m-\beta}^{\downarrow}f_{\alpha+1,1}^{\uparrow}.
\]
 Then compute all brackets of these components. Most of these brackets
can be computed as above, but here Lemma \ref{lem:ArFuncLC} will
be needed as well. Setting 
\begin{eqnarray*}
\omega_{\theta_{k},\psi_{m}} & = & \omega_{f_{n+k-\alpha,k},f_{m,n+m-\beta}}+\omega_{f_{n+k-\alpha,k},f_{\alpha+1,1}}\\
 &  & +\omega_{f_{1,\beta+1},f_{m,n+m-\beta}}+\omega_{f_{1,\beta+1},f_{\alpha+1,1},}
\end{eqnarray*}
 the result is 
\[
\left\{ \theta_{k},\psi_{m}\right\} _{\alpha\beta}=\omega_{\theta_{k},\psi_{m}}\theta_{k}\psi_{m}.
\]
 The other possible combinations are symmetric (e.g., $\left\{ \psi_{k},\psi_{m}\right\} $
is symmetric to $\left\{ \theta_{k},\theta_{m}\right\} $). Lemmas
\ref{lem:NewFuncLCandCfSym} and \ref{lem:SpeFuncBrcktSym} can be
used instead of \ref{lem:NewFuncLCandCf} and \ref{lem:SpeFuncBrckt},
respectively. 
\end{proof}

\section{The cluster structure $\mathcal{C}_{\alpha\beta}$ \label{sec:The-cluster-structure}}

\subsection{Stable variables}

Recall the definition of $\mathcal{B}^{D}={\varphi_{ij}^{D}\left(X,Y\right)|i,j\in\left[n\right]}$
from Section \ref{sub:Constructing-a-log}. Look at the set $S=\left\{ \varphi_{i1}^{D},\varphi_{1j}^{D}|i\neq\alpha+1,j\neq\beta+1\right\} $
and let $\tilde{S}$ be the restrictions of these functions to the
diagonal subgroup.

Though the following proposition is not required for the proof of
the main theorem, it does give further information about the cluster
structure: the set $\tilde{S}$ will be the set of stable variables.
As indicated in \cite{gekhtman2013exotic}, in all known cluster structures
on Poisson varieties, the frozen variables have two important properties:
they behave well under certain natural group actions, and they are
log canonical with certain globally defined coordinate functions.
Proposition (\ref{prop:StblVars}) states that these two properties
hold in our case, and therefore supports the choice of $\tilde{S}$
as the set of stable variables. 
\begin{prop}
\label{prop:StblVars}1. The elements of $S$ are semi-invariants
of the left and right action of $D_{-}$ in $D\left(GL_{n}\right)$. 

2. The elements of $\tilde{S}$ are log canonical with all matrix
entries $x_{ij}$.\end{prop}
\begin{proof}
1. The subgroup $D_{-}$ of $D\left(GL_{n}\right)$ that corresponds
to the subalgebra $\mathfrak{g}_{-}$ of $\mathfrak{g}$ is given
by 
\[
D_{-}=\left(U,L\right)
\]
 with 
\[
U=\left[\begin{array}{cccccc}
a_{1} & \star & \star & \star & \star & \star\\
0 & \ddots & \star & \star & \star & \star\\
0 & 0 & a_{\alpha-1} & \star & \star & \star\\
0 & 0 & 0 & A & \star & \star\\
0 & 0 & 0 & 0 & a_{\alpha+2} & \star\\
0 & 0 & 0 & 0 & 0 & \ddots
\end{array}\right]
\]
 and 
\[
L=\left[\begin{array}{ccccccc}
a_{n+\alpha-\beta+1} & 0 & 0 & 0 & 0 & 0 & 0\\
\star & \ddots & 0 & 0 & 0 & 0 & 0\\
\star & \star & a_{n} & 0 & 0 & 0 & 0\\
\star & \star & \star & \ddots & 0 & 0 & 0\\
\star & \star & \star & \star & a_{\alpha-1} & 0 & 0\\
\star & \star & \star & \star & \star & A & 0\\
\star & \star & \star & \star & \star & \star & \ddots
\end{array}\right].
\]
 where $A\in GL_{2}$ and the indices of the diagonal entries $a_{i}$
are taken modulo $n$. The $\star$'s will not play any role in further
computations. The left and right action of $D_{-}$ can be parametrized
by 
\[
\left(X,Y\right)\mapsto\left(A_{1}XA'_{1},A_{2}YA'_{2}\right)
\]
 with matrices 
\[
A_{1}=\left[\begin{array}{cccccc}
a_{1} & \star & \star & \star & \star & \star\\
0 & \ddots & \star & \star & \star & \star\\
0 & 0 & a_{\alpha-1} & \star & \star & \star\\
0 & 0 & 0 & A & \star & \star\\
0 & 0 & 0 & 0 & a_{\alpha+2} & \star\\
0 & 0 & 0 & 0 & 0 & \ddots
\end{array}\right],
\]
 
\[
A'_{1}=\left[\begin{array}{cccccc}
a'_{1} & \star & \star & \star & \star & \star\\
0 & \ddots & \star & \star & \star & \star\\
0 & 0 & a'_{\alpha-1} & \star & \star & \star\\
0 & 0 & 0 & A' & \star & \star\\
0 & 0 & 0 & 0 & a'_{\alpha+2} & \star\\
0 & 0 & 0 & 0 & 0 & \ddots
\end{array}\right],
\]
 
\[
A_{2}=\left[\begin{array}{ccccccc}
a_{n+\alpha-\beta+1} & 0 & 0 & 0 & 0 & 0 & 0\\
\star & \ddots & 0 & 0 & 0 & 0 & 0\\
\star & \star & a_{n} & 0 & 0 & 0 & 0\\
\star & \star & \star & \ddots & 0 & 0 & 0\\
\star & \star & \star & \star & a_{\alpha-1} & 0 & 0\\
\star & \star & \star & \star & \star & A & 0\\
\star & \star & \star & \star & \star & \star & \ddots
\end{array}\right],
\]
 and 
\[
A'_{2}=\left[\begin{array}{ccccccc}
a'_{n+\alpha-\beta+1} & 0 & 0 & 0 & 0 & 0 & 0\\
\star & \ddots & 0 & 0 & 0 & 0 & 0\\
\star & \star & a'_{n} & 0 & 0 & 0 & 0\\
\star & \star & \star & \ddots & 0 & 0 & 0\\
\star & \star & \star & \star & a'_{\alpha-1} & 0 & 0\\
\star & \star & \star & \star & \star & A' & 0\\
\star & \star & \star & \star & \star & \star & \ddots
\end{array}\right].
\]
 There are three kinds of functions in $S$: minors of $X$, minors
of $Y$ and ``mixed'' functions. A function $f_{X}\in S$ that is
a minor of $X$ is a semi-invariant of this action: $f_{X}$ is the
determinant of a submatrix $X_{\left[i,n\right]}^{\left[1,\mu\right]}$
with $\mu=n-i+1$. One has $i\notin\left\{ \alpha+1,n-\alpha+1\right\} $
(see the construction in Section \ref{sub:Constructing-a-log}). The
action of $D_{-}$ multiplies each row $k\in[i,n]$ of $X$ by the
corresponding entry $a_{k}$ and each column $\ell\in[\mu]$ of $X$
by $a'_{\ell+n+\alpha-\beta-1}$, with two exceptions: rows $\alpha$
and $\alpha+1$ are multiplied together by $A$, and columns $\alpha$
and $\alpha+1$ are multiplied by $A'$. So as long as one of these
rows (or columns) does not occur in the submatrix $X_{\left[i,n\right]}^{\left[1,\mu\right]}$
without the other, $f_{X}$ is still a semi-invariant of the action.
If $\alpha\in[i,n]$ then clearly $\alpha+1\in[i,n].$ On the other
hand, the only case with $\alpha+1\in[i,n]$ and $\alpha\notin[i,n]$
is when $i=\alpha+1.$ But such a minor can not be in $S$ according
to the construction (Section \ref{sub:Constructing-a-log}). Looking
at columns, it is easy to see that if $\alpha+1\in[1,\mu]$ (that
is, the column $\alpha+1$ occurs in the submatrix $X_{\left[i,n\right]}^{\left[1,\mu\right]}$),
then $\alpha\in[1,\mu]$. The only way to have $\alpha\in[1,\mu]$
and $\alpha+1\notin[1,\mu]$ is $\mu=\alpha$. But this implies $i=n-\alpha+1$,
and this minor is also not in the set $S$.

For $f_{Y}\in S$ which is a minor of $Y$ similar arguments hold. 

Look now at the function 
\[
\theta=\det\left[\begin{array}{cccccc}
x_{i1} & \cdots & x_{i\alpha} & x_{i,\alpha+1} & 0 & \cdots\\
\vdots & \ddots & \vdots & \vdots & \vdots\\
x_{n1} & \cdots & x_{n\alpha} & x_{n,\alpha+1} & 0 & \cdots\\
0 & \cdots & y_{1\beta} & y_{1\beta+1} & \cdots & y_{1n}\\
\vdots &  & \vdots &  & \ddots & \vdots\\
0 & \cdots & y_{\mu\beta} & \cdots & \cdots & y_{\mu n}
\end{array}\right],
\]
 with $i=n+1-\alpha$. It is not hard to see that $\theta$ is a semi
invariant of the action of $D_{-}$: the block of $x_{ij}$'s is subject
to the same arguments as above, except for when $\alpha=\frac{n}{2}$,
which will be treated later. The same holds for the block of $y_{ij}$'s,
unless $\beta=\frac{n}{2}$. Therefore $\theta$ is a semi-invariant
of this action. 

Symmetric arguments show that 
\[
\psi=\det\left[\begin{array}{cccccc}
y_{1,n+1-\beta} & \cdots & y_{1n} & 0 & \cdots & 0\\
\vdots & \ddots & \vdots & \vdots &  & \vdots\\
y_{\beta j} & \cdots & y_{\beta n} & x_{\alpha1} & \cdots & x_{\alpha\mu}\\
y_{\beta+1,j} & \cdots & y_{\beta+1,n} & x_{\alpha+1,1} &  & \vdots\\
0 & \cdots & 0 & \vdots & \ddots & \vdots\\
\vdots &  & \vdots & x_{n1} & \cdots & x_{n\mu}
\end{array}\right]
\]
 is a also semi-invariant.

Last, we look at the special case $\alpha=\frac{n}{2}$: here there
is only one matrix in $\mathcal{M}$ with elements of both $X$ and
$Y$. The ``building blocks'' of this matrix are submatrices of
$X$ and $Y$ that satisfy the restrictions above, so the determinant
of this matrix is also a semi-invariant of the action. The case $\beta=\frac{n}{2}$
is symmetric. 

2. First, look at a function $\varphi\in\tilde{S}\cap\mathcal{B}_{std}$.
In this case it is not hard to see that $\left\{ \varphi,x_{ij}\right\} _{\alpha\beta}=\left\{ \varphi,x_{ij}\right\} _{std}$
(according to Lemma \ref{lem:PsnBrcktDiff}) and therefore $\varphi$
is log canonical with all $x_{ij}$. There are only two other functions
in $\tilde{S}$: $\theta=\varphi_{n-\alpha+1,1}$ and $\psi=\varphi_{1,n-\beta+1}$.
Start with $\theta=f_{n-\alpha+1,1}f_{1,\beta+1}-f_{n-\alpha+1,1}^{\rightarrow}f_{1,\beta+1}^{\leftarrow}$.
Following the line of the proof of Lemma \ref{lem:SumCfeq0}, it can
be shown that $f_{n-\alpha+1,1}^{\rightarrow}$ is log canonical with
every $x_{ij}$ with $j\neq\alpha$, and similarly $f_{1,\beta+1}^{\leftarrow}$
is log canonical with every $x_{ij}$ with $j\neq\beta+1$, with respect
to the standard bracket. The cases $j=\alpha$ and $j=\beta+1$ are
exactly the cases when the standard bracket and the $\alpha\beta$
bracket do not coincide. Adding the difference that was described
in Lemma \ref{lem:PsnBrcktDiff}, 
\[
\left\{ f_{n-\alpha+1,1},x_{ij}\right\} _{\alpha\beta}=\begin{cases}
\left\{ f_{n-\alpha+1,1},x_{ij}\right\} _{std} & \text{ if }j\neq\beta+1\\
\left\{ f_{n-\alpha+1,1},x_{ij}\right\} _{std}+f_{n-\alpha+1,1}^{\rightarrow}x_{i\beta} & \text{ if }j=\beta+1,
\end{cases}
\]
 and 
\[
\left\{ f_{1,\beta+1},x_{ij}\right\} _{\alpha\beta}=\begin{cases}
\left\{ f_{1,\beta+1},x_{ij}\right\} _{std} & \text{ if }j\neq\alpha\\
\left\{ f_{1,\beta+1},x_{ij}\right\} _{std}+f_{1,\beta+1}^{\leftarrow}x_{i,\alpha+1} & \text{ if }j=\alpha,
\end{cases}
\]
 shows that with respect to the bracket $\left\{ \cdot,\cdot\right\} _{\alpha\beta}$,
the pairs $f_{n-\alpha+1,1}^{\rightarrow},x_{i\alpha}$ and $f_{1,\beta+1}^{\leftarrow},x_{i,\beta+1}$
are log canonical. The coefficients $\omega_{f_{n-\alpha+1,1}^{\rightarrow},x_{ij}}=\frac{\left\{ f_{n-\alpha+1,1}^{\rightarrow},x_{ij}\right\} }{f_{n-\alpha+1,1}^{\rightarrow}\cdot x_{ij}}$
and $\omega_{f_{1,\beta+1}^{\leftarrow},x_{ij}}=\frac{\left\{ f_{1,\beta+1}^{\leftarrow},x_{ij}\right\} }{f_{1,\beta+1}^{\leftarrow}\cdot x_{ij}}$
can be computed like in the proof of Lemma \ref{lem:SumCfeq0}, showing
that $\theta$ is log canonical with $x_{ij}$. Symmetric arguments
hold for $\psi$.
\end{proof}
Note that the set of stable variables $\tilde{S}$ is the set of determinants
of all matrices in the set $\mathcal{M}_{\alpha\beta}$. By its definition
in \ref{eq:Mabdef}, $\mathcal{M}_{\alpha\beta}$ has  $2\left(n-1\right)-2$
matrices, and therefore 
\[
\left|\tilde{S}\right|=2\left(n-1\right)-2=2\left|\Delta\setminus\Gamma_{1}\right|,
\]
 as in Statement \ref{Conj:NumStbVar} of Conjecture \ref{Conj:GSV-BD-CS}.

\subsection{The quiver $Q_{\alpha\beta}^{n}$}

\label{sub:The-quiver}

To describe the quiver $Q_{\alpha\beta}^{n},$ start with the standard
quiver $Q_{std}^{n}$ as given in Section \ref{sub:BisLC}. The vertex
in the $i$-th row and $j$-th column corresponds to the cluster variable
$f_{ij}=\det X_{\left[i,\mu\left(i,j\right)\right]}^{\left[j,\mu\left(i,j\right)\right]}.$
The quiver of the exotic cluster structure with BD data $\alpha\mapsto\beta$
is very close: a vertex $\left(i,j\right)$ now represents the cluster
variable $\varphi_{ij}$, and the quiver takes these changes: 
\begin{enumerate}
\item Vertices $\left(\alpha+1,1\right)$ and $\left(1,\beta+1\right)$
are not frozen. 
\item The arrows $\left(\alpha,1\right)\to\left(\alpha+1,1\right)$ and
$\left(1,\beta\right)\to\left(1,\beta+1\right)$ are added. 
\item The arrows $\left(n,\alpha+1\right)\to\left(1,\beta+1\right),\ \left(1,\beta+1\right)\to\left(n,\alpha\right)$
are added. 
\item The arrows $\left(\beta+1,n\right)\to\left(\alpha+1,1\right),\ \left(\alpha+1,1\right)\to\left(\beta,n\right)$
are added. 
\end{enumerate}
The example of $1\mapsto2$ on $SL_{5}$ is given in Figure \ref{fig:GL51to2}.
The dashed arrows are the arrows that were added to the standard quiver.

\begin{figure}
\begin{centering}
\includegraphics[scale=0.45]{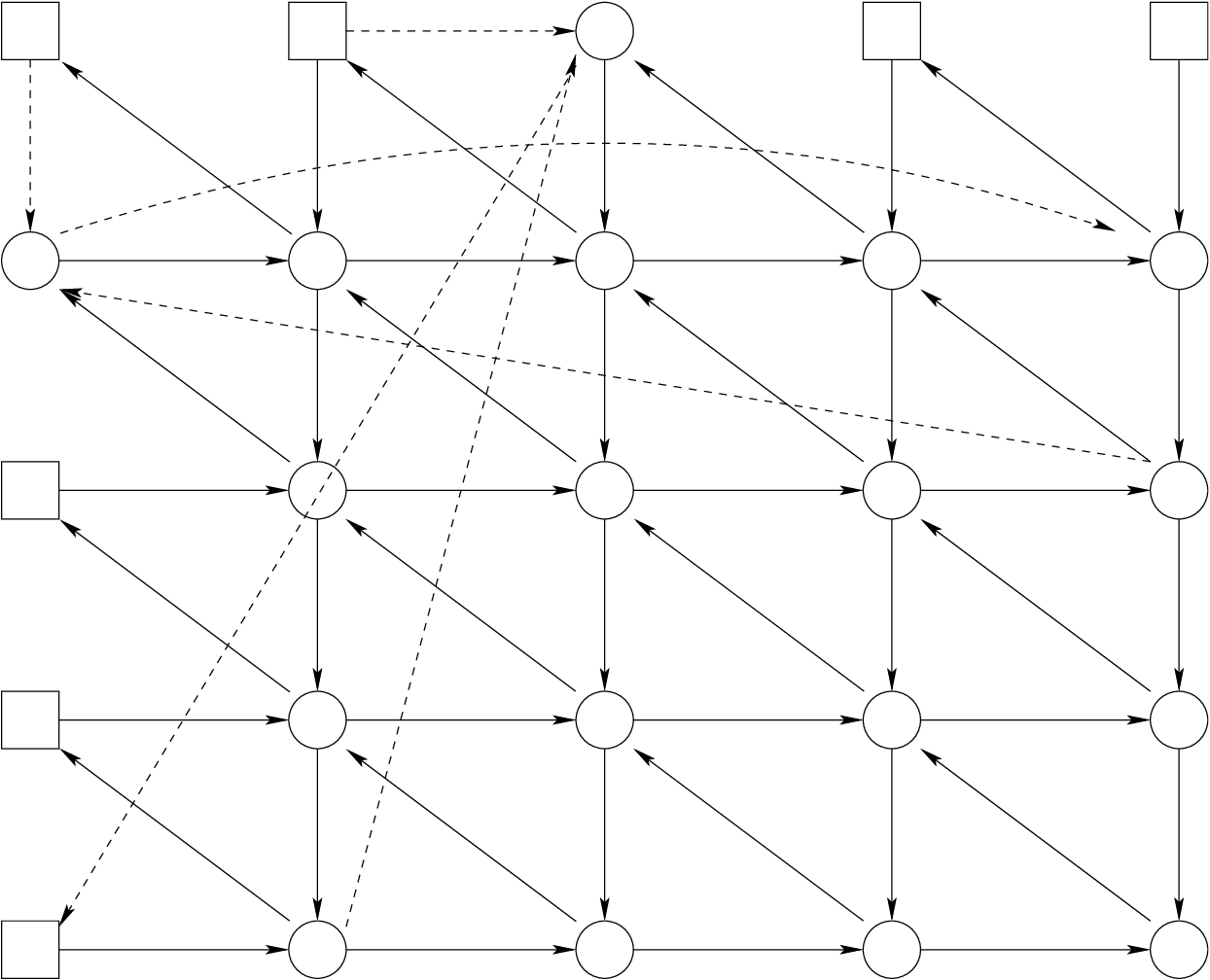} 
\par\end{centering}

\caption{The $1\protect\mapsto2$ quiver for $GL_{5}$}

\centering{}\label{fig:GL51to2} 
\end{figure}

Through this section we use $B,\omega,\Omega$ to denote the exchange
matrix, Poisson coefficients and Poisson matrix in the standard case,
and $\overline{B},\overline{\Omega},\overline{\omega}$ for their
counterparts in the $\alpha\mapsto\beta$ case. We now prove that
the cluster structure described in the previous section is indeed
compatible with the bracket $\left\{ \cdot,\cdot\right\} _{\alpha\beta}$:
\begin{thm}
\label{thm:Compatible}The cluster structure $\mathcal{C}_{\alpha\beta}=\mathcal{C}\left(\overline{B}_{\alpha\beta}\right)$
is compatible with the bracket $\left\{ \cdot,\cdot\right\} _{\alpha\beta}$. \end{thm}
\begin{proof}
According to Proposition \ref{prop:PoissCompStruc}, it is sufficient
to prove 
\begin{equation}
\bar{B}\bar{\Omega}=\left[I\ 0\right],\label{eq:Compatibility}
\end{equation}
 which can be rephrased as 
\[
\left(\bar{B}\bar{\Omega}\right)_{pq}=\sum_{p\leftarrow k}\overline{\omega}_{kq}-\sum_{p\rightarrow k}\overline{\omega}_{kq}=\begin{cases}
1 & p=q\\
0 & p\neq q,
\end{cases}
\]
 where the first sum is over vertices $k$ with an arrow pointing
from $k$ to $p$, and the second sum is over vertices $k$ with an
arrow pointing from $p$ to $k$. Recall that the standard case has
\[
\sum_{p\leftarrow k}\omega_{kq}-\sum_{p\rightarrow k}\omega_{kq}=\begin{cases}
1 & p=q\\
0 & p\neq q.
\end{cases}
\]
 Label a row (or column) of the exchange matrices $B$ and $\overline{B}$
by $\left(i,j\right)$ if it corresponds to the cluster variable $f_{ij}$
(in $B$) or $\varphi_{ij}$ (in $\overline{B}$). Now compute $\left(\bar{B}\bar{\Omega}\right)_{pq}$
in the following cases:

1. The $p$-th row of $\overline{B}$ is equal to the $p$-th row
of $B$. This is true for almost all rows of $\overline{B}$, or more
precisely when 
\begin{equation}
p\notin\left\{ \left(1,\beta+1\right),\left(n,\alpha\right),\left(n,\alpha+1\right),\left(\alpha+1,1\right),\left(\beta,n\right),\left(\beta+1,n\right)\right\} ,\label{eq:SetNewRows}
\end{equation}
and we have these possible situations:

(a) $p$ corresponds to a cluster variable in $\mathcal{B}_{std}\cap\mathcal{B}_{\alpha\beta}$.

i. Assume all cluster variables adjacent to $p$ are in $\mathcal{B}_{std}\cap\mathcal{B}_{\alpha\beta}$.
If $q$ is also in $\mathcal{B}_{std}\cap\mathcal{B}_{\alpha\beta}$,
this is just the same as the standard case, and $\left(\bar{B}\bar{\Omega}\right)_{pq}=\left(B\Omega\right)_{pq}.$
If $q$ is not a standard basis function, then either it corresponds
to a cluster variable of the form $\varphi_{n+m-\alpha,m}$, and then
\[
\overline{\omega}_{kq}=\omega_{kq}+\omega_{k,f_{n,\alpha+1}}-\omega_{k,f_{n\alpha}}
\]
 or one of the form $\varphi_{m,n+m-\beta}$, and then 
\[
\overline{\omega}_{kq}=\omega_{kq}+\omega_{k,f_{\beta+1,n}}-\omega_{k,f_{\beta n}}.
\]
So in the first case, 
\begin{eqnarray*}
\sum_{p\leftarrow k}\overline{\omega}_{kq}-\sum_{p\rightarrow k}\overline{\omega}_{kq} & = & \sum_{p\leftarrow k}\omega_{kq}-\sum_{p\rightarrow k}\omega_{kq}+\sum_{p\leftarrow k}\omega_{k,f_{n,\alpha+1}}\\
 &  & -\sum_{p\rightarrow k}\omega_{k,f_{n,\alpha+1}}-\sum_{p\leftarrow k}\omega_{k,f_{n\alpha}}+\sum_{p\rightarrow k}\omega_{k,f_{n\alpha}}
\end{eqnarray*}
and since $p\neq\left(n,\alpha\right),\left(n,\alpha+1\right)$ (from
\eqref{eq:SetNewRows}), the standard case tells us 
\[
\sum_{p\leftarrow k}\omega_{k,f_{n,\alpha+1}}-\sum_{p\rightarrow k}\omega_{k,f_{n,\alpha+1}}=\sum_{p\leftarrow k}\omega_{k,f_{n\alpha}}-\sum_{p\rightarrow k}\omega_{k,f_{n\alpha}}=0.
\]
 The case of $\varphi_{m,n+m-\beta}$ is symmetric. 

ii. The cluster variable $p$ has at least one neighbor that is not
in $\mathcal{B}_{std}\cap\mathcal{B}_{\alpha\beta}$. 

Looking at the quiver one can easily see that the number of such neighbors
can be either one or two.

A. $p$ has exactly one such neighbor. The quiver has only two such
vertices: $p=\left(n,\alpha+1\right)$ or $p=\left(\beta+1,n\right)$.
In both cases it means that the $i$-th row of $\overline{B}$ is
different from that row of $B$, because the quiver $Q_{\alpha\beta}^{n}$
has arrows $\left(n,\alpha+1\right)\to\left(1,\beta+1\right)$ and
$\left(\beta+1,n\right)\to\left(\alpha+1,1\right)$, which $Q_{std}^{n}$
does not have. This case will be handled later on.

B. $p$ has two non standard neighbors. In this case these two neighbors
are connected to $p$ by arrows in opposite directions (i.e., one
of them is pointing at $p$ and the other one from $p$). These two
non standard neighbors must both belong to the same ``family'' of
functions, either $\left\{ \psi_{m}\right\} $ or $\left\{ \theta_{m}\right\} $
(as defined in Section \ref{sub:BisLC}). We have seen that the Poisson
coefficients of these function differ from their standard counterparts
by a constant, e.g., for every function $g\in\mathcal{B}_{std}$,
\[
\omega_{\varphi_{n+m-\alpha,m},g}=\omega_{f_{n+m-\alpha,m},g}+\omega_{f_{1,\beta+1},g}.
\]
When summing over all neighbors of $p$, this constant is then added
once, for the vertex with an arrow pointing at $p$, and subtracted
once, for the vertex with an arrow pointing from $p$ to it. These
cancel each other and the sum remains as it was in the standard case. 

(b) $p$ is not in $\mathcal{B}_{std}\cap\mathcal{B}_{\alpha\beta}$,
which means $p=\left(n+m-\alpha,m\right)$ or $p=\left(m,n+m-\beta\right)$.
Assume $m<\alpha$ (for the first one) or $m<\beta$ (second), because
$p=\left(n,\alpha\right)$ and $p=\left(\beta,n\right)$ are in \eqref{eq:SetNewRows}
and will be treated later. If $m=1$ it is a frozen variable. Again,
look at the first case (second is just the same): if $1<m<\alpha$
then two neighbors of $p=\left(n+m-\alpha,m\right)$ are non standard.
These are $\left(n+m+1-\alpha,m+1\right)$ and $\left(n+m-1-\alpha,m-1\right)$
with arrows pointing in opposite directions. Since we know that 
\[
\overline{\omega}_{f_{n+m-\alpha,m},q}=\omega_{f_{n+m-\alpha,m},q}+\omega_{f_{1,\beta+1},q}
\]
and therefore a constant is added to the sum for the vertex
$\left(n+m+1-\alpha,m+1\right)$ and then subtracted for the vertex
$\left(n+m-1-\alpha,m-1\right)$. This constant is added to all $\omega$'s
in the sum, and they cancel each other.

2. Here the $p$-th row of $\overline{B}$ is not equal to the $p$-th
row of $B$. 

(a) If $p=\left(1,\beta+1\right)$ then $B$ does not have this row
(it was a frozen variable in the standard case). Its neighbors are
now $\varphi_{n,\alpha+1},\varphi_{2,\beta+2},\varphi_{1\beta}$ with
arrows pointing to $p$, and $\varphi_{2,\beta+1},\varphi_{n\alpha}$
with arrows from $p$ to them . So we have
\begin{eqnarray}
\sum_{p\leftarrow k}\overline{\omega}_{kq}-\sum_{p\rightarrow k}\overline{\omega}_{kq} & = & +\overline{\omega}_{\varphi_{1\beta},\varphi_{q}}+\overline{\omega}_{\varphi_{n,\alpha+1},\varphi_{q}}+\overline{\omega}_{\varphi_{2,\beta+2},\varphi_{q}}\label{eq:CfSumat1b+1}\\
 &  & -\overline{\omega}_{\varphi_{2,\beta+1},\varphi_{q}}+\overline{\omega}_{\varphi_{n\alpha},\varphi_{q}}\nonumber \\
 & = & \omega_{f_{1\beta},f_{q}}+\omega_{f_{n,\alpha+1},f_{q}}+\omega_{f_{2,\beta+2},f_{q}}\nonumber \\
 &  & -\omega_{f_{2,\beta+1},f_{q}}-\omega_{f_{n\alpha},f_{q}}-\omega_{f_{1,\beta+1},f_{q}}\nonumber 
\end{eqnarray}
 In the standard case, since exchange relation must hold at $\left(2,\beta+1\right)$
we can write 
\begin{equation}
\omega_{f_{1,\beta},f_{j}}+\omega_{f_{2,\beta+2},f_{j}}+\omega_{f_{3,\beta+1},f_{j}}-\omega_{f_{1,\beta+1},f_{j}}-\omega_{f_{2,\beta},f_{j}}-\omega_{f_{3,\beta+2},f_{j}}=\begin{cases}
1 & j=\left(2,\beta+1\right)\\
0 & j\neq\left(2,\beta+1\right)
\end{cases}\label{eq:ExReat2b+1}
\end{equation}
and we continue, using standard exchange relation at $\left(i,\beta+1\right)$
\[
\omega_{f_{i-1,\beta},f_{j}}+\omega_{f_{i,\beta+2},f_{j}}+\omega_{f_{i+1,\beta+1},f_{j}}-\omega_{f_{i-1,\beta+1},f_{j}}-\omega_{f_{i,\beta},f_{j}}-\omega_{f_{i+1,\beta+2},f_{j}}=\begin{cases}
1 & j=\left(i,\beta+1\right)\\
0 & j\neq\left(i,\beta+1\right)
\end{cases}
\]
 and assuming $j\neq\left(i,\beta+1\right)$
\begin{equation}
\omega_{f_{i-1,\beta},f_{j}}-\omega_{f_{i-1,\beta+1},f_{j}}+\omega_{f_{i,\beta+2},f_{j}}=\omega_{f_{i,\beta},f_{j}}-\omega_{f_{i+1,\beta+1},f_{j}}+\omega_{f_{i+1,\beta+2},f_{j}}\label{eq:ExReatib+1}
\end{equation}
and eventually for $i=n$
\[
\omega_{f_{n,\beta},f_{j}}+\omega_{f_{n-1,\beta+1}f_{j}}-\omega_{f_{n-1,\beta},f_{j}}-\omega_{f_{n,\beta+2},f_{j}}=\begin{cases}
1 & j=\left(n,\beta+1\right)\\
0 & j\neq\left(n,\beta+1\right)
\end{cases}.
\]
 The standard exchange relation at $\left(n,\beta+1\right)$ implies
\[
\omega_{f_{n,\beta},f_{j}}+\omega_{f_{n-1,\beta+1},f_{j}}-\omega_{f_{n-1,\beta},f_{j}}-\omega_{f_{n,\beta+2},f_{j}}=0
\]
or 
\[
\omega_{f_{n,\beta},f_{j}}=\omega_{f_{n,\beta+2},f_{j}}-\omega_{f_{n-1,\beta+1},f_{j}}+\omega_{f_{n-1,\beta},f_{j}}.
\]
 So 
\[
\omega_{f_{n,\beta},f_{j}}-\omega_{f_{n,\beta+1},f_{j}}=\omega_{f_{n-2,\beta},f_{j}}+\omega_{f_{n,-1\beta+2},f_{j}}-\omega_{f_{n-1,\beta+1},f_{j}}-\omega_{f_{n-2,\beta+1},f_{j}},
\]
 and using \eqref{eq:ExReatib+1} recursively 
\begin{equation}
\omega_{f_{n,\beta},f_{j}}-\omega_{f_{n,\beta+1},f_{j}}=\omega_{f_{i-2,\beta},f_{j}}+\omega_{f_{i,-1\beta+2},f_{j}}-\omega_{f_{i-1,\beta+1},f_{j}}-\omega_{f_{i-2,\beta+1},f_{j}}.\label{eq:recExrltb+1}
\end{equation}
 Now we only need $\omega_{f_{n,\alpha+1},f_{j}}-\omega_{f_{n,\alpha},f_{j}}=\omega_{f_{n,\beta+1},f_{j}}-\omega_{f_{n,\beta},f_{j}}$.
This is true from Lemma \ref{lem:SumCfeq0} and the assumption $j\neq\left(i,\beta+1\right)$,
so \eqref{eq:CfSumat1b+1} turns to 
\[
\sum_{p\leftarrow k}\overline{\omega}_{kq}-\sum_{p\rightarrow k}\overline{\omega}_{kq}=0,\quad\forall q\neq\left(p,\beta+1\right).
\]
 If, on the other hand $q=\left(1,\beta+1\right)$, this still holds,
but Lemma \ref{lem:SumCfeq0} now says $\omega_{f_{n,\alpha+1},f_{j}}-\omega_{f_{n,\alpha},f_{j}}=\omega_{f_{n,\beta+1},f_{j}}-\omega_{f_{n,\beta},f_{j}}+1$,
so that 
\[
\sum_{p\leftarrow k}\overline{\omega}_{kq}-\sum_{p\rightarrow k}\overline{\omega}_{kq}=1.
\]
\\
Last, let $q=\left(p,\beta+1\right)$ with $p>1$. So in \eqref{eq:ExReatib+1}
we need to add $1$ to the right hand side. This $1$ is then added
to the sum of coefficients over neighbors of $\left(1,\beta+1\right)$,
but now Lemma \ref{lem:SumCfeq0} says $\omega_{f_{n,\alpha+1},f_{j}}-\omega_{f_{n,\alpha},f_{j}}=\omega_{f_{n,\beta+1},f_{j}}-\omega_{f_{n,\beta},f_{j}}+1$
, so again 
\[
\sum_{p\leftarrow k}\overline{\omega}_{kq}-\sum_{p\rightarrow k}\overline{\omega}_{kq}=0.
\]
 The special case $\beta=n-1$ is somewhat different, because here
vertices $\left(i,\beta+1\right)=\left(i,n\right)$ do not have neighbors
on the right. However, the same arguments still hold, and since the
exchange relations in the standard quiver are similar, the final conclusion
is identical.

(b) Let $p=\left(n,\alpha\right)$ then in the standard quiver its
neighbors were $f_{n,\alpha+1},f_{n-1,\alpha-1}$ (with arrows from
$p$ to them), and $f_{n,\alpha-1},f_{n-1,\alpha}$ (with arrows pointing
to $p$). In $Q_{\alpha\beta}^{n}$ an arrow is added from $\left(1,\beta+1\right)$
to $p$. We then have $\varphi_{n-1,\alpha-1}=f_{n-1,\alpha-1}\cdot f_{1,\beta+1}-f_{n-1,\alpha-1}^{\rightarrow}\cdot f_{1,\beta+1}^{\leftarrow}$.
So using $\overline{\omega}_{\varphi_{n-1,\alpha-1},g}=\omega_{f_{n-1,\alpha-1},g}+\omega_{f_{1,\beta+1},g}$
we get 
\begin{eqnarray*}
\sum_{p\leftarrow k}\overline{\omega}_{kq}-\sum_{p\rightarrow k}\overline{\omega}_{kq} & = & \overline{\omega}_{\varphi_{n,\alpha-1},\varphi_{q}}+\overline{\omega}_{\varphi_{n-1,\alpha},\varphi_{q}}+\overline{\omega}_{\varphi_{1,\beta+1},\varphi_{q}}\\
 &  & -\overline{\omega}_{\varphi_{n,\alpha+1},\varphi_{q}}-\overline{\omega}_{\varphi_{n-1,\alpha-1},\varphi_{q}}\\
 & = & \omega_{f_{n,\alpha-1},f_{q}}+\omega_{f_{n-1,\alpha},f_{q}}+\omega_{f_{1,\beta+1},f_{q}}\\
 &  & -\omega_{f_{n,\alpha+1},f_{q}}-\omega_{f_{n-1,\alpha-1},f_{q}}-\omega_{f_{1,\beta+1},f_{q}}\\
 & = & \sum_{p\leftarrow k}\omega_{k,f_{q}}-\sum_{p\rightarrow k}\omega_{k,f_{q}},
\end{eqnarray*}
and the last term is the one from the standard case, which equals
$\delta_{ij}$. 

(c) $p=\left(n,\alpha+1\right)$. In the standard quiver there are
arrows from $p$ to $\left(n,\alpha\right),\left(n-1,\alpha+1\right)$
and from $\left(n,\alpha+2\right),\left(n-1,\alpha\right)$ to $p$.
In $Q_{\alpha\beta}^{n}$ there is a new arrow $\left(n,\alpha+1\right)\to\left(1,\beta+1\right)$.
Again, 
\begin{eqnarray*}
\sum_{p\leftarrow k}\overline{\omega}_{kq}-\sum_{p\rightarrow k}\overline{\omega}_{kq} & = & \overline{\omega}_{f_{n,\alpha+2}}+\overline{\omega}_{f_{n-1,\alpha}}+\overline{\omega}_{f_{1,\beta+1}}\\
 &  & -\overline{\omega}_{\varphi_{n\alpha}}-\overline{\omega}_{f_{n-1,\alpha+1}}\\
 & = & \omega_{f_{n,\alpha+2}}+\omega_{f_{n-1,\alpha}}+\omega_{f_{1,\beta+1}}\\
 &  & -\omega_{f_{n\alpha}}-\omega_{f_{1,\beta+1}}-\omega_{f_{n-1,\alpha+1}}\\
 & = & \sum_{p\leftarrow k}\omega_{k,f_{n,\alpha+1}}-\sum_{p\rightarrow k}\omega_{k,f_{n,\alpha+1}}
\end{eqnarray*}
 which is also equal to the standard. 
\end{proof}
Note that an immediate corollary from Theorem \ref{thm:Compatible}
is that the exchange matrix $\overline{B}$ is of maximal rank, since
$\rank\left(\bar{B}\bar{\Omega}\right)\leq\min\left(\rank\overline{B},\rank\overline{\Omega}\right)$,
and \eqref{eq:Compatibility} implies that $\bar{B}\bar{\Omega}$
has maximal rank.

\section{Local regularity}

\label{sec:Regularity}

We now show that the initial seed $\mathcal{B}_{\alpha\beta}$ defined
in Section \ref{sub:Constructing-a-log} is locally regular. Clearly,
all the elements of $\mathcal{B}_{\alpha\beta}$ are regular functions,
so local regularity is equivalent to:
\begin{thm}
For every exchangeable variable $\varphi$ in the initial cluster,
the exchanged variable $\varphi'$ is a regular function.\end{thm}
\begin{proof}
We can use the similarity of the exchange quivers $Q_{\alpha\beta}^{n}$
and $Q_{std}^{n}$ . The exchange relation \eqref{eq:exrltn} involves
the cluster variable $\varphi$ and its neighbors in the exchange
quiver. 

Consider the following cases:

1. $\varphi$ is in $\mathcal{B}_{\alpha\beta}\cap\mathcal{B}_{std}$
and all its neighbors are also in $\mathcal{B}_{\alpha\beta}\cap\mathcal{B}_{std}$.
\\
This means the exchange rule is the same as in the standard case,
and therefore the exchanged cluster variable is equal to the one in
the standard case, which is regular.

2. $\varphi$ is in $\mathcal{B}_{\alpha\beta}\cap\mathcal{B}_{std}$,
but at least one of its neighbors is not in 
$\mathcal{B}_{\alpha\beta}\cap\mathcal{B}_{std}$.

(a) Two neighbors of $\varphi$ are not in 
$\mathcal{B}_{\alpha\beta}\cap\mathcal{B}_{std}$.\label{enu:Two-neighbors-nin}\\
The exchange rule is now $\varphi\cdot\varphi'=\varphi_{ij}\cdot p_{1}+\varphi_{i+1,j+1}\cdot p_{2}$
where $\varphi_{ij}$ and $\varphi_{i+1,j+1}$ are the two non standard
neighbors and $p_{1},p_{2}$ some monomials. Now recall that in this
case $\varphi_{ij}=f_{ij}h-\overline{f}_{ij}\overline{h}$ where $\overline{f}_{ij}$
is either $f_{ij}^{\rightarrow}$ or $f_{ij}^{\downarrow}$ , and
$\overline{h}$ is either $f_{1,\beta+1}^{\leftarrow}$ or $f_{\alpha+1,1}^{\uparrow}$,
respectively. The exchange rule is then 
\begin{eqnarray*}
\varphi\cdot\varphi' & = & \left(f_{ij}h-\overline{f}_{ij}\overline{h}\right)p_{1}+\left(f_{i+1,j+1}h-\overline{f}_{i+1,j+1}\overline{h}\right)p_{2}\\
 & = & h\left(f_{ij}p_{1}+f_{i+1,j+1}p_{2}\right)-\overline{h}\left(\overline{f}_{ij}p_{1}+\overline{f}_{i+1,j+1}p_{2}\right)
\end{eqnarray*}
 the first part is just the standard exchange rule multiplied by $h$.
The term in the second parenthesis can be regarded as a Desnanot--Jacobi
identity \eqref{eq:DesJacId}. It is equal to the standard one with
just one change: the last column (row) that was $\alpha$ ($\beta$)
in the standard case is now replaced by $\alpha+1$ ($\beta+1$).
It is not hard to conclude that the result is a product of $\varphi$
and some regular function, as it was in the standard case.

(b) Only one neighbor of $\varphi$ is not in $\mathcal{B}_{\alpha\beta}\cap\mathcal{B}_{std}$
. \label{enu:One-neighbor-nin}\\
There are only two such vertices: $\varphi_{n,\alpha+1}$ and $\varphi_{\beta+1,n}$
. The vertex $\varphi_{n,\alpha+1}=x_{n,\alpha+1}$ has neighbors
$\varphi_{n,\alpha},\varphi_{n-1,\alpha},\varphi_{n-1,\alpha+1},\varphi_{n,\alpha+2}$
and $\varphi_{1,\beta+1}$. Figure \ref{fig:localfna+1} shows the
relevant subquiver of $Q_{\alpha\beta}^{n}$. Recall that $\varphi_{n,\alpha}=x_{n,\alpha}f_{1,\beta+1}-x_{n,\alpha+1}f_{1,\beta+1}^{\leftarrow}$
, so 
\begin{eqnarray*}
\varphi_{n,\alpha+1}\cdot\varphi'_{n,\alpha+1} & = & \varphi_{n,\alpha}\varphi_{n-1,\alpha+1}+\varphi_{n-1,\alpha}\varphi_{n,\alpha+2}\varphi_{1,\beta+1}\\
 & = & f_{1,\beta+1}\left(x_{n,\alpha}f_{n-1,\alpha+1}+f_{n-1,\alpha}f_{n,\alpha+2}\right)-f_{n,\alpha+1}f_{1,\beta+1}^{\beta+1\leftarrow\beta}f_{n-1,\alpha+1}.
\end{eqnarray*}
 
\begin{figure}
\begin{centering}
\includegraphics[scale=0.3]{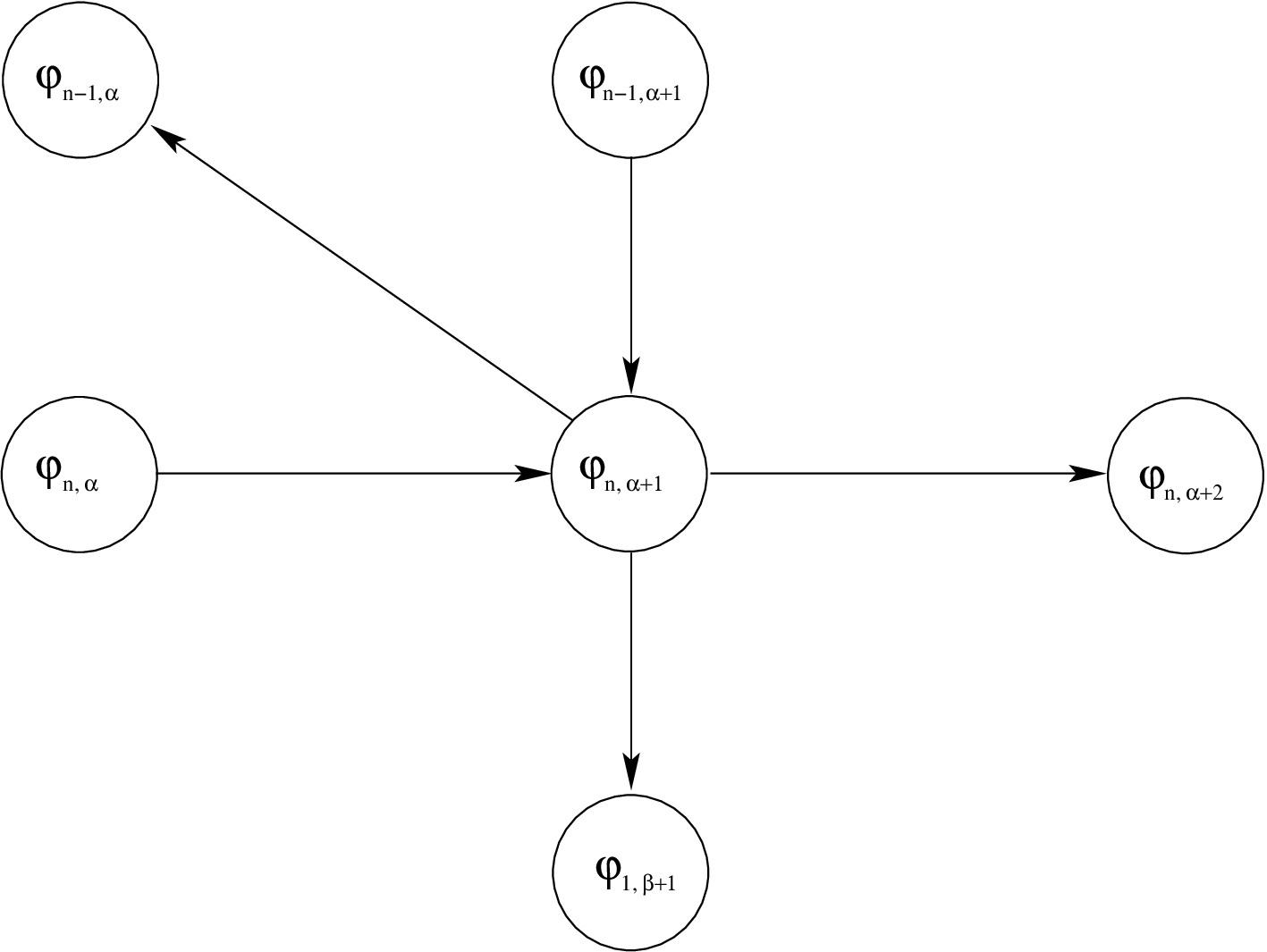} 
\par\end{centering}

\caption{The neighbors of $f_{n,\alpha+1}$}

\centering{}\label{fig:localfna+1} 
\end{figure}

The term in parenthesis is the exchange rule in the standard case,
so it is the product of $f_{n,\alpha+1}$ and some other regular function,
and the second term is clearly divisible by $f_{n,\alpha+1}$. Therefore
the exchanged variable is regular. Similar arguments hold for the
vertex $\varphi_{\beta+1,n}$.

3. $\varphi$ is not in $\mathcal{B}_{\alpha\beta}\cap\mathcal{B}_{std}$
.

(a) $\varphi$ is either $\varphi_{n\alpha}$ or $\varphi_{\beta n}$.\\
Assume $\varphi=\varphi_{n,\alpha}=x_{n,\alpha}f_{1,\beta+1}-x_{n,\alpha+1}f_{1,\beta+1}^{\leftarrow}$
. Assume $\alpha>1$ because if $\alpha=1$, the variable $\varphi_{n1}$
must be frozen. The adjacent vertices correspond to $\varphi_{n,\alpha-1},\varphi_{n-1,\alpha},\varphi_{n,\alpha+1},\varphi_{n-1,\alpha-1},\varphi_{1,\beta+1}$
where $\varphi_{n-1,\alpha-1}=f_{n-1,\alpha-1}f_{1,\beta+1}-f_{n-1,\alpha-1}^{\rightarrow}f_{1,\beta+1}^{\leftarrow}$
, as shown in Figure \ref{fig:localna}. The exchange rule is 
\begin{eqnarray*}
\varphi_{n,\alpha}\cdot\varphi_{n,\alpha}^{\prime} & = & \varphi_{n,\alpha+1}\varphi_{n-1,\alpha-1}+\varphi_{n,\alpha-1}\varphi_{n-1,\alpha}\varphi_{1,\beta+1}\\
 & = & x_{n,\alpha+1}\left(f_{n-1,\alpha-1}f_{1,\beta+1}-f_{n-1,\alpha-1}^{\rightarrow}f_{1,\beta+1}^{\leftarrow}\right)+x_{n,\alpha-1}f_{n-1,\alpha}f_{1,\beta+1}\\
 & = & f_{1,\beta+1}\left(x_{n,\alpha+1}f_{n-1,\alpha-1}+x_{n,\alpha-1}f_{n-1,\alpha}\right)-x_{n,\alpha+1}f_{n-1,\alpha-1}^{\rightarrow}f_{1,\beta+1}^{\leftarrow}
\end{eqnarray*}
 and the term in parenthesis is just the standard exchange rule, which
is $x_{n\alpha}\cdot f_{n-1,\alpha-1}^{\rightarrow}$. Therefore,
\[
\varphi_{n,\alpha}\cdot\varphi_{n,\alpha}^{\prime}=\left(x_{n\alpha}f_{1,\beta+1}-x_{n,\alpha+1}f_{1,\beta+1}^{\leftarrow}\right)f_{n-1,\alpha-1}^{\rightarrow}=\varphi_{n,\alpha}f_{n-1,\alpha-1}^{\rightarrow},
\]
and $\varphi_{n\alpha}^{\prime}=f_{n-1,\alpha-1}^{\rightarrow}$ is
a regular function. \\
Symmetric arguments show that $\varphi_{\beta n}^{\prime}$ is also
regular. 
\begin{figure}
\begin{centering}
\includegraphics[scale=0.3]{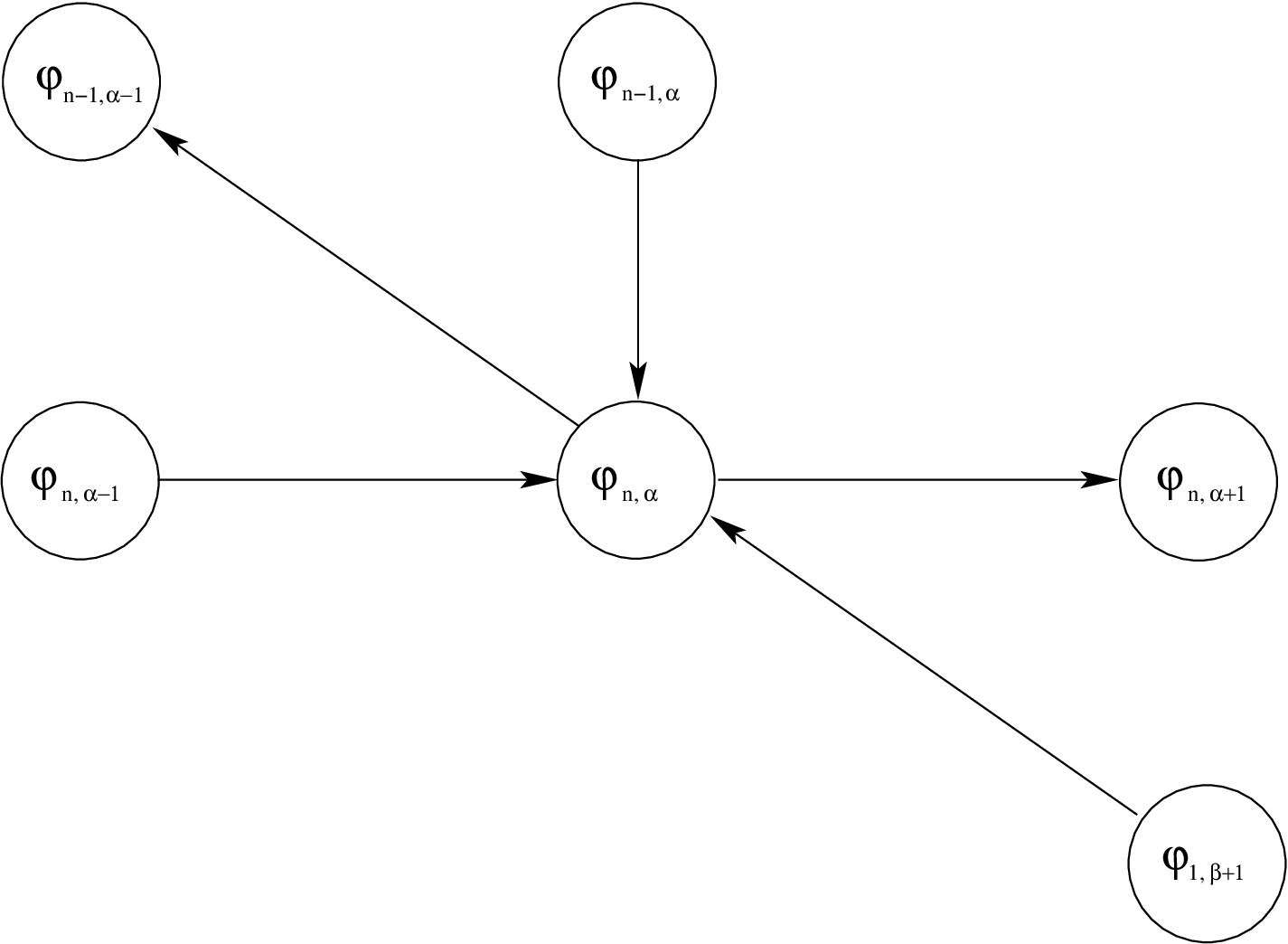} 
\par\end{centering}

\caption{The neighbors of $\varphi_{n\alpha}$}

\centering{}\label{fig:localna} 
\end{figure}

(b) $\varphi$ has two non standard neighbors. \\
This happens when $\varphi=\varphi_{ij}$, and the two non standard
neighbors are 
\begin{eqnarray*}
\varphi_{i-1,j-1} & = & f_{i-1,j-1}f_{1,\beta+1}-\overline{f}_{i-1,j-1}\overline{f}_{1,\beta+1}\\
\varphi_{i+1,j+1} & = & f_{i+1,j+1}f_{1,\beta+1}-\overline{f}_{i+1,j+1}\overline{f}_{1,\beta+1}.
\end{eqnarray*}
The other neighbors are the same neighbors from the standard case.
Denote the corresponding standard exchange rule at $f_{ij}$ by $e_{f_{ij}}.$
This is a Desnanot-Jacobi identity \eqref{eq:DesJacId} or the modified
version of it \eqref{eq:DesJacIdM}. Let $\overline{e}_{f_{ij}}$
be the same identity with column $\alpha$ (or row $\beta$ ) replaced
by column $\alpha+1$ (or row $\beta+1$, respectively). In other
words if $e_{f_{ij}}=f_{ij}\cdot g$ then $\overline{e}_{f_{ij}}=\overline{f}_{ij}\cdot g$,
and the exchange rule is 
\begin{eqnarray*}
\varphi_{ij}\cdot\varphi_{ij}^{\prime} & = & f_{1,\beta+1}e_{f_{ij}}-\overline{f}_{1,\beta+1}\overline{e}_{f_{ij}}\\
 & = & \left(f_{1,\beta+1}f_{ij}-\overline{f}_{1,\beta+1}\overline{f}_{ij}\right)g=\varphi_{ij}\cdot g
\end{eqnarray*}
and $g$ is the same regular function as in the standard case. 

4. $\varphi$ is $f_{1,\beta+1}$ or $f_{\alpha+1,1}$ (which were
frozen in the standard case).\\
Assume $\varphi=f_{1,\beta+1}$ with neighbors $\varphi_{n,\alpha},\varphi_{n,\alpha+1},\varphi_{1,\beta},\varphi_{2,\beta+1},\varphi_{2,\beta+2}$
(Figure \ref{fig:localf1b+1}). The exchange rule is then 
\[
\varphi\cdot\varphi^{\prime}=\varphi_{n\alpha}\varphi_{2,\beta+1}+\varphi_{n,\alpha+1}\varphi_{1,\beta}\varphi_{2,\beta+2}.
\]
 
\begin{figure}
\begin{centering}
\includegraphics[scale=0.3]{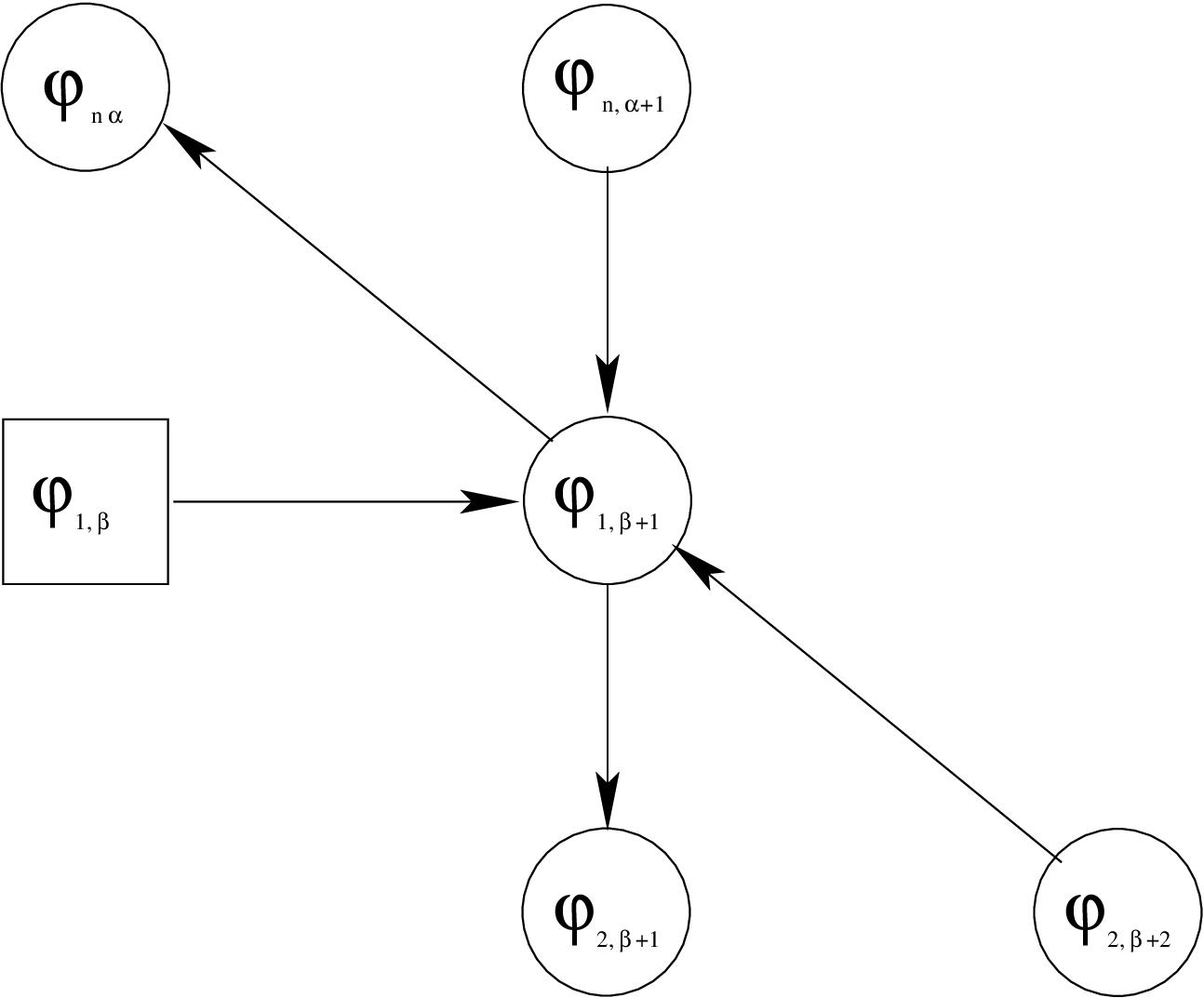} 
\par\end{centering}

\caption{The neighbors of $\varphi_{1,\beta+1}$}

\centering{}\label{fig:localf1b+1} 
\end{figure}

If we put 
\[
A=\left[\begin{array}{ccccc}
x_{n\alpha} & x_{n,\alpha+1} & 0 & \cdots & 0\\
x_{1,\beta} & x_{1,\beta+1} & \cdots & \cdots & x_{1n}\\
\vdots & x_{2,\beta+1} & x_{2,\beta+2} &  & \vdots\\
\vdots &  &  & \ddots\\
\\
\end{array}\right],
\]
 then the exchange rule reads 
\[
\varphi\cdot\varphi^{\prime}=\det A\cdot\det A_{\hat{1}\hat{2}}^{\hat{1}\hat{n}}+\det A_{\hat{2}}^{\hat{1}}\cdot\det A_{\hat{1}}^{\hat{n}}
\]
 and according to \eqref{eq:DesJacId} $\varphi\cdot\varphi^{\prime}=\det A_{\hat{1}}^{\hat{1}}\det A_{\hat{2}}^{\hat{n}}.$
Since $\varphi=\det A_{\hat{1}}^{\hat{1}},$ we get $\varphi^{\prime}=\det A_{\hat{2}}^{\hat{n}}$
, which is regular. The case of $\varphi=\varphi_{\alpha+1,1}$ is
symmetric.

This competes the proof, since in all cases the exchanged variable
$\varphi^{\prime}$ is a regular function.
\end{proof}

\section{Technical results and computations\label{sec:Technical}}

We present here the proofs to some technical results that were used
in previous sections. The bracket $\left\{ \cdot,\cdot\right\} $
will be computed through the operator $R_{+}$. Lemma \ref{lem:TheOperatorRpls}
explains this operator, while the rest of this section gives more
information about the standard bracket and the difference between
the $\alpha\beta$ bracket and the standard bracket.

\subsection{The operator $R_{+}$ \label{sub:operatorRpls}}

Following Lemma 4.1 in \cite{gekhtman2013arxiv}, we compute the Sklyanin
bracket $\left\{ f,g\right\} $ associated with an R-matrix $r$ through
\begin{equation}
\left\{ f,g\right\} (X)=\left\langle R_{+}\left(\nabla f(X)X\right),\nabla g(X)X\right\rangle -\left\langle R_{+}\left(X\nabla f(X)\right),X\nabla g(X)\right\rangle ,\label{eq:BrcktthTrc}
\end{equation}
 where $\left\langle X,Y\right\rangle =\Tr\left(XY\right)$, $\nabla$
is the gradient with respect to the trace-form, and $R_{+}\in\End\mathfrak{gl}_{n}$
as defined in \eqref{eq:RplsDef}. For the computations it will be
convenient to describe $R_{+}$ in a different way: for an element
$\eta\in\mathfrak{gl}_{n}$, let $\eta_{>0}$ and $\eta_{0}$ be the
projections of $\eta$ onto the subalgebra spanned by positive roots,
and the Cartan subalgebra $\mathfrak{h}$, respectively. Let $h_{i}=e_{ii}-e_{i+1,i+1}$
be a basis for $\mathfrak{h}$. The dual basis (defined by $\left\langle \hat{h}_{i},h_{j}\right\rangle =\delta_{ij}$)
is then 
\[
\hat{h}_{i}=\frac{1}{n}\diag\left(\underbrace{\left(n-i\right),\ldots,\left(n-i\right)}_{i},\underbrace{-i,\ldots,-i}_{\left(n-i\right)}\right).
\]
Defining 
\begin{equation}
s_{k}\left(j\right)=\begin{cases}
n-j & j\geq k\\
-j & j<k
\end{cases}\label{eq:sfuncdef}
\end{equation}
 we can write $\left(\hat{h}_{i}\right)_{kk}=\frac{1}{n}s_{k}\left(i\right).$
Next, define the operator $R_{\diag}$ on $\mathfrak{h}$ by 
\begin{eqnarray}
R_{\diag}\left(e_{kk}\right) & = & \sum_{j=1}^{n-1}s_{k}\left(j\right)\left(\hat{h}_{j}-\hat{h}_{j-1}\right)+\left(s_{k}\left(\beta-1\right)-s_{k}\left(\beta\right)\right)\hat{h}_{\alpha}\nonumber \\
 &  & +\left(s_{k}\left(\alpha\right)-s_{k}\left(\alpha+1\right)\right)\hat{h}_{\beta}+s_{k}\left(\beta\right)\hat{h}_{\alpha+1}-s_{k}\left(\alpha\right)\hat{h}_{\beta-1},\label{eq:Rplsonekk}
\end{eqnarray}

Last, for the Belavin--Drinfeld data $\left\{ \alpha\right\} \mapsto\left\{ \beta\right\} $
define 
\[
R_{BD}\left(\eta\right)=\eta_{\alpha,\alpha+1}e_{\beta,\beta+1}-\eta_{\beta+1,\beta}e_{\alpha+1,\alpha}.
\]

\begin{lem}
\label{lem:TheOperatorRpls}The operator $R_{+}$ acts on $\eta\in\mathfrak{gl}_{n}$
by 
\begin{equation}
R_{+}\left(\eta\right)=\eta_{>0}+R_{\diag}\left(\eta\right)+R_{BD}\left(\eta\right).\label{eq:RplsSmplForm}
\end{equation}
 \end{lem}
\begin{proof}
Recall the construction of the R-matrix $r^{\alpha\beta}$ according
to \eqref{eq:RmtxCons}: there is some freedom in choosing the diagonal
part $r_{0}$. Following \cite[Ch. 3]{chriprsly}, $r_{0}=\sum_{i,j}a_{ij}\hat{h}_{\alpha_{i}}\otimes\hat{h}_{\alpha_{j}}$
is determined by the coefficient matrix $\left(a_{ij}\right)$, which
is subject to the conditions 
\begin{eqnarray}
a_{ij}+a_{ji}=\left(\alpha_{i},\alpha_{j}\right) &  & \text{if }\alpha_{i},\alpha_{j}\in\Delta\label{eq:r0CfMatcnd1}\\
a_{\gamma\left(i\right),j}+a_{ji}=0 &  & \text{if }\alpha_{i}\in\Gamma_{1},\alpha_{j}\in\Delta.\label{eq:r0CfMatcnd2}
\end{eqnarray}
 Define two matrices, 
\begin{align*}
A_{ij} & =\begin{cases}
1 & i=j\\
-1 & i=j+1\\
0 & \text{otherwise, }
\end{cases}\\
\\
B_{ij} & =\begin{cases}
1 & \left(\alpha_{i},\alpha_{j}\right)\in\left\{ \left(\alpha,\beta\right),\left(\beta-1,\alpha\right),\left(\beta,\alpha+1\right)\right\} \\
-1 & \left(\alpha_{i},\alpha_{j}\right)\in\left\{ \left(\beta,\alpha\right),\left(\alpha,\beta-1\right),\left(\alpha+1,\beta\right)\right\} \\
0 & \text{otherwise.}
\end{cases}
\end{align*}
It is not hard to see that for the BD triple $T_{\alpha\beta}=\left(\left\{ \alpha\right\} ,\left\{ \alpha+1\right\} ,\gamma:\alpha\mapsto\alpha+1\right)$
(i.e., $\beta=\alpha+1$) the matrix $A$ satisfies conditions \eqref{eq:r0CfMatcnd1}
and \eqref{eq:r0CfMatcnd2}, and can serve as the coefficient matrix
that determines $r_{0}$. When $\beta>\alpha+1$, we take $A+B$ as
the coefficient matrix, and conditions \eqref{eq:r0CfMatcnd1} and
\eqref{eq:r0CfMatcnd2} are satisfied again. So for $\beta=\alpha+1$
take

\begin{equation}
r_{0}^{\alpha\beta}=\sum_{i=1}^{n-1}\hat{h}_{i}\otimes\hat{h}_{i}-\sum_{i=1}^{n-2}\hat{h}_{i+1}\otimes\hat{h}_{i}+\hat{h}_{\alpha}\wedge\hat{h}_{\beta}+\hat{h}_{\beta-1}\wedge\hat{h}_{\alpha}+\hat{h}_{\beta}\wedge\hat{h}_{\alpha+1},\label{eq:r0def}
\end{equation}
 and for $\beta\neq\alpha+1$ 
\begin{equation}
r_{0}^{\alpha\beta}=\sum_{i=1}^{n-1}\hat{h}_{i}\otimes\hat{h}_{i}-\sum_{i=1}^{n-2}\hat{h}_{i+1}\otimes\hat{h}_{i};\label{eq:r0defbneqa+1}
\end{equation}
Note that in the standard case condition \eqref{eq:r0CfMatcnd2} is
empty, so we can use $r_{0}^{\alpha\beta}$ in the the standard case
as well. \\
To prove the Lemma, it is enough to show that \eqref{eq:RplsSmplForm}
holds for all elements of the basis $\left\{ E_{\delta}\right\} _{\delta\in\Phi}\cup\left\{ h_{k}\right\} _{k=1}^{n-1}$.
Recall that $\left\langle E_{-i},E_{j}\right\rangle =\delta_{ij}$: 

1. $\eta=E_{\delta},$ with $\delta\in\Phi^{+}$. so 
\begin{eqnarray*}
\left\langle r,\eta\otimes\zeta\right\rangle  & = & \left\langle \left(E_{-\delta}\otimes E_{\delta}+\sum_{\alpha\prec\beta}E_{-\alpha}\wedge E_{\beta}\right),E_{\delta}\otimes\zeta\right\rangle \\
 & = & \begin{cases}
\left\langle E_{\delta},\zeta\right\rangle  & \delta\notin\Gamma_{1}\\
\left\langle E_{\delta},\zeta\right\rangle +\sum_{\delta\prec\beta}\left\langle E_{\beta},\zeta\right\rangle  & \delta\in\Gamma_{1},
\end{cases}
\end{eqnarray*}
 and $R_{+}\left(E_{\delta}\right)=\begin{cases}
E_{\delta} & \delta\notin\Gamma_{1}\\
\left(E_{\delta}-\sum_{\delta\prec\beta}E_{\beta}\right) & \delta\in\Gamma_{1},
\end{cases}$in accordance with \eqref{eq:RplsSmplForm}.

2. $\eta=E_{-\delta},$ with $\delta\in\Phi^{+}$. 
\begin{eqnarray*}
\left\langle r,\eta\otimes\zeta\right\rangle  & = & \left\langle \left(\sum_{\alpha\prec\beta}E_{-\alpha}\wedge E_{\beta}\right),E_{-\delta}\otimes\zeta\right\rangle \\
 & = & \begin{cases}
0 & \delta\notin\Gamma_{2}\\
-\sum_{\alpha\prec\delta}\left\langle E_{-\alpha},\zeta\right\rangle  & \delta\in\Gamma_{2},
\end{cases}
\end{eqnarray*}
 hence $R_{+}\left(E_{-\delta}\right)=\begin{cases}
0 & \delta\notin\Gamma_{2}\\
-\sum_{\alpha\prec\delta}\left\langle E_{-\alpha},\zeta\right\rangle  & \delta\in\Gamma_{2},
\end{cases}$ which also fits \eqref{eq:RplsSmplForm}.

3. $\eta=h_{k}.$ 
\begin{eqnarray*}
\left\langle r,\eta\otimes\zeta\right\rangle  & = & \left\langle \sum_{i=1}^{n-2}\hat{h}_{i}\wedge\hat{h}_{i+1},h_{k}\otimes\zeta\right\rangle \\
 & = & \left\langle \hat{h}_{k+1},\zeta\right\rangle -\left\langle \hat{h}_{k-1},\zeta\right\rangle 
\end{eqnarray*}
 (with $\hat{h}_{0}=\hat{h}_{n}=0$ ). Therefore $R_{+}\left(h_{k}\right)=\left(\hat{h}_{k+1}-\hat{h}_{k-1}\right)$.
Expressing $e_{kk}$ as a linear combination of $\left\{ h_{i}\right\} _{i=1}^{n-1}\cup\left\{ \Id\right\} $
implies \eqref{eq:RplsSmplForm}.

4. Last, look at $\eta=\Id$. Here it is clear that 
\[
\left\langle r,\Id\otimes\zeta\right\rangle =0.
\]
  This implies $R_{+}\left(\Id\right)=0,$ and the proof is complete. 
\end{proof}

\subsection{Bracket computations \label{sub:Bracket-computations}}
\begin{lem}
\label{lem:PsnBrcktDiff}For any two functions $f,g$ on $SL_{n},$
\begin{eqnarray}
\left\{ f,g\right\} _{\alpha\beta}-\left\{ f,g\right\} _{std} & = & f^{\alpha\leftarrow\alpha+1}g^{\beta+1\leftarrow\beta}-f^{\beta+1\leftarrow\beta}g^{\alpha\leftarrow\alpha+1}\nonumber \\
 &  & +f_{\beta\leftarrow\beta+1}g_{\alpha+1\leftarrow\alpha}-f_{\alpha+1\leftarrow\alpha}g_{\beta\leftarrow\beta+1}.\label{eq:PsnBrcktDiff}
\end{eqnarray}
\end{lem}
\begin{proof}
Let $r_{\alpha\beta}$ and $r_{std}$ be the R-matrices associated
with BD data $\left\{ \alpha\right\} \to\left\{ \beta\right\} $ and
$\emptyset\to\emptyset$, respectively. Using \eqref{eq:BrcktthTrc},
it is easy to see that the difference \eqref{eq:PsnBrcktDiff} comes
from the difference $R_{+}^{\alpha\beta}-R_{+}^{std}$. According
to Lemma \ref{lem:TheOperatorRpls}, this is 
\begin{equation}
R_{+}^{\alpha\beta}\left(\eta\right)-R_{+}^{std}\left(\eta\right)=\eta_{\alpha,\alpha+1}e_{\beta,\beta+1}-\eta_{\beta+1,\beta}e_{\alpha+1,\alpha}.
\end{equation}
Write $R_{d}=R_{+}^{\alpha\beta}-R_{+}^{std}$, so now 
\begin{align*}
 & \left\{ f,g\right\} _{\alpha\beta}-\left\{ f,g\right\} _{std}\\
 & =\left\langle R_{d}\left(\nabla f(X)X\right),\nabla g(X)X\right\rangle -\left\langle R_{d}\left(X\nabla f(X)\right),X\nabla g(X)\right\rangle \\
 & =\left(\nabla f(X)X\right)_{\alpha,\alpha+1}\left(\nabla g(X)X\right)_{\beta+1,\beta}-\left(\nabla f(X)X\right)_{\beta+1,\beta}\left(\nabla g(X)X\right)_{\alpha,\alpha+1}\\
 & -\left(X\nabla f(X)\right)_{\alpha,\alpha+1}\left(X\nabla g(X)\right)_{\beta+1,\beta}-\left(X\nabla f(X)\right)_{\beta+1,\beta}\left(X\nabla g(X)\right)_{\alpha,\alpha+1}\\
 & =f^{\alpha\leftarrow\alpha+1}g^{\beta+1\leftarrow\beta}-f^{\beta+1\leftarrow\beta}g^{\alpha\leftarrow\alpha+1}-f_{\alpha+1\leftarrow\alpha}g_{\beta\leftarrow\beta+1}+f_{\beta\leftarrow\beta+1}g_{\alpha+1\leftarrow\alpha}.
\end{align*}
 \end{proof}
\begin{cor}
If $f,g\in\mathcal{B}_{\alpha\beta}\cap\mathcal{B}_{std}$ then $\left\{ f,g\right\} _{\alpha\beta}=\left\{ f,g\right\} _{std}$.\label{cor:fginSbrckteq}\end{cor}
\begin{proof}
All functions in $\mathcal{B}_{std}$ are determinants of submatrices
of $X$. Let $f_{ij}$ be such a function as defined in \eqref{eq:fijAsXSubMat}.
The term $f_{ij}^{k\leftarrow m}$ is the determinant of a similar
submatrix, with column $m$ replacing column $k$ (i.e., every instance
of $x_{pk}$ is replaced by $x_{pm}$). Therefore, for $f_{ij}\in\mathcal{B}_{std}$,
the function $f_{ij}^{\alpha\leftarrow\alpha+1}$ is non zero only
if $f_{ij}$ is a determinant of a submatrix that contains column
$\alpha$ but not column $\alpha+1$. The only functions with this
property in $\mathcal{B}_{std}$ are determinants of submatrices of
the form $X_{\left[n+j-\alpha,n\right]}^{\left[j,\alpha\right]}$,
that is, the functions $f_{n+j-\alpha,j}$ with $j\in\left[\alpha\right]$.
But these functions are not in $\mathcal{B}_{\alpha\beta}$, because
$\alpha\in\Gamma_{1}$ (see the construction in Section \ref{sub:Constructing-a-log}).
Similarly, $f_{m\leftarrow k}$ is the determinant of the matrix obtained
by replacing the $m$-th of $X$ row by the $k$-th row. So the function
$\left(f_{ij}\right)_{\beta\leftarrow\beta+1}$ is non zero only if
$f_{ij}$ is the determinant of a submatrix with row $\beta$ and
without row $\beta+1$. The only functions in $\mathcal{B}_{std}$
that satisfy this condition are $f_{i,n+i-\beta},$ and these functions
are not in $\mathcal{B}_{\alpha\beta}$ because $\beta\in\Gamma_{2}$
(see Section \ref{sub:Constructing-a-log} again).
\end{proof}
The next Lemma describes the ``building blocks'' of the functions
in $\mathcal{B}_{\alpha\beta}\setminus\mathcal{B}_{std}$ and the
Poisson coefficients of these functions with respect to the standard
bracket. 
\begin{lem}
\label{lem:NewFuncLCandCf}1. The function $f_{n+k-\alpha,k}^{\rightarrow}$
(with $k\in[\alpha]$) is log canonical with all functions $g\in\mathcal{B}_{\alpha\beta}\cap\mathcal{B}_{std}$,
provided $g\neq f_{n+m-\alpha,m}$ for some $m>k$, w.r.t. the standard
bracket $\left\{ \cdot,\cdot\right\} _{std}$. In this case the Poisson
coefficient is 
\begin{equation}
\omega_{f_{n+k-\alpha,k}^{\rightarrow},g}=\omega_{f_{n+k-\alpha,k},g}+\omega_{x_{n,\alpha+1},g}-\omega_{x_{n\alpha},g}.
\end{equation}
2. The function $f_{1,\beta+1}^{\leftarrow}$ is log canonical with
all functions $g\in\mathcal{B}_{\alpha\beta}\cap B_{std}$, provided
 $g\neq f_{m,\beta+1}$ for some $m\in[2,n]$ w.r.t. the standard
bracket $\left\{ \cdot,\cdot\right\} _{std}$. In this case the Poisson
coefficient is 
\begin{equation}
\omega_{f_{1,\beta+1}^{\leftarrow},g}=\omega_{f_{1,\beta+1},g}+\omega_{x_{n\beta},g}-\omega_{x_{n,\beta+1},g}.
\end{equation}
 \end{lem}
\begin{proof}
The proof will use the \emph{Desnanot--Jacobi }identity (see \cite{Bressoud}):
for a square matrix $A$, denote by ``hatted'' subscripts and superscripts
deleted rows and columns, respectively. Then 
\begin{equation}
\det A\cdot\det A_{\hat{r}_{1},\hat{r}_{2}}^{\hat{c}_{1},\hat{c}_{2}}=\det A_{\hat{r}_{1}}^{\hat{c}_{1}}\cdot\det A_{\hat{r}_{2}}^{\hat{c}_{2}}-\det A_{\hat{r}_{2}}^{\hat{c}_{1}}\cdot\det A_{\hat{r}_{1}}^{\hat{c}_{2}}.\label{eq:DesJacId}
\end{equation}
 By adding an appropriate row, we get a similar result for a non square
matrix $B$ with number of rows greater by one than the number of
columns: 
\begin{equation}
\det B_{\hat{r_{1}}}\det B_{\hat{r}_{2}\hat{r}_{3}}^{\hat{c}_{1}}=\det B_{\hat{r}_{2}}\det B_{\hat{r}_{1}\hat{r}_{3}}^{\hat{c}_{1}}-\det B_{\hat{r}_{3}}\det B_{\hat{r}_{1}\hat{r_{2}}}^{\hat{c}_{1}},\label{eq:DesJacIdM}
\end{equation}
 and naturally, a similar identity can be obtained for a matrix with
number of columns greater by one than the number of rows. 

Start with statement 1. of the Lemma. We will show that $f_{n+k-\alpha,k}^{\rightarrow}$
is a cluster variable that can be obtained from the initial cluster
through the mutation sequence $\left(f_{n\alpha},f_{n-1,\alpha-1},\ldots,f_{n+k-1-\alpha,k-1}\right)$.
Look at the initial quiver described in Section \ref{sub:BisLC},
and mutate in direction $f_{n\alpha}$. We can assume $\alpha>1$
because if $\alpha=1$ then $f_{n\alpha}$ is frozen. In this case
statement 1 holds trivially, with $f_{n\alpha}^{\rightarrow}=f_{n,\alpha+1}\in\mathcal{B}_{std}$.
For $\alpha>1$ the exchange rule is 
\begin{eqnarray}
f_{n\alpha}\cdot f{}_{n\alpha}^{\prime} & = & f_{n,\alpha-1}f_{n-1,\alpha}+f_{n-1,\alpha-1}f_{n,\alpha+1}\nonumber \\
 & = & x_{n,\alpha-1}\left|\begin{array}{cc}
x_{n-1,\alpha} & x_{n-1,\alpha+1}\\
x_{n\alpha} & x_{n,\alpha+1}
\end{array}\right|+\left|\begin{array}{cc}
x_{n-1,\alpha-1} & x_{n-1,\alpha}\\
x_{n,\alpha-1} & x_{n\alpha}
\end{array}\right|x_{n,\alpha+1}\label{eq:ExRl@fna}\\
 & = & x_{n\alpha}\left|\begin{array}{cc}
x_{n-1,\alpha-1} & x_{n-1,\alpha+1}\\
x_{n,\alpha-1} & x_{n,\alpha+1}
\end{array}\right|=x_{n\alpha}f_{n-1,\alpha-1}^{\rightarrow},\nonumber 
\end{eqnarray}
 which implies $f{}_{n\alpha}^{\prime}=f_{n-1,\alpha-1}^{\rightarrow}.$
The arrows of the quiver take the following changes:\\
the arrows $\left(n,\alpha+1\right)\to\left(n-1,\alpha\right),\left(n-1,\alpha-1\right)\to\left(n,\alpha-1\right)$
and $\left(n-1,\alpha-1\right)\to\left(n-1,\alpha\right)$ are removed,
and an arrow $\left(n,\alpha-1\right)\to\left(n,\alpha+1\right)$
is added. All arrows incident to $\left(n,\alpha\right)$ are inverted.
Therefore the exchange rule at $\left(n-1,\alpha-1\right)$ is now
\[
f_{n-1,\alpha-1}\cdot f{}_{n-1,\alpha-1}^{\prime}=f{}_{n\alpha}^{\prime}f_{n-2,\alpha-2}+f_{n-2,\alpha-1}f_{n-1,\alpha-2}.
\]
 Proceed with the mutation sequence $\left(f_{n\alpha},f_{n-1,\alpha-1},\ldots,f_{n+k-1-\alpha,k-1}\right)$.
Assume by induction that for $m\in[\alpha]$, mutating at $f_{n+m-\alpha,m}$
yields the exchanged variable 
\begin{equation}
f{}_{n+m-\alpha,m}^{\prime}=f_{n-m-1,\alpha-m-1}^{\rightarrow},
\end{equation}

and that the exchange rule at $f_{n-m-1,\alpha-m-1}$ is now 
\begin{eqnarray*}
f_{n-m-1,\alpha-m-1}\cdot f{}_{n-m-1,\alpha-m-1}^{\prime} & = & f{}_{n+m-\alpha,m}^{\prime}f_{n-m-2,\alpha-m-2}\\
 & + & f_{n-m-2,\alpha-m-1}f_{n-m-1,\alpha-m-2}.
\end{eqnarray*}
 Write $A=X_{\left[n-m-2,n\right]}^{\left[\alpha-m-2,\alpha+1\right]}$
and let $\ell$ be the last column of $A$. Using \eqref{eq:DesJacIdM}
we get 
\begin{eqnarray*}
 &  & f_{n-m-1,\alpha-m-1}\cdot f{}_{n-m-1,\alpha-m-1}^{\prime}\\
 & = & \det A_{\hat{1}}^{\hat{1},\widehat{\ell-1}}\det A^{\hat{\ell}}+\det A_{\hat{1}}^{\widehat{\ell-1},\hat{\ell}}\det A^{\hat{m-2}}\\
 & = & \det A_{\hat{1}}^{\hat{1},\hat{\ell}}\det A^{\widehat{\ell-1}}\\
 & = & f_{n-m-1,\alpha-m-1}\cdot f_{n-m-2,\alpha-m-2}^{\rightarrow},
\end{eqnarray*}
and therefore 
\[
f{}_{n-m-1,\alpha-m-1}^{\prime}=f_{n+m-2,\alpha-m-2}^{\rightarrow}
\]
 The quiver mutates as follows: arrows $\left(n-m-2,\alpha-m-2\right)\to\left(n-m-2,\alpha-m-1\right)$and
$\left(n-m-2,\alpha-m-2\right)\to\left(n-m-1,\alpha-m-2\right)$ are
removed, arrows $\left(n-m-2,\alpha-m-1\right)\to\left(n+m-\alpha,m\right)$
and $\left(n-m-1,\alpha-m-2\right)\to\left(n+m-\alpha,m\right)$ added,
and all arrows incident to $n-m-1,\alpha-m-1$ are inverted. Therefore
the mutation rule at the next cluster variable of the sequence will
be now 
\begin{eqnarray*}
f_{n-m-2,\alpha-m-2}\cdot f{}_{n-m-2,\alpha-m-2}^{\prime} & = & f{}_{n-m-1,\alpha-m-1}^{\prime}f_{n-m-3,\alpha-m-3}\\
 &  & +f_{n-m-3,\alpha-m-2}f_{n-m-2,\alpha-m-3}.
\end{eqnarray*}
 That proves that after the mutation sequence $\left(f_{n\alpha},f_{n-1,\alpha-1},\ldots,f_{n-k+1,\alpha-k+1}\right)$
we get $f_{n-k+1,\alpha-k+1}^{\prime}=f_{n+k-\alpha,k}^{\rightarrow}$
and therefore it is log canonical with all cluster variables of the
initial cluster, except for $\left(f_{n\alpha},f_{n-1,\alpha-1},\ldots,f_{n-k+1,\alpha-k+1}\right)$
that were mutated.

Now for $g\neq f_{n+m-\alpha,m}$ with $m>k+1$, the coefficient $\omega_{f_{n+k-\alpha,k},g}$
can be computed: from the Leibniz rule for Poisson brackets, any triple
of functions $f_{1},f_{2},g$ such that $\left\{ f_{1},g\right\} =\omega_{1}f_{1}g$
and $\left\{ f_{2},g\right\} =\omega_{2}f_{2}g,$ must satisfy 
\[
\left\{ f_{1}f_{2},g\right\} =\left(\omega_{2}+\omega_{2}\right)f_{1}f_{2}g,
\]
 or, in other words $\omega_{f_{1}f_{2},g}=\omega_{f_{1},g}+\omega_{f_{2},g}$.
Applying this together with the linearity of the bracket to the exchange
rule \eqref{eq:ExRl@fna} we get 
\[
\omega_{f_{n\alpha},g}+\omega_{f^{\prime}}=\omega_{f_{n-1,\alpha-1},g}+\omega_{x_{n,\alpha+1},g},
\]
 which is 
\begin{equation}
\omega_{f_{n-1,\alpha-1}^{\rightarrow},g}=\omega_{f_{n-1,\alpha-1},g}+\omega_{x_{n,\alpha+1},g}-\omega_{f_{n\alpha},g}.
\end{equation}
Again, we proceed inductively: assume that 
\[
\omega_{f_{n-k+1,\alpha-k+1}^{\rightarrow},g}=\omega_{f_{n+k-\alpha,k},g}+\omega_{x_{n,\alpha+1},g}-\omega_{x{}_{n\alpha},g}
\]
 and the exchange rule at $f_{n+k-\alpha,k}$ is 
\[
f_{n+k-\alpha,k}\cdot f_{n-k-1,\alpha-k-1}^{\rightarrow}=f{}_{n-k+1,\alpha-k+1}^{\rightarrow}f_{n-k-1,\alpha-k-1}+f_{n+k-\alpha,k-1}f_{n-k-1,\alpha-k}.
\]
 This means that 
\[
\omega_{f_{n-k-1,\alpha-k-1}^{\rightarrow},g}=\omega_{f_{n-k+1,\alpha-k+1}^{\rightarrow},g}+\omega_{f_{n-k-1,\alpha-k-1},g}-\omega_{f_{n+k-\alpha,k},g},
\]
 and recursively this leads to 
\begin{equation}
\omega_{f_{n-k-1,\alpha-k-1}^{\rightarrow},g}=\omega_{f_{n-k-1,\alpha-k-1},g}+\omega_{x_{n,\alpha+1},g}-\omega_{x_{n\alpha},g},
\end{equation}
 which complete the proof of statement 1.

Next, look at statement 2. Here also, we will show that $f_{1,\beta+1}^{\leftarrow}$
is a cluster variable that can be obtained through a mutation sequence,
which in this case is $\left(f_{n,\beta+1},f_{n-1,\beta+1},\ldots,f_{2,\beta+1}\right)$.
First, mutate at $f_{n,\beta+1}.$ It is easy to see, just like in
\eqref{eq:ExRl@fna} that 
\[
f^{\prime}=\left|\begin{array}{cc}
x_{n-1,\beta} & x_{n-1,\beta+2}\\
x_{n,\beta} & x_{n,\beta+2}
\end{array}\right|=f_{n-1,\beta+1}^{\leftarrow}.
\]
 Just like we have already showed above, arrows 
 $\left(n,\beta+2\right)\to\left(n-1,\beta+1\right),\left(,n-1,\beta\right)\to\left(n,\beta\right)$
and $\left(n-1,\beta\right)\to\left(n-1,\beta+1\right)$ are removed
from the quiver, and an arrow $\left(n,\beta\right)\to\left(n,\beta+2\right)$
is added to it. In addition, all the arrows adjacent to $\left(n,\beta+1\right)$
are inverted. The exchange rule at $f_{n-1,\beta+1}$ is now 
\[
f_{n-1,\beta+1}f{}_{n-1,\beta+1}^{\prime}=f{}_{n,\beta+1}^{\prime}f_{n-2,\beta+1}+f_{n-1,\beta+2}f_{n-2,\beta}.
\]
Again we use induction on $m$ with the mutation sequence $\left(f_{n,\beta+1},f_{n-1,\beta+1},\ldots,f_{m,\beta+1}\right)$.
Assume that after mutating at $f_{m+1,\beta+1}$ we got 
\[
f^{\prime}=f_{m,\beta+1}^{\leftarrow}
\]
 and that the exchange rule at $f_{m,\beta+1}$ is 
\begin{equation}
f_{m,\beta+1}\cdot f{}_{m,\beta+1}^{\prime}=f{}_{m+1,\beta+1}^{\prime}f_{m-1,\beta+1}+f_{m,\beta+2}f_{m-1,\beta}.
\end{equation}
 If $m>\beta+1$ then we can set $\mu=\mu\left(\beta,m-1\right)$
and $B=X_{\left[m-1,n\right]}^{\left[\beta,\mu+1\right]}$. Then the
exchange rule is 
\begin{eqnarray*}
f_{m,\beta+1}\cdot f{}_{m,\beta+1}^{\prime} & = & \det B^{\hat{\beta}}\det B_{\hat{m-1}}^{\hat{2}\widehat{\mu+\beta}}+\det B^{\widehat{\mu+1}}\det B_{\hat{m-1}}^{\hat{2}\hat{\beta}}\\
 & = & \det B^{\hat{2}}\det B_{\hat{m-1}}^{\hat{\beta}\widehat{\mu+1}}=f_{m-1,\beta+1}^{\leftarrow}f_{m,\beta+1}.
\end{eqnarray*}
 If, on the other hand, $m\leq\beta+1$ we set $\mu=\mu\left(\beta,m-1\right)$
and $A=X_{\left[m-1,\mu\right]}^{\left[\beta,n\right]}$ so that the
exchange rule becomes 
\begin{eqnarray*}
f_{m,\beta+1}\cdot f{}_{m,\beta+1}^{\prime} & = & \det A\det A_{\widehat{m-1}\hat{\mu}}^{\hat{\beta}\widehat{\beta+1}}+\det A_{\hat{\mu}}^{\hat{\beta}}\det A_{\widehat{m-1}}^{\widehat{\beta+1}}\\
 & = & \det A_{\widehat{m-1}}^{\hat{\beta}}\det A_{\hat{\mu}}^{\widehat{\beta+1}}=f_{m,\beta+1}f_{m-1,\beta+1}^{\leftarrow},
\end{eqnarray*}
 hence 
\[
f{}_{m,\beta+1}^{\prime}=f_{m-1,\beta+1}^{\leftarrow}.
\]
 It is easy to see that the mutation of the quiver also agrees with
the induction hypothesis, and we can conclude that after the mutation
sequence 
\begin{equation}
f{}_{2,\beta+1}^{\prime}=f_{1,\beta+1}^{\leftarrow},
\end{equation}
 and therefore $f_{1,\beta+1}^{\leftarrow}$ is log canonical with
all functions $f_{ij}\in\mathcal{B}_{std}$, excluding the functions
$f_{m,\beta+1}$ that were mutated on the way. 

We can now compute the coefficients $\omega_{f_{m-1,\beta+1}^{\leftarrow},g}$
recursively like we did in the first statement and get for every $f_{m,\beta+1}\neq g\in\mathcal{B}_{std}$,
\begin{equation}
\omega_{f_{1,\beta+1}^{\leftarrow},g}=\omega_{f_{1,\beta+1},g}+\omega_{x_{n,\beta},g}-\omega_{x_{n,\beta+1},g}.
\end{equation}
This completes the proof for statement 2. 
\end{proof}
The functions $f_{1,\beta+1}^{\leftarrow}$ and $f_{\alpha+1,1}^{\uparrow}$
need some special attention:
\begin{lem}
\label{lem:ArFuncLC}For $k\in\left[\beta\right]$, the function $f_{k,n+k-\beta}^{\downarrow}$
is log canonical with $f_{1,\beta+1}^{\leftarrow}$ and $f_{\alpha+1,1}^{\uparrow}$.
The Poisson coefficients are 
\begin{eqnarray*}
\omega_{f_{k,n+k-\beta}^{\downarrow},f_{1,\beta+1}^{\leftarrow}} & = & \omega_{f_{k,n+k-\beta},f_{1,\beta+1}}+\omega_{f_{\beta+1,n},f_{1,\beta+1}}-\omega_{f_{\beta n},f_{1,\beta+1}}\\
 &  & +\omega_{f_{k,n+k-\beta},f_{n\beta}}+\omega_{f_{\beta+1,n},f_{n\beta}}-\omega_{f_{\beta n},f_{n\beta}}\\
 &  & -\omega_{f_{k,n+k-\beta},f_{n,\beta+1}}-\omega_{f_{\beta+1,n},f_{n,\beta+1}}+\omega_{f_{\beta n},f_{n,\beta+1}},\\
\omega_{f_{k,n+k-\beta}^{\downarrow},f_{\alpha+1,1}^{\uparrow}} & = & \omega_{f_{k,n+k-\beta},f_{\alpha+1,1}}+\omega_{f_{\beta+1,n},f_{\alpha+1,1}}-\omega_{f_{\beta n},f_{\alpha+1,1}}\\
 &  & +\omega_{f_{k,n+k-\beta},f_{\alpha n}}+\omega_{f_{\beta+1,n},f_{\alpha n}}-\omega_{f_{\beta n},f_{\alpha n}}\\
 &  & -\omega_{f_{k,n+k-\beta},f_{\alpha+1,n}}-\omega_{f_{\beta+1,n},f_{\alpha+1,n}}+\omega_{f_{\beta n},f_{\alpha+1,n}}.
\end{eqnarray*}
\end{lem}
\begin{proof}
Naturally, Lemma \ref{lem:NewFuncLCandCf} could be helpful, but it
may seem that the proof does not hold: since the path $\left(f_{\beta n},f_{\beta-1,n-1},\ldots\right)$
crosses the paths $\left(f_{\alpha+1,n},f_{\alpha+1,n-1},\ldots\right)$
and $\left(f_{n,\beta+1},f_{n-1,\beta+1},\ldots\right)$, applying
the first mutation sequence followed by the second (or the third)
one, will not yield the function $f_{\alpha+1,1}^{\uparrow}$ (or
$f_{1,\beta+1}^{\leftarrow}$), because one of the cluster variables
had been mutated in the first sequence. However, this can be easily
settled. First, apply the sequence $\left(f_{\beta n},f_{\beta-1,n-1},\ldots,\right)$.
Now shift every mutated vertex $\left(\beta-m,n-m\right)$ of the
new quiver to the place $\left(\beta-m-1,n-m-1\right)$ i.e., move
it one row up and one column to the left. The quiver now looks locally
just like the initial one, with two changes at $f_{\beta n}$ and
at $f_{1,n+1-\beta}$. Then, set $\ell=2\beta-n+1$. Note that if
$\ell\leq1$ the paths do not cross each other, and there is no problem.
Now apply the sequence $\left(f_{n,\beta+1},f_{n-1,\beta+1},\ldots,f_{\ell+1,\beta+1}\right)$.
The quiver then reads the exact same exchange rules as the initial
quiver. At $f{}_{\ell+1,\beta+1}$ the exchange rule is then almost
the same as it was in the proof above, with one change: the function
$f_{\ell,\beta+1}$ is now replaced by $f_{\ell-1,\beta+1}^{\downarrow}$.
The exchange rule is 
\[
f_{\ell+1,\beta+1}\cdot f{}_{\ell+1,\beta+1}^{\prime}=f_{\ell,\beta}f_{\ell+2,\beta+2}+f{}_{\ell+1,\beta+1}^{\leftarrow}f{}_{\ell,\beta+1}^{\downarrow}.
\]
 So write $A=X_{\left[\ell,\beta+1\right]}^{\left[\beta,n\right]}$
and then 
\begin{eqnarray*}
f_{\ell+1,\beta+1}\cdot f{}_{\ell+1,\beta+1}^{\prime} & = & \det A\det A_{\hat{\ell}\widehat{\ell+1}}^{\hat{\beta}\widehat{\beta+1}}+\det A_{\hat{\ell}}^{\widehat{\beta+1}}\det A_{\hat{\beta}}^{\hat{\beta}}\\
 & = & \det A_{\hat{\ell}}^{\hat{\beta}}\det A_{\hat{\beta}}^{\widehat{\beta+1}}=f_{\ell+1,\beta+1}f_{\ell,\beta+1}^{\downarrow\leftarrow}.
\end{eqnarray*}

The picture is slightly different in the special case of $\beta=n-1$,
because now the column $\beta+1$ is the last one, but it is not hard
to see that the result is still 
\begin{equation}
f_{\ell+1,\beta+1}^{\prime}=f_{nn}^{\prime}=x_{n-2,n-1}=f_{\ell,\beta+1}^{\downarrow\leftarrow}.
\end{equation}
 Moving to the next step of the sequence, we mutate at $\left(\ell,\beta+1\right)$.
The corresponding cluster variable is $f_{\ell,\beta+1}^{\prime}=f_{\ell-1,\beta+1}^{\leftarrow}$
(since it was mutated in the sequence 
$\left(f_{\beta n},f_{\beta-1,n-1},\ldots,\right)$). 
The exchange rule here reads 
\begin{eqnarray*}
f_{\ell,\beta+1}^{\prime}\cdot f_{\ell,\beta+1}^{\prime\prime} & = & f_{\ell,\beta+1}^{\prime}f_{\ell,\beta+2}+f_{\ell+1,\beta+1}^{\prime}f_{\ell-1,\beta+1}\\
 & = & f_{\ell-1,\beta+1}^{\leftarrow}f_{\ell,\beta+2}+f_{\ell\beta}^{\downarrow}f_{\ell-1,\beta+1}
\end{eqnarray*}
 and \eqref{eq:DesJacId} can be used again, with $A=X_{\left[\ell-1,\beta-1,\beta+1\right]}^{\left[\beta,n\right]}$.
The result is 
\[
f_{\ell,\beta+1}^{\prime\prime}=f_{\ell,\beta+1}^{\leftarrow},
\]
 and again, the same result can be obtained in the case $\beta=n-1.$
So just like in the proof of Lemma \ref{lem:NewFuncLCandCf} we still
get $f_{1,\beta+1}^{\leftarrow}$ as a cluster variable, and so it
is log canonical with all the functions of the form $f_{k,n+k-\beta}^{\downarrow}$. 

The Poisson coefficients can now be computed just like in Lemma \ref{lem:NewFuncLCandCf}
so 
\begin{align*}
\omega_{f_{k,n+k-\beta}^{\downarrow},f_{1,\beta+1}^{\leftarrow}} & =\omega_{f_{k,n+k-\beta},f_{1,\beta+1}}+\omega_{f_{\beta+1,n},f_{1,\beta+1}}-\omega_{f_{\beta n},f_{1,\beta+1}}\\
 & +\omega_{f_{k,n+k-\beta},f_{n\beta}}+\omega_{f_{\beta+1,n},f_{n\beta}}-\omega_{f_{\beta n},f_{n\beta}}\\
 & -\omega_{f_{k,n+k-\beta},f_{n,\beta+1}}-\omega_{f_{\beta+1,n},f_{n,\beta+1}}+\omega_{f_{\beta n},f_{n,\beta+1}}.
\end{align*}
This can be done in the same way with the sequence $\left(f_{\alpha+1,n},\ldots,f_{\alpha+1,2}\right)$
to show that $f_{\alpha+1,1}^{\uparrow}$ is also log canonical with
all $f_{k,n+k-\beta}^{\downarrow}.$ The Poisson coefficient will
be 
\begin{align*}
\omega_{f_{k,n+k-\beta}^{\downarrow},f_{\alpha+1,1}^{\uparrow}} & =\omega_{f_{k,n+k-\beta},f_{\alpha+1,1}}+\omega_{f_{\beta+1,n},f_{\alpha+1,1}}-\omega_{f_{\beta n},f_{\alpha+1,1}}\\
 & +\omega_{f_{k,n+k-\beta},f_{\alpha n}}+\omega_{f_{\beta+1,n},f_{\alpha n}}-\omega_{f_{\beta n},f_{\alpha n}}\\
 & -\omega_{f_{k,n+k-\beta},f_{\alpha+1,n}}-\omega_{f_{\beta+1,n},f_{\alpha+1,n}}+\omega_{f_{\beta n},f_{\alpha+1,n}}.
\end{align*}

\end{proof}
The following Lemma computes the brackets of a function $f\in\mathcal{B}_{\alpha\beta}\cap\mathcal{B}_{std}$
with certain families of functions in $\mathcal{B}_{std}$. 
\begin{lem}
\label{lem:SpeFuncBrckt}1. Let $g=f_{k,\beta+1}$ with $k\in\left[2,n\right]$.
Then 
\begin{equation}
\left\{ f_{1,\beta+1}^{\leftarrow},g\right\} _{std}=\left(\omega_{f_{1,\beta+1},g}+\omega_{x_{n\beta},g}-\omega_{x_{n,\beta+1},g}\right)f_{1,\beta+1}^{\leftarrow}g+f_{1,\beta+1}g^{\leftarrow}\label{eq:brcktf1btandfkb}
\end{equation}
2. For $k\in\left[\alpha-1\right]$, let $g=f_{n+m-\alpha,m}$ with
$m\in\left[k+1,\alpha\right]$ . Then 
\begin{equation}
\left\{ g,f_{n+k-\alpha,k}^{\rightarrow}\right\} _{std}=\left(\omega_{g,f_{n+k-\alpha,k}}-\omega_{g,x_{n\alpha}}+\omega_{g,x_{n,\alpha+1}}\right)gf_{n+k-\alpha,k}^{\rightarrow}+f_{n+k-\alpha,k}g^{\rightarrow}\label{eq:brcktftkandfm}
\end{equation}
\end{lem}
\begin{proof}
1. Let $g=f_{k,\beta+1}$. we compute the bracket $\left\{ f_{1,\beta+1}^{\leftarrow},g\right\} _{std}$
directly using \eqref{eq:BrcktthTrc}. Recall that 
\[
\left(\nabla f_{1,\beta+1}^{\leftarrow}\cdot X\right)_{ij}=\sum_{m=1}^{n}\frac{\partial f_{1,\beta+1}^{\leftarrow}}{\partial x_{mi}}x_{mj}=\left(f_{1,\beta+1}^{\leftarrow}\right)^{i\leftarrow j}
\]
 and since $f_{1,\beta+1}^{\leftarrow}=\det X_{\left[1,n-\beta\right]}^{\left[\beta,\beta+2,\ldots,n\right]}$
we have 
\[
\left(f_{1,\beta+1}^{\leftarrow}\right)^{i\leftarrow j}=0
\]
 for $i<\beta$ and $i=\beta+1$. Similarly, the term 
\[
\left(\nabla g\cdot X\right)_{ij}=\sum_{m=1}^{n}\frac{\partial g}{\partial x_{mi}}x_{mj}=g^{i\leftarrow j}
\]
 vanishes for $i<\beta+1$. On the other hand, looking at the second
trace form in \eqref{eq:BrcktthTrc}, 
\[
\left(X\cdot\nabla f_{1,\beta+1}^{\leftarrow}\right)_{ij}=\sum_{m=1}^{n}\frac{\partial f_{1,\beta+1}^{\leftarrow}}{\partial x_{jm}}x_{im}=\left(f_{1,\beta+1}^{\leftarrow}\right)_{j\leftarrow i},
\]
 which vanishes for $j>n-\beta$, and also 
\[
\left(X\cdot\nabla g\right)_{ij}=\sum_{m=1}^{n}\frac{\partial g}{\partial x_{jm}}x_{im}=g_{j\leftarrow i}
\]
 is non zero only for $k\leq j\leq n+k-\beta-1$. Applying $R_{+}$
to the matrices $\nabla\tilde{f}_{k}\cdot X$ and $X\cdot\nabla\tilde{f}_{k}$
vanishes all entries below the main diagonal. On the main diagonal
we have only the original function with some coefficients $\xi_{i}$.
So we can write \eqref{eq:BrcktthTrc} as: 
\begin{eqnarray*}
\left\{ f_{1,\beta+1}^{\leftarrow},g\right\}  & = & \left\langle R_{+}\left(\nabla f_{1,\beta+1}^{\leftarrow}\cdot X\right),\nabla g\cdot X\right\rangle -\left\langle R_{+}\left(X\cdot\nabla f_{1,\beta+1}^{\leftarrow}\right),X\cdot\nabla g\right\rangle `\\
 & = & \sum_{i<j}\left(f_{1,\beta+1}^{\leftarrow}\right)^{i\leftarrow j}g^{j\leftarrow i}+\sum_{i}\xi_{i}f_{1,\beta+1}^{\leftarrow}g^{i\leftarrow i}\\
 &  & -\sum_{i<j}\left(f_{1,\beta+1}^{\leftarrow}\right)_{j\leftarrow i}g_{i\leftarrow j}-\sum_{i}\xi'_{i}f_{1,\beta+1}^{\leftarrow}g_{i\leftarrow i},
\end{eqnarray*}
 Look at the term $\left(f_{1,\beta+1}^{\leftarrow}\right)^{i\leftarrow j}$:
whenever $\left(i,j\right)\neq\left(\beta,\beta+1\right)$ it vanishes,
because $f_{1,\beta+1}^{\leftarrow}=\det X_{\left[1,n-\beta\right]}^{\left[\beta,\beta+2,\ldots,n\right]}$
and so $\left(f_{1,\beta+1}^{\leftarrow}\right)^{i\leftarrow j}$
is the determinant of a submatrix with two identical columns ($j>i$).
The only non zero term here is then $\left(f_{1,\beta+1}^{\leftarrow}\right)^{\beta\leftarrow\beta+1}=f_{1,\beta+1}.$
Similarly, $\left(f_{1,\beta+1}\right)_{j\leftarrow i}$ must vanish
when $i<j$, because it is the determinant of a submatrix with two
identical rows. Therefore, the only non zero terms of the trace form
are$f_{1,\beta+1}g^{\leftarrow}$ and the diagonal ones. The latter
are just the product of the two functions multiplied by the coefficients
$\xi_{i}$ and $\xi'_{i}$. Note that $f_{1,\beta+1}^{i\leftarrow i}$
vanishes when $i<\beta+1$, and $\tilde{f}_{k}^{i\leftarrow i}$ vanishes
for $i<\beta$ and for $i=\beta+1$. Comparing these coefficients
with the coefficients of the bracket $\left\{ f_{1,\beta+1},g\right\} ,$
we see that the only difference is the contribution of the elements
in entries $\left(\beta,\beta\right)$ and $\left(\beta+1,\beta+1\right)$:
\begin{eqnarray*}
\left(\nabla f_{1,\beta+1}\cdot X\right)_{\beta,\beta} & = & 0\\
\left(\nabla f_{1,\beta+1}\cdot X\right)_{\beta+1,\beta+1} & = & f_{1,\beta+1}\\
\left(\nabla f_{1,\beta+1}^{\leftarrow}\cdot X\right)_{\beta,\beta} & = & f_{1,\beta+1}^{\leftarrow}\\
\left(\nabla f_{1,\beta+1}^{\leftarrow}\cdot X\right)_{\beta+1,\beta+1} & = & 0
\end{eqnarray*}
And this is just the same for $x_{n,\beta}$ and $x_{n,\beta+1}$:
\begin{eqnarray*}
\left(\nabla x_{n,\beta+1}\cdot X\right)_{\beta,\beta} & = & 0\\
\left(\nabla x_{n,\beta+1}\cdot X\right)_{\beta+1,\beta+1} & = & x_{n,\beta+1}\\
\left(\nabla x_{n,\beta}\cdot X\right)_{\beta,\beta} & = & x_{1,\beta}\\
\left(\nabla x_{n,\beta}\cdot X\right)_{\beta+1,\beta+1} & = & 0.
\end{eqnarray*}
 Hence, we can conclude 
\begin{equation}
\left\{ f_{1,\beta+1}^{\leftarrow},g\right\} _{std}=\left(\omega_{f_{1,\beta+1},g}+\omega_{x_{n,\beta},g}-\omega_{x_{n,\beta+1},g}\right)f_{1,\beta+1}^{\leftarrow}g+f_{1,\beta+1}g^{\leftarrow}.\label{eq:brcktf1btandfkb-2}
\end{equation}
2. The proof here follows a similar path:
from \eqref{eq:BrcktthTrc} we have 
\begin{eqnarray*}
\left\{ g,f_{n+k-\alpha,k}^{\rightarrow}\right\} _{std} & = & \left\langle R_{+}\left(\nabla g\cdot X\right),\nabla f_{n+k-\alpha,k}^{\rightarrow}\cdot X\right\rangle \\
 &  & -\left\langle R_{+}\left(X\cdot\nabla g\right),X\cdot\nabla f_{n+k-\alpha,k}^{\rightarrow}\right\rangle 
\end{eqnarray*}
 and since $R_{+}$ annihilates all the entries below the main diagonal,
\begin{eqnarray*}
\left\{ g,f_{n+k-\alpha,k}^{\rightarrow}\right\} _{std} & = & \sum_{i=m}^{\alpha}\sum_{j=i+1}^{\alpha-1}g^{i\leftarrow j}\left(f_{n+k-\alpha,k}^{\rightarrow}\right)^{j\leftarrow i}\\
 &  & +\sum_{i=m}^{\alpha}g^{i\leftarrow\alpha+1}\left(f_{n+k-\alpha,k}^{\rightarrow}\right)^{\alpha+1\leftarrow i}+\sum_{j=1}^{n}\xi_{j}gf_{n+k-\alpha,k}^{\rightarrow}\\
 &  & -\sum_{j=n+m-\alpha}^{n}\sum_{i<j}g_{j\leftarrow i}\left(f_{n+k-\alpha,k}^{\rightarrow}\right)_{i\leftarrow j}\\
 &  & -\sum_{j=1}^{n}\xi'_{j}gf_{n+k-\alpha,k}^{\rightarrow}
\end{eqnarray*}
 where $\xi_{j}$ and $\xi'_{j}$ are some coefficients. But $f_{n+k-\alpha,k}^{\rightarrow}=\det X_{\left[n-k,n\right]}^{\left[n-k,\ldots,\alpha-1,\alpha+1\right]}$,
and therefore for every $i\in\left[m,\alpha-1\right]$ and $j\in[i+1,\alpha-1]$
we get 
\[
\left(f_{n+k-\alpha,k}^{\rightarrow}\right)^{j\leftarrow i}=0,
\]
 because it is the determinant of a matrix with two identical columns.
For the same reason, $\left(f_{n+k-\alpha,k}^{\rightarrow}\right)^{\alpha+1\leftarrow i}$
vanishes for every $i\neq\alpha.$ Likewise, the term $\left(f_{n+k-\alpha,k}^{\rightarrow}\right)_{i\leftarrow j}$
is zero for every $j\in[n+m-\alpha,n]$ and $i<j$, because this is
also a determinant of a matrix with two identical columns. So we are
left with 
\begin{eqnarray*}
\left\{ g,f_{n+k-\alpha,k}^{\rightarrow}\right\} _{std} & =\xi gf_{n+k-\alpha,k}^{\rightarrow} & +g^{\alpha\leftarrow\alpha+1}\left(f_{n+k-\alpha,k}^{\rightarrow}\right)^{\alpha+1\leftarrow\alpha}\\
 & =\xi gf_{n+k-\alpha,k}^{\rightarrow}+ & g^{\rightarrow}f_{n+k-\alpha,k},
\end{eqnarray*}
 for some coefficient $\xi$. Now, compare the coefficients $\xi_{j}$
and $\xi'_{j}$ in the bracket $\left\{ g,f_{n+k-\alpha}^{\rightarrow}\right\} $
to those of $\left\{ g,f_{n+k-\alpha,k}\right\} $. The difference
is equal to the difference between these coefficients in $\left\{ g,x_{n,\alpha+1}\right\} $
and $\left\{ g,x_{n\alpha}\right\} .$ to see that, note that these
functions are determinants of submatrices of $X$ that are distinguished
only by the last column, which is $\alpha+1$ in the first case and
$\alpha$ in the second. The result, like in \eqref{eq:brcktf1btandfkb}
is 
\begin{equation}
\left\{ g,f_{n+k-\alpha,k}^{\rightarrow}\right\} _{std}=\left(\omega_{1}-\omega_{2}+\omega_{3}\right)f_{n+k-\alpha,k}^{\rightarrow}g+f_{n+k-\alpha,k}g^{\rightarrow},
\end{equation}
 with 
\begin{eqnarray*}
\omega_{1} & = & \omega_{g,f_{n+k-\alpha,k}}\\
\omega_{2} & = & \omega_{g,x_{n,\alpha}}\\
\omega_{3} & = & \omega_{g,x_{n,\alpha+1}.}
\end{eqnarray*}
 
\end{proof}
The Lemmas \ref{lem:NewFuncLCandCf}, \ref{lem:ArFuncLC} and \ref{lem:SpeFuncBrckt}
can be rephrased in a symmetric way: transpose rows and columns of
the matrix, so $x_{ij}\longleftrightarrow x_{ji}$ (and therefore
$f_{ij}\longleftrightarrow f_{ji}$) and switch $\alpha$ and $\beta$.
The proofs are identical.
\begin{lem}
\label{lem:NewFuncLCandCfSym} 1. The function $f_{k,n+k-\beta}^{\downarrow}$
(with $k\in[\beta]$) is log canonical with all functions $g\in\mathcal{B}_{\alpha\beta}\cap\mathcal{B}_{std}$,
provided $g\neq f_{m,n+m-\beta}$ for some $m<k$, w.r.t. the standard
bracket $\left\{ \cdot,\cdot\right\} _{std}$. In this case the Poisson
coefficient is 
\begin{equation}
\omega_{f_{k,n+k-\beta}^{\downarrow},g}=\omega_{f_{k,n+k-\beta},g}+\omega_{x_{\beta+1,n},g}-\omega_{x_{\beta n},g}.
\end{equation}
2. The function $f_{\alpha+1,1}^{\uparrow}$ is log canonical with
all functions $g\in\mathcal{B}_{\alpha\beta}\cap\mathcal{B}_{std}$,
provided $g\neq f_{\alpha+1,m}$ for some $m\in[2,n]$ w.r.t. the
standard bracket $\left\{ \cdot,\cdot\right\} _{std}$. In this case
the Poisson coefficient is 
\begin{equation}
\omega_{f_{\alpha+1,1}^{\uparrow},g}=\omega_{f_{\alpha+1,1},g}+\omega_{x_{\alpha n},g}-\omega_{x_{n,\alpha+1},g}.
\end{equation}
\end{lem}
\begin{proof}
See proof of Lemma \ref{lem:NewFuncLCandCf}.\end{proof}
\begin{lem}
For $k\in\left[\alpha\right]$, the function $f_{n+k-\alpha,k}^{\rightarrow}$
is log canonical with $f_{1,\beta+1}^{\leftarrow}$ and $f_{\alpha+1,1}^{\uparrow}$.\end{lem}
\begin{proof}
See Proof of Lemma \ref{lem:ArFuncLC}. The path $\left(\left(n,\alpha\right),\left(n-1,\alpha-1\right),\ldots\right)$
can not cross the path $\left(\left(n,\beta+1\right),\left(n-1,\beta+1\right),\ldots\right)$
since we assume $\alpha<\beta$. Proving that $f_{n+k-\alpha,k}^{\rightarrow}$
and $f_{\alpha+1,1}^{\uparrow}$ are log canonical is symmetric proving
Lemma \ref{lem:ArFuncLC}.\end{proof}
\begin{lem}
\label{lem:SpeFuncBrcktSym}1. Let $g=f_{\alpha+1,k}$ with $k\in\left[2,n\right]$.
Then 
\begin{equation}
\left\{ f_{\alpha+1,1}^{\uparrow},g\right\} _{std}=\left(\omega_{f_{\alpha+1,1},g}+\omega_{x_{\alpha n},g}-\omega_{x_{\alpha+1,n},g}\right)f_{\alpha+1,1}^{\uparrow}g+f_{1,\beta+1}g^{\uparrow}\label{eq:brcktfa1tandfka-1}
\end{equation}
2. For $k\in\left[\beta-1\right]$, let $g=f_{m,n+m-\beta}$ with
$m\in\left[k+1,\beta\right]$ . Then
\begin{equation}
\left\{ g,f_{k,n+k-\beta}^{\downarrow}\right\} _{std}=\left(\omega_{g,f_{k,n+k-\beta}}-\omega_{g,x_{\beta n}}+\omega_{g,x_{\beta+1,n}}\right)gf_{k,n+k-\beta}^{\downarrow}-f_{k,n+k-\beta}g^{\downarrow}\label{eq:brcktfakandfb-k}
\end{equation}
\end{lem}
\begin{proof}
Same as the proof of Lemma \ref{lem:SpeFuncBrckt}.\end{proof}
\begin{lem}
\label{lem:SumCfeq0} 1. Let $g\in\mathcal{B}_{std}$ be a function
of the initial standard cluster, and let 
\begin{equation}
s\omega_{\alpha\beta}\left(g\right)=\omega_{f_{n\alpha},g}-\omega_{f_{n,\alpha+1},g}-\omega_{f_{n\beta},g}+\omega_{f_{n,\beta+1},g}.\label{eq:swabDef}
\end{equation}
 Then 
\begin{equation}
s\omega_{\alpha\beta}\left(g\right)=\begin{cases}
1 & \text{ if }g=f_{n+k-\alpha,k}\\
-1 & \text{ if }g=f_{i,\beta+1}\\
0 & \text{ otherwise.}
\end{cases}
\end{equation}
2. Let $g\in\mathcal{B}_{std}$ be a function of the initial standard
cluster and let 
\[
s'\omega_{\alpha\beta}\left(g\right)=\omega_{f_{\alpha n},g}-\omega_{f_{\alpha+1,n},g}-\omega_{f_{\beta n},g}+\omega_{f_{\beta+1,n},g}.
\]
Then
\begin{eqnarray}
s'\omega_{\alpha\beta}\left(g\right) & = & \begin{cases}
1 & \text{ if }g=f_{k,n+k-\beta}\\
-1 & \text{ if }g=f_{\alpha+1,j}\\
0 & \text{ otherwise.}
\end{cases}
\end{eqnarray}
 \end{lem}
\begin{proof}
1. We will compute the coefficients through \eqref{eq:BrcktthTrc}.
Since $\nabla x_{nk}=e_{kn}$ we have 
\[
\left(\nabla x_{nk}X\right)_{ij}=\begin{cases}
x_{nj} & i=k\\
0 & i\neq k,
\end{cases}\text{ and }\left(X\nabla x_{nk}\right)_{ij}=\begin{cases}
x_{ik} & j=n\\
0 & j\neq n.
\end{cases}
\]
 According to Lemma \ref{lem:TheOperatorRpls}, 
\[
R_{+}\left(\nabla x_{nk}X\right)_{ij}=\begin{cases}
x_{nj} & i=k,\ j>k\\
\xi_{j}x_{nk} & i=j
\end{cases}
\]
 with some coefficients $\xi_{j}$, and 
\[
R_{+}\left(X\nabla x_{nk}\right)_{ij}=\begin{cases}
x_{ik} & j=n\\
\xi'_{j}x_{nk} & i=j
\end{cases}
\]
 with other coefficients $\xi'_{j}$. Plugging this into \eqref{eq:BrcktthTrc}
gives 
\begin{eqnarray*}
\left\{ x_{nk},g\right\} _{std} & = & \sum_{j=k+1}^{n}x_{nj}\sum_{i=1}^{n}\frac{\partial g}{\partial x_{ij}}x_{ik}+\sum_{j=1}^{n}\xi_{j}x_{nk}g\\
 &  & -\sum_{i=1}^{n-1}x_{ik}\sum_{j=1}^{n}x_{nj}\frac{\partial g}{\partial x_{ij}}-\sum_{j=1}^{n}\xi'_{j}x_{nk}g.
\end{eqnarray*}
 Split the last term into the ``diagonal'' part 
\[
D=\sum_{j=1}^{n}\xi_{j}x_{nk}g-\sum_{j=1}^{n}\xi'_{j}x_{nk}g
\]
 and the ``non diagonal'' part 
\[
N=\sum_{j=k+1}^{n}x_{nj}\sum_{i=1}^{n}\frac{\partial g}{\partial x_{ij}}x_{ik}-\sum_{i=1}^{n-1}x_{ik}\sum_{j=1}^{n}x_{nj}\frac{\partial g}{\partial x_{ij}}.
\]
Start with the diagonal part $D$. We need the coefficients $\xi_{j}$
and $\xi'_{j}$ for $k=\alpha,\beta,\alpha+1,\beta+1$. Recall \eqref{eq:Rplsonekk}:
\begin{eqnarray*}
R_{+}\left(e_{\alpha\alpha}\right) & = & \frac{1}{n}\sum_{j=1}^{n-1}s_{\alpha}\left(j\right)\left(\hat{h}_{j}-\hat{h}_{j-1}\right)+\hat{h}_{\alpha}\\
 &  & +\hat{h}_{\beta}-\left(n-\alpha\right)\hat{h}_{\beta-1}+\left(n-\beta\right)\hat{h}_{\alpha+1}\\
R_{+}\left(e_{\beta\beta}\right) & = & \frac{1}{n}\sum_{j=1}^{n-1}s_{\beta}\left(j\right)\left(\hat{h}_{j}-\hat{h}_{j-1}\right)\\
 &  & +\left(1-n\right)\hat{h}_{\alpha}+\alpha\hat{h}_{\beta-1}+\left(n-\beta\right)\hat{h}_{\alpha+1}\\
 &  & +\hat{h}_{\beta}\begin{cases}
1 & \beta>\alpha+1\\
\left(1-n\right) & \beta=\alpha+1
\end{cases}\\
R_{+}\left(e_{\alpha+1,\alpha+1}\right) & = & \frac{1}{n}\sum_{j=1}^{n-1}s_{\alpha+1}\left(j\right)\left(\hat{h}_{j}-\hat{h}_{j-1}\right)\\
 &  & +\left(1-n\right)\hat{h}_{\beta}+\alpha\hat{h}_{\beta-1}+\left(n-\beta\right)\hat{h}_{\alpha+1}\\
 &  & +\hat{h}_{\alpha}\begin{cases}
1 & \beta>\alpha+1\\
\left(1-n\right) & \beta=\alpha+1
\end{cases}\\
R_{+}\left(e_{\beta+1,\beta+1}\right) & = & \frac{1}{n}\sum_{j=1}^{n-1}s_{\beta+1}\left(j\right)\left(\hat{h}_{j}-\hat{h}_{j-1}\right)+\hat{h}_{\alpha}+\hat{h}_{\beta}+\alpha\hat{h}_{\beta-1}-\beta\hat{h}_{\alpha+1}\\
R_{+}\left(e_{nn}\right) & = & \frac{1}{n}\sum_{j=1}^{n-1}-j\left(\hat{h}_{j}-\hat{h}_{j-1}\right)+\hat{h}_{\alpha}+\hat{h}_{\beta}+\alpha\hat{h}_{\beta-1}-\beta\hat{h}_{\alpha+1}.
\end{eqnarray*}
 Using $\left(\hat{h}_{j}-\hat{h}_{j-1}\right)=\frac{1}{n}\diag\left(-1,\ldots,-1\right)+e_{jj}$
and the fact 
\begin{eqnarray*}
s_{\alpha}\left(j\right)-s_{\alpha+1}\left(j\right) & = & \begin{cases}
n & j=\alpha\\
0 & j\neq\alpha
\end{cases}
\end{eqnarray*}
 we get 
\begin{eqnarray*}
 &  & \sum_{j=1}^{n-1}s_{\alpha}\left(j\right)\left(\hat{h}_{j}-\hat{h}_{j-1}\right)-\sum_{j=1}^{n-1}s_{\alpha+1}\left(j\right)\left(\hat{h}_{j}-\hat{h}_{j-1}\right)\\
 &  & =\diag\left(-1,\ldots,-1\right)+ne_{\alpha\alpha},
\end{eqnarray*}
 and 
\[
\sum_{j=1}^{n-1}s_{\beta+1}\left(j\right)\left(\hat{h}_{j}-\hat{h}_{j-1}\right)-\sum_{j=1}^{n-1}s_{\beta}\left(j\right)\left(\hat{h}_{j}-\hat{h}_{j-1}\right)=\diag\left(1,\ldots,1\right)-ne_{\beta\beta}.
\]
 Putting everything together gives 
\begin{eqnarray}
 &  & R_{+}\left(e_{\alpha\alpha}-e_{\beta\beta}-e_{\alpha+1,\alpha+1}+e_{\beta+1,\beta+1}\right)\label{eq:Rplsonaba+1b+1}\\
 &  & =n\left(\hat{h}_{\alpha}+\hat{h}_{\beta}-\hat{h}_{\alpha+1}-\hat{h}_{\beta-1}\right)+ne_{\alpha\alpha}-ne_{\beta\beta},
\end{eqnarray}
 for $\beta>\alpha+1$, or in the case $\beta=\alpha+1$: 
\begin{eqnarray*}
R_{+}\left(e_{\alpha\alpha}-e_{\beta\beta}-e_{\alpha+1,\alpha+1}+e_{\beta+1,\beta+1}\right) & = & e_{\alpha\alpha}-e_{\alpha+1,\alpha+1},
\end{eqnarray*}
 and since $\hat{h}_{\alpha}-\hat{h}_{\alpha+1}=\frac{1}{n}\diag\left(1,\ldots,1\right)-e_{\alpha+1,\alpha+1}$,
and $\hat{h}_{\beta}-\hat{h}_{\beta-1}=\frac{1}{n}\diag\left(-1,\ldots,-1\right)+e_{\beta\beta}$,
\eqref{eq:Rplsonaba+1b+1} turns to 
\[
R_{+}\left(e_{\alpha\alpha}-e_{\beta\beta}-e_{\alpha+1,\alpha+1}+e_{\beta+1,\beta+1}\right)=\left(e_{\alpha\alpha}-e_{\alpha+1,\alpha+1}\right).
\]
 Since $D$ is a trace of two matrices, we are only interested in
products of the diagonal elements in $R_{+}\left(\nabla x_{nk}\cdot X\right)$,
$R_{+}\left(X\cdot\nabla x_{nk}\right)$ with the corresponding diagonal
elements in $\nabla g\cdot X$ and $X\cdot\nabla g$. These products
vanish for all $g\in\mathcal{B}_{std}$ except $g=f_{i,\alpha+1}$
(which is a the determinant of a submatrix that has column $\alpha+1$
but not col. $\alpha$) or $g=f_{n-\alpha+k,k}$ (a determinant of
a submatrix that has column $\alpha$ but not column $\alpha+1$).
Write $\omega_{f_{nk},g}^{D}=\frac{D}{f_{nk}g}$, So the sum of coefficients
of the diagonal part is 
\[
\omega_{f_{n\alpha},g}^{D}-\omega_{f_{n\beta},g}^{D}-\omega_{f_{n,\alpha+1},g}^{D}+\omega_{f_{n,\beta+1},g}^{D}=\begin{cases}
1 & g=f_{n-\alpha+k,k}\\
-1 & g=f_{i,\alpha+1}\\
0 & \text{otherwise.}
\end{cases}
\]
We now turn to the non diagonal part $N$: recall 
\begin{eqnarray*}
\left(\nabla x_{nk}X\right)_{ij} & = & \left(x_{nk}\right)^{i\leftarrow j}=\begin{cases}
x_{nj} & i=k\\
0 & i\neq k,
\end{cases}\\
\left(X\nabla x_{nk}\right)_{ij} & = & \left(x_{nk}\right)_{j\leftarrow i}=\begin{cases}
x_{ik} & j=n\\
0 & j\neq n,
\end{cases}
\end{eqnarray*}
 and we have 
\[
R_{+}\left(\nabla x_{nk}X\right)=\left[\begin{array}{ccccc}
0\\
 & \ddots & x_{n,k+1} & \cdots & x_{n,n}\\
 &  & \ddots\\
 &  &  & \ddots\\
 &  &  &  & \ddots
\end{array}\right],
\]
 and 
\[
R_{+}\left(X\nabla x_{nk}\right)=\left[\begin{array}{ccccc}
0 &  &  &  & x_{1k}\\
 & \ddots &  &  & \vdots\\
 &  & \ddots &  & \vdots\\
 &  &  & \ddots & x_{n-1,k}\\
 &  &  &  & \ddots
\end{array}\right],
\]
 so when computing the bracket with \eqref{eq:BrcktthTrc}, 
\begin{eqnarray*}
N & = & \sum_{j=k+1}^{n}x_{nj}\sum_{i=1}^{n}\frac{\partial g}{\partial x_{ij}}x_{ik}-\sum_{i=1}^{n-1}x_{ik}\sum_{j=1}^{n}x_{nj}\frac{\partial g}{\partial x_{ij}}\\
 & = & x_{nk}\sum_{j=k+1}^{n}x_{nj}\frac{\partial g}{\partial x_{nj}}-\sum_{i=1}^{n-1}x_{ik}\sum_{j=1}^{k}x_{nj}\frac{\partial g}{\partial x_{ij}}\\
 & = & x_{nk}\sum_{j=k+1}^{n}x_{nj}\frac{\partial g}{\partial x_{ij}}-\sum_{i=1}^{n}x_{ik}\sum_{j=1}^{k}x_{nj}\frac{\partial g}{\partial x_{ij}}+x_{nk}\sum_{j=1}^{k}x_{nj}\frac{\partial g}{\partial x_{nj}}\\
 & = & x_{nk}\sum_{j=1}^{n}x_{nj}\frac{\partial g}{\partial x_{nj}}-\sum_{i=1}^{n}x_{ik}\sum_{j=1}^{k}x_{nj}\frac{\partial g}{\partial x_{ij}}\\
 & = & x_{nk}g-\sum_{i=1}^{n}x_{ik}\sum_{j=1}^{k}x_{nj}\frac{\partial g}{\partial x_{ij}}.
\end{eqnarray*}
 Now, since $g$ is a determinant of some submatrix $A$ of $X$,
let $g^{\max}$and $g^{\min}$ denote the maximal (right) and minimal
(left) columns of $A$. Similarly, let $g_{\max}$ be the last row
of $A$. Then 
\[
\sum_{i=1}^{n}x_{ik}\sum_{j=1}^{k}x_{nj}\frac{\partial g}{\partial x_{ij}}=\begin{cases}
0 & g^{\min}>k\\
x_{nk}g & g^{\min}\leq k\leq g^{\max}\\
x_{nk}g & g^{\max}<k\longrightarrow g_{\max=n}
\end{cases}
\]
 so that 
\[
N=\begin{cases}
x_{nk}g & g^{\min}>k\\
0 & g^{\min}\leq k.
\end{cases}
\]
 Defining $\omega_{f_{nk},g}^{N}=\frac{N}{f_{nk}g}$, summing over
$k=\alpha,\alpha+1,\beta,\beta+1$ we get $\sum\omega_{f_{nk},g}^{N}\neq0$
only when $g=f_{i,\alpha+1}$ or $g=f_{i,\beta+1},$ or in the ``special''
case $\beta=\alpha+1$ : we can then write the sum of these coefficients:

\[
\sum\omega_{f_{nk},g}^{N}=\begin{cases}
0 & g^{\min}=\alpha\\
1 & g^{\min}=\alpha+1\\
0 & g^{\min}=\beta\text{ and }\beta>\alpha+1\\
-1 & g^{\min}=\beta+1.
\end{cases}
\]

We now add the diagonal part coefficients, for $s\omega_{\alpha\beta}(g)=\sum\omega_{f_{nk},g}^{N}+\sum\omega_{f_{nk},g}^{D}$,
so 
\begin{enumerate}
\item If $g=f_{i,\alpha+1}$, then the sum of non diagonal coefficients
is $1$. We have seen that in this case the sum of diagonal coefficients
is $-1$, and therefore $s\omega_{\alpha\beta}(g)=0$. 
\item If $g=g_{i,\beta}$ and $\beta=\alpha+1$, just like in 1., it is
$s\omega_{\alpha\beta}(g)=0$. 
\item If $g=f_{i,\beta+1}$ then $s\omega_{\alpha\beta}(g)=-1$. 
\item If $g=f_{n+k-\alpha,k}$ then $s\omega_{\alpha\beta}(g)=-1$. 
\item For any other $g\in\mathcal{B}_{std}$, $s\omega_{\alpha\beta}(g)=0$. 
\end{enumerate}
This completes the proof of part 1. of the Lemma. Part 2. is similar,
using the symmetries $x_{ij}\longleftrightarrow x_{ji}$ (and therefore
$f_{ij}\longleftrightarrow f_{ji}$), and $\alpha\longleftrightarrow\beta.$ 
\end{proof}

\section*{Acknowledgments}

The author was supported by ISF grant \#162/12. The author thanks
Michael Gekhtman for his helping comments and answers. Special thanks
to Alek Vainshtein for his support and encouragement, as well as his
mathematical, technical and editorial advices.

\pagebreak{}

\bibliographystyle{abbrv}
\bibliography{ECSBDdms}

\end{document}